\documentclass[12pt,a4paper]{article}
\usepackage{modnatbib}
\usepackage{amssymb,amsmath,amsthm,epsfig,verbatim,pstricks,url}
\usepackage{tikz,pgfplots,tikz-3dplot}
\usetikzlibrary{fit}

\usepackage{ifpdf}  
\newtheorem{thm}{\sc Theorem.}[section]
\newtheorem{lem}[thm]{\sc Lemma.}
\newtheorem{rem}[thm]{\sc Remark.}
\newtheorem{ass}[thm]{\sc Assumption.}
\renewcommand{\theequation}{\arabic{section}.\arabic{equation}}
\newenvironment{AMS}%
{{\upshape\bfseries AMS subject classifications. }\ignorespaces}{}
\newenvironment{keywords}{{\upshape\bfseries Key words. }\ignorespaces}{}

\newcommand{\bRplus}{{\mathbb R}_{>0}}
\newcommand{\bRgeq}{{\mathbb R}_{\geq 0}}

\newcommand{\bR}{{\mathbb R}}

\newcommand{\bZ}{\mathbb{Z}}
\newcommand{\spa}{\operatorname{span}}

\newcommand{\doctorkappa}{\mathfrak{K}}
\newcommand{\doctorZ}{\mathfrak{Z}}
\newcommand{\dH}[1]{\;{\rm d}{\mathcal{H}}^{#1}} 

\newcommand{\drho}{\;{\rm d}\rho}
\newcommand{\spont}{{\overline\varkappa}}

\newcommand{\Vhi}{\underline{V}^h_i}
\newcommand{\Whi}{W^h_i}
\newcommand{\Vyczcoi}{\mathbb{Y}_i}
\newcommand{\Vycztest}{\mathbb{Y}_{C^0}}
\newcommand{\Vycotest}{\mathbb{Y}_{C^1}}
\newcommand{\Vhycztest}{\mathbb{Y}_{C^0}^h}
\newcommand{\Vhycotest}{\mathbb{Y}_{C^1}^h}
\newcommand{\nabS}{\nabla_{\!\mathcal{S}}}
\newcommand{\nabSi}{\nabla_{\!\mathcal{S}_i}}
\newcommand{\Id}{\rm Id}
\newcommand{\id}{\rm id}
\newcommand{\deldel}[1]{\frac{\delta}{{\delta}#1}}
\newcommand{\dd}[1]{\frac{\rm d}{{\rm d}#1}}
\newcommand{\ddt}{\dd{t}}
\newcommand{\xspace}{\mathbb {X}}
\newcommand{\xspaceh}{\mathbb {X}^h}
\newcommand{\yspace}{\mathbb {Y}}

\newcommand{\unitn}{\vec{\rm n}}

\newcommand{\ek}{e}

\newcommand{\BGNpwf}{\mathcal{P}}

\newcommand{\Gauss}{{\mathcal{K}}}

\newcommand{\ttau}{\Delta t}
\def\epsilon{\varepsilon} 

\newcommand{\mat}[1]{\underline{\underline{#1}}\rule{0pt}{0pt}}

\newcommand{\errorXx}{\|\Gamma - \Gamma^h\|_{L^\infty}}

\begin{document}
\title{
Structure preserving discretisations of gradient flows for axisymmetric 
two-phase biomembranes
}
\author{
        Harald Garcke\footnotemark[2]\ \and 
        Robert N\"urnberg\footnotemark[3]}

\renewcommand{\thefootnote}{\fnsymbol{footnote}}
\footnotetext[2]{Fakult{\"a}t f{\"u}r Mathematik, Universit{\"a}t Regensburg, 
93040 Regensburg, Germany}
\footnotetext[3]{Department of Mathematics, 
Imperial College London, London, SW7 2AZ, UK}

\date{}

\maketitle

\begin{abstract}
The form and evolution of multi-phase biomembranes is of fundamental
importance in order to understand living systems. In order to describe these
membranes, we consider a mathematical model based on a
Canham--Helfrich--Evans two-phase elastic energy, which will lead to
fourth order geometric evolution problems involving highly nonlinear
boundary conditions. We develop a parametric finite element method in
an axisymmetric setting. Using a variational approach, it is possible
to derive weak formulations for the highly nonlinear boundary value
problems such that energy decay laws, as well as conservation
properties, hold for spatially discretised problems. We will prove
these properties and show that the fully discretised schemes are
well-posed. Finally, several numerical computations demonstrate that
the numerical method can be used to compute complex, experimentally
observed two-phase biomembranes. 
\end{abstract} 

\begin{keywords} 
biomembranes, multi-phase Canham--Helfrich--Evans energy, Willmore flow, 
parametric finite elements, stability, numerical simulations
\end{keywords}

\begin{AMS} 
65M60, 65M12, 35K55, 53C44
\end{AMS}
\renewcommand{\thefootnote}{\arabic{footnote}}

\begin{center}
\textbf{Dedicated to the memory of John W. Barrett}
\end{center}

\setcounter{equation}{0}
\section{Introduction}

Biomembranes and vesicles formed by lipid bilayers play a fundamental role
in many living systems, and synthesised artificial vesicles are used
in pharmaceutical applications as potential drug carriers. The basic structure
of such membranes is a bilayer consisting of phospholipids. As the thickness of
these membranes is small, the membrane is typically described as a 
hypersurface.
It is well-known that the energy of these membranes can be modelled with the
help of a curvature elasticity theory, 
see \cite{Canham70,Helfrich73,Evans74,Seifert97}.
Curvature terms in the energy account for bending stresses, but biomembranes 
have no or little lateral shear stresses, which hence are neglected in these
elasticity models.
 
Often micro-domains (or rafts) are formed due to the
clustering of certain molecules within the membrane. This leads to multi-phase
membranes with coexisting phases. It is observed that the membrane can have a
preferred curvature stemming, for example, from an asymmetry within the 
bilayer.
This so-called spontaneous curvature can depend on the phase. Moreover,
the bending rigidities appearing in the energy are typically also 
phase-dependent.
The simplest curvature energies involve the mean curvature, but neglect the
Gaussian curvature. For homogeneous biomembranes this is justified with the 
help of topological arguments, as long as the Gaussian bending
rigidity is constant, and as long as the topology of the membrane does not
change. However, for multi-phase membranes the Gaussian bending rigidity is
phase-dependent, and will thus influence membrane shapes. 
A combination of the above phase-dependent properties can lead to a multitude 
of different phenomena, including budding, fingering and fusion, see 
\cite{BaumgartHW03}.
In this paper, we consider a geometrical evolution law of gradient flow type 
for two-phase biomembranes that decreases the governing energy.
The energy we consider takes elastic energy as well as line energy into 
account. Where appropriate, the evolution
will conserve volume enclosed by the membrane, 
as well as the areas of the appearing phases.
We will derive a stable numerical method in an axisymmetric setting that is
structure preserving, in the sense that a semidiscrete variant decreases energy
and, when applicable, also conserves volume and areas exactly.
Axisymmetric formulations numerically have the advantage that they are
extremely efficient, and hence they allow for a more detailed resolution of
the shapes, in particular close to budding, for example.

Based on the fundamental work of \cite{JulicherL93,JulicherL96} we now
introduce a generalised Canham--Helfrich--Evans energy for a two-phase
biomembrane. The energy is defined for a two-phase surface 
$\mathcal{S} = (\mathcal{S}_1,\mathcal{S}_2)$,
consisting of two sufficiently smooth surfaces ${\mathcal{S}_i}$, $i=1,2$, 
in $\bR^3$, which have a common boundary $\gamma$ that is assumed to be a 
sufficiently smooth curve. 
In addition, it is assumed that $\mathcal{S}$ encloses a volume
$\Omega(\mathcal{S})$. 
The energy proposed by \cite{JulicherL93,JulicherL96} takes curvature effects,
as well as line energy effects, into account, and is given by
\begin{equation} \label{eq:E}
E(\mathcal{S}) = \sum_{i=1}^2 \left[ \tfrac12\,\alpha_i\,
\int_{\mathcal{S}_i} (k_{m,i} - \spont_i)^2 \dH{2} 
+ \alpha^G_i\,\int_{\mathcal{S}_i} k_{g,i} \dH{2} \right]
+ \varsigma\,\mathcal{H}^1(\gamma) \,.
\end{equation}
Here the constants $\alpha_i \in \bRplus$ 
and $\alpha^G_i \in \bR$ are the mean and the Gaussian bending rigidities
of the two phases,
and the constants $\spont_i \in \bR$ are the spontaneous curvatures. Note that
all these quantities might attain different values in the two phases. 
Moreover, $k_{m,i}$ and $k_{g,i}$ denote the mean and the Gaussian curvature of
$\mathcal{S}_i$, $i=1,2$, and $\varsigma$ is the energy density of the 
interface, often called line tension.
Finally, $\mathcal{H}^2$ and $\mathcal{H}^1$ are the surface and length
measures in $\bR^3$.

For the attachment conditions on $\gamma$
two cases have been considered in the literature, see
\cite{JulicherL96,Helmers11}:
\begin{subequations}
\begin{align}
\text{$C^0$--case :} & 
\quad \gamma = \partial\mathcal{S}_1 = \partial\mathcal{S}_2\,,\label{eq:C0}\\ 
\text{$C^1$--case :} & 
\quad \gamma = \partial\mathcal{S}_1 = \partial\mathcal{S}_2 
\quad \text{ and }\quad
\unitn_{\mathcal{S}_1} = \unitn_{\mathcal{S}_2} \quad\text{on } \gamma\,, 
\label{eq:C1}
\end{align}
\end{subequations}
where $\unitn_{\mathcal{S}_i}$ denotes the outer unit normal of
$\mathcal{S}_i$.
Of course, in the case (\ref{eq:C1}) it also holds that 
$\vec\mu_{\partial\mathcal{S}_1} =
-\vec\mu_{\partial\mathcal{S}_2}$, 
where $\vec\mu_{\partial\mathcal{S}_i}$ denotes the outer unit conormal to
$\mathcal{S}_i$ on $\gamma$.

It is discussed in \cite{pwfc0c1} that the contributions
\begin{equation*} 
\sum_{i=1}^2\left[\tfrac12\,\alpha_i\, 
\int_{\mathcal{S}_i} k_{m,i}^2 \dH{2} + 
\alpha^G_i\,\int_{\mathcal{S}_i} k_{g,i} \dH{2}\right]
\end{equation*}
to the energy (\ref{eq:E}) are nonnegative if 
\begin{equation*} 
\alpha^G_i \in [-2\,\alpha_i,0]\,,\ i=1,2\,. 
\end{equation*}
In the $C^1$--case, recall (\ref{eq:C1}), however, 
one can use the Gauss--Bonnet theorem, see (\ref{eq:GB}) below, to show 
that the energy (\ref{eq:E}), when restricted to a fixed topology,
can be bounded from below if
$\alpha^G_i \geq \max\{\alpha^G_1,\alpha^G_2\}-2\,\alpha_i$ for
$i=1,2$, which will hold whenever
\begin{equation} \label{eq:alphaGbound}
\min\{\alpha_1,\alpha_2\} \geq \tfrac12\,|\alpha^G_1 - \alpha^G_2|\,,
\end{equation}
see \cite{Nitsche93,pwfc0c1} for details.

It is crucial for a numerical treatment that the Gaussian curvature
term can be computed efficiently in the discrete setting.
In this context, a reformulation of the energy using the Gauss--Bonnet theorem 
is important.
In fact, the Gauss--Bonnet theorem yields
\begin{equation} \label{eq:GB}
\int_{\mathcal{S}_i} k_{g,i} \dH{2} = 2\,\pi\,m(\mathcal{S}_i) 
+ \int_{\partial\mathcal{S}_i} k_{\partial \mathcal{S}_i,\mu} \dH{1} \,,
\end{equation}
where $m(\mathcal{S}_i) \in \bZ$ denotes the Euler characteristic of 
$\mathcal{S}_i$ and $k_{\partial \mathcal{S}_i,\mu}$ 
is the geodesic curvature of $\partial\mathcal{S}_i$. 
Using this equality for the integrated Gaussian curvature, we can rewrite
the energy  (\ref{eq:E}) as
\begin{equation}\label{eq:EE}
E(\mathcal{S}) = \sum_{i=1}^2 \left[
\tfrac12\,\alpha_i\,\int_{\mathcal{S}_i} (k_{m,i} - \spont_i)^2 \dH{2}
 + \alpha^G_i \left[\int_{\gamma} 
k_{\partial\mathcal{S}_i,\mu} \dH{1} + 
2\,\pi\,m(\mathcal{S}_i) \right] \right]
+ \varsigma\, \mathcal{H}^1(\gamma) \,.
\end{equation}

We now need to compute the geodesic curvatures
$k_{\partial \mathcal{S}_i,\mu}$.
In order to do so, we first define 
the conormal, $\vec\mu_{\partial\mathcal{S}_i}$, 
to $\mathcal{S}_i$ on $\gamma$ to be 
\begin{equation} \label{eq:mupS}
\vec\mu_{\partial\mathcal{S}_i}= \pm\, \unitn_{\mathcal{S}_i} \times \vec\id_s
\quad \text{on } \gamma\,,\ i = 1,2\,,
\end{equation}
where $\vec\id$ denotes the identity in $\bR^3$
and $s$ denotes arclength on the curve $\gamma \subset \bR^3$, and the
sign in (\ref{eq:mupS}) is chosen so that $\vec\mu_{\partial\mathcal{S}_i}$
points out of $\mathcal{S}_i$, $i=1,2$.
It holds that
\begin{equation} \label{eq:idss}
\vec\id_{ss} = \vec k_{\gamma} =
k_{\partial \mathcal{S}_i,\rm n}\,\unitn_{\mathcal{S}_i} +
k_{\partial \mathcal{S}_i,\mu}\,\vec\mu_{\partial\mathcal{S}_i} 
\quad \text{on } \gamma\,,\ i = 1,2\,,
\end{equation}
where $\vec k_{\gamma}$ is the curvature vector on
$\gamma$, and where $k_{\partial \mathcal{S}_i,\rm n}$ 
is the normal curvature and $k_{\partial \mathcal{S}_i,\mu}$ 
is the geodesic curvature of $\partial\mathcal{S}_i$, $i=1,2$.

In applications for biomembranes, cf.\ \cite{JulicherL96,Tu13}, 
the surface areas of $\mathcal{S}_1$ and $\mathcal{S}_2$
need to stay constant during the evolution, as well as the volume of
the set $\Omega(\mathcal{S})$  enclosed by
$\mathcal{S}$. In this case one can consider the energy
\begin{equation}
E_\lambda (\mathcal{S}) =
E (\mathcal{S}) + \lambda_V\,\mathcal{L}^3(\Omega(\mathcal{S}))
+ \sum_{i=1}^2 \lambda_{A,i}\,\mathcal{H}^{2} (\mathcal{S}_i) \,,
\label{eq:areaE}
\end{equation}
where $\mathcal{L}^3$ denotes the Lebesgue measure in $\mathbb{R}^3$.
Here $\lambda_{A,i}$ are Lagrange multipliers for the
area constraints, which can be interpreted as a surface tension,
and $\lambda_V$ is a Lagrange multiplier for the volume constraint,
which might be interpreted as a pressure difference.

We now introduce the governing evolution equations that we consider in this
paper. We will consider the $L^2$--gradient flow of the energy $E_\lambda$,
leading to a time-dependent family of surfaces
$\mathcal{S}(t)$ and time-dependent Lagrange multipliers $\lambda_V(t)$ and
$\lambda_{A,i}(t)$, $i=1,2$. This
will lead to an equation for the normal velocity of the surfaces
$\mathcal{S}_i$, $i=1,2$, as well as to equations on the curve $\gamma$. 
The reformulation (\ref{eq:EE}) of the energy shows that a variation of
the energy, which only affects points away from $\gamma$, will not change
the Gaussian curvature part of the energy. This is reflected by the fact that,
in the gradient flow formulation, the normal velocities
$\mathcal{V}_{\mathcal{S}_i}$ on the surfaces $\mathcal{S}_i$, $i=1,2$,
do not contain terms stemming from the Gaussian curvature contribution
to the energy.  
In fact, we have from \citet[(2.16)]{pwfc0c1} that
\begin{align}
\mathcal{V}_{\mathcal{S}_i} & 
= - \alpha_i\,\Delta_{\mathcal{S}_i}\,k_{m,i} +
\tfrac12\,\alpha_i\,(k_{m,i} - \spont_i)^2 \,k_{m,i} -
\alpha_i\,(k_{m,i} - \spont_i)\,|\nabSi\,\unitn_{\mathcal{S}_i}|^2 +
\lambda_{A,i}\,k_{m,i} - \lambda_V \nonumber \\ &
= -\alpha_i\,\Delta_{\mathcal{S}_i} \,k_{m,i}
+ 2\, \alpha_i\,(k_{m,i}- \spont_i) \, k_{g,i}
-\left[\tfrac12\,\alpha_i\,(k_{m,i}^2 - \spont_i^2) -\lambda_{A,i} \right]
k_{m,i} - \lambda_V \nonumber \\ & \hspace{9cm}
\quad\text{on } \mathcal{S}_i(t)\,,\ i = 1,2\,,
\label{eq:gradflowlambda}
\end{align}
where $\Delta_{\mathcal{S}_i}$ and $\nabSi$ denote the surface Laplacian
and surface gradient on $\mathcal{S}_i$, respectively, and where 
we have observed that
$|\nabSi\,\unitn_{\mathcal{S}_i}|^2 = k_{m,i}^2 - 2\,k_{g,i}$,
see e.g.\ \citet[Lemma~12(iv)]{bgnreview}.

However, the Gaussian curvature energy contributions have an effect
on the boundary. In the $C^0$--junction case, for $t\in [0,T]$, 
the boundary conditions on $\gamma(t)$ are given by
\begin{subequations} \label{eq:C0bc}
\begin{align}
& \alpha_i\,(k_{m,i} - \spont_i) + \alpha^G_i\,
\vec k_{\gamma}\,.\,\unitn_{\mathcal{S}_i} = 0\,,\ i = 1,2\,,
\\
& \sum_{i=1}^2 \left[
 (\alpha_i\,(\nabSi\,k_{m,i})\,.\,\vec\mu_{\partial \mathcal{S}_i} -
 \alpha^G_i\,(\mathfrak{t}_i)_s)\, \unitn_{\mathcal{S}_i}
-(\tfrac12\,\alpha_i\,(k_{m,i} - \spont_i)^2
+ \alpha^G_i\,k_{g,i} + \lambda_{A,i})\,\vec\mu_{\partial\mathcal{S}_i} \right] 
\nonumber \\ & \qquad
+  \varsigma\, \vec k_\gamma  = \vec 0\,,  
\end{align}
\end{subequations}
see \citet[(2.19)]{pwfc0c1}, 
where $\mathfrak{t}_i =- (\unitn_{\mathcal{S}_i})_s\,.\,
\vec\mu_{\partial\mathcal{S}_i}$ is the geodesic torsion of $\gamma(t)$ on 
$\mathcal{S}_i(t)$.
In case of a $C^1$--junction, we have that
$\unitn_{\mathcal{S}}=\unitn_{\mathcal{S}_1}=\unitn_{\mathcal{S}_2}$ and
$\vec\mu_{\partial\mathcal{S}}=
\vec\mu_{\partial\mathcal{S}_2}=-\vec\mu_{\partial\mathcal{S}_1}$ at the
junction, and the governing equations on the curve $\gamma(t)$ for $t\in[0,T]$
are
\begin{subequations} \label{eq:C1bc}
\begin{align}
& [\alpha_i\,(k_{m,i} - \spont_i)]_1^2 + [\alpha^G_i]_1^2\,
\vec k_{\gamma}\,.\,\unitn_{\mathcal{S}}  = 0
\,, \\ 
& [\alpha_i\,(\nabSi\,k_{m,i} )]_1^2\,.\,\vec\mu_{\partial\mathcal{S}}
+\varsigma\,\vec k_{\gamma}\,.\,\unitn_{\mathcal{S}} 
 - [\alpha^G_i]_1^2\,\mathfrak{t}_s = 0 
\,,\\ 
& [-\tfrac12\,\alpha_i\,(k_{m,i}  - \spont_i)^2
+\alpha_i\,(k_{m,i} - \spont_i)\,
(k_{m,i} - \vec k_{\gamma}\,.\,\unitn_{\mathcal{S}} )
- \lambda_{A,i}]_1^2
+ [\alpha^G_i]_1^2\,\mathfrak{t}^2
 +\varsigma\,\vec k_{\gamma}\,.\,\vec\mu_{\mathcal{S}} = 0 
\,, 
\end{align}
\end{subequations}
see \citet[(2.20)]{pwfc0c1}, 
where $[a_i]_1^2 = a_2 - a_1$
denotes the jump of the quantity $a$ across $\gamma(t)$, and where
$\mathfrak{t} = \mathfrak{t}_2 = -\mathfrak{t}_1$.

For more basic information on the biophysics of vesicles and
biomembranes we refer to \cite{Seifert93}. Two-component membranes are
discussed in \cite{Lipowsky92,JulicherL93,Seifert93,JulicherL96,%
TuO-Y04,BaumgartDWJ05,Tu13,YangDT17,SahebifardSZ-R17a}.

Many mathematical results are known on the problem of minimising the
Willmore and Helfrich functional, see 
\cite{Nitsche93,MarquesN14,DeckelnickGR17}, 
and for the corresponding gradient flows, see
\cite{Simonett01,KuwertS02}. However, problems involving the
multi-phase Canham--Helfrich--Evans have not been treated
mathematically in much detail yet. We refer to 
\cite{ChoksiMV13,Helmers13,Helmers15,BrazdaLS19preprint} for first results. 
Available
related results for the corresponding gradient flow are restricted to boundary
value problems for Willmore flow with line tension, cf.\ \cite{AbelsGM15}, and
to the evolution of elastic flows with junctions, see 
\cite{GarckeMP19,DallAcquaLP19}. 

Numerical approaches for the evolution of two-phase membranes often
rely on phase field methods, see \cite{WangD08,LowengrubRV09,ElliottS10,%
ElliottS10a,ElliottS13,nsns2phase}. \cite{CoxL15} numerically
studied solutions for the shape equations for two-phase vesicles
numerically and \cite{pwfc0c1} solved the gradient flow dynamics of
two-phase biomembranes formulated in a sharp interface setting
numerically.
A numerical method for the evolution of elastic flows with junctions has been
proposed in \cite{pwftj}. 
In this paper, we will present a parametric finite element method 
for the $L^2$--gradient flow of \eqref{eq:areaE} in
an axisymmetric setting. Throughout the paper, we will make extensive use of
our recent work \cite{axipwf}, in which the analogous gradient flow for a
more general energy for a single surface has been treated.

The outline of the paper is as follows. In Section~\ref{sec:axis} we
derive the axisymmetric version of the governing equations. For the
finite element method it is important to derive a weak formulation for
the highly nonlinear problem. This is done in Section~\ref{sec:weak}
using an approach based on a Lagrangian method. In
Section~\ref{sec:sd} a semidiscretisation is developed which
preserves important energy decay and conservation properties.
In Section~\ref{sec:fd} we analyze a fully discrete version of the
method developed in the previous section and show that the resulting
equations are well-posed. Numerical results are given in
Section~\ref{sec:nr} and a comparison with the seminal experimental
paper by \cite{BaumgartHW03} is given. Finally, in an Appendix, we
show that the weak formulation introduced in Section~\ref{sec:weak} is
consistent with the strong formulation.

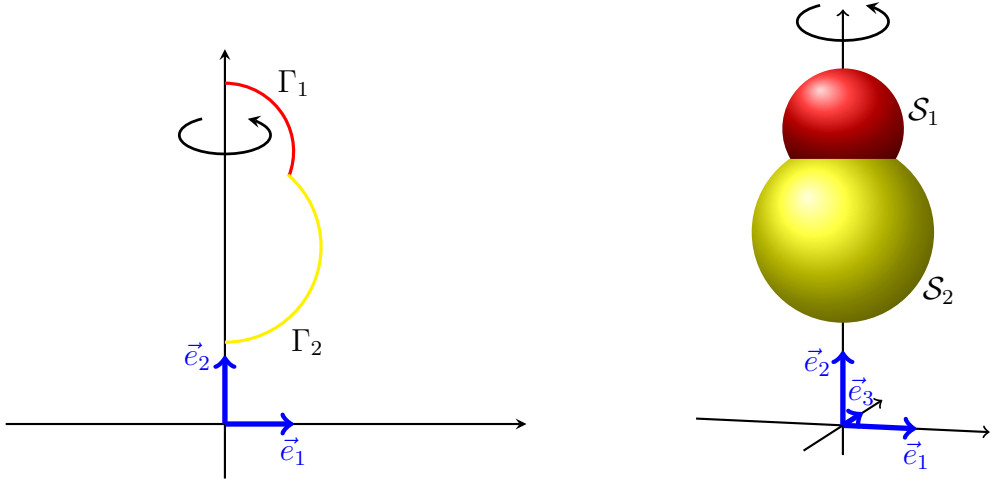
\begin{figure}
\center
\newcommand{\AxisRotator}[1][rotate=0]{%
    \tikz [x=0.25cm,y=0.60cm,line width=.2ex,-stealth,#1] \draw (0,0) arc (-150:150:1 and 1);%
}
\begin{tikzpicture}[every plot/.append style={very thick}, scale = 1]
\begin{axis}[axis equal,axis line style=thick,axis lines=center, xtick style ={draw=none}, 
ytick style ={draw=none}, xticklabels = {}, 
yticklabels = {}, 
xmin=-0.2, xmax = 0.8, ymin = -0.4, ymax = 2.75]
after end axis/.code={  
   \node at (axis cs:0.0,2.1) {\AxisRotator[rotate=-90]};
   \draw[blue,->,line width=2pt] (axis cs:0,0) -- (axis cs:0.5,0);
   \draw[blue,->,line width=2pt] (axis cs:0,0) -- (axis cs:0,0.5);
   \node[blue] at (axis cs:0.5,-0.2){$\vec\ek_1$};
   \node[blue] at (axis cs:-0.2,0.5){$\vec\ek_2$};
   \draw[red,very thick] (axis cs: 0,2.5) arc[radius = 50, start angle= 90, end angle= -21];
   \node[black] at (axis cs:0.5,2.5){$\Gamma_1$};
   \draw[yellow,very thick] (axis cs: 0,0.6) arc[radius = 70, start angle= -90, end angle= 49];
   \node[black] at (axis cs:0.6,0.6){$\Gamma_2$};
}
\end{axis}
\end{tikzpicture} \qquad \qquad
\tdplotsetmaincoords{100}{75}
\begin{tikzpicture}[scale=2, tdplot_main_coords,axis/.style={->},thick]
\draw[axis] (-1, 0, 0) -- (1, 0, 0);
\draw[axis] (0, -1, 0) -- (0, 1, 0);
\draw[axis] (0, 0, -0.2) -- (0, 0, 2.8);
\draw[blue,->,line width=2pt] (0,0,0) -- (0,0.5,0) node [below] {$\vec\ek_1$};
\draw[blue,->,line width=2pt] (0,0,0) -- (0,0.0,0.5);
\draw[blue,->,line width=2pt] (0,0,0) -- (0.5,0.0,0);
\node[blue] at (0.1,0.1,0.22){$\vec\ek_3$};
\node[blue] at (0.1,-0.2,0.4){$\vec\ek_2$};
\node[black] at (0.2,0.5,2.1){$\mathcal{S}_1$};
\node[black] at (0.2,0.6,0.9){$\mathcal{S}_2$};
\node at (0.0,0.0,2.7) {\AxisRotator[rotate=-90]};

\tdplottransformmainscreen{0}{0}{1.3}
\shade[tdplot_screen_coords, ball color = yellow] (\tdplotresx,\tdplotresy) circle (0.6);
\tdplottransformmainscreen{0}{0}{1.8}
\tdplottransformmainscreen{0}{0.36}{1.81}
\shade[tdplot_screen_coords, ball color = red] (\tdplotresx,\tdplotresy) arc (-30:210:0.4);
\end{tikzpicture}
\caption{Sketch of $\Gamma_i$ and $\mathcal{S}_i$, $i=1,2$,
as well as the unit vectors $\vec\ek_1$, $\vec\ek_2$ and $\vec\ek_3$.}
\label{fig:sketch}
\end{figure}

\setcounter{equation}{0}
\section{The axisymmetric setting}\label{sec:axis}

For the axisymmetric setting, we assume that 
$\vec x_i(t) : \overline I_i \to \bRgeq\times\bR$ are
parameterisations of $\Gamma_i(t)$, $i=1,2$,
with $I_1 = (0,\frac12)$ and $I_2 = (\frac12,1)$, and such that 
$\vec x_1(\tfrac12,t) = \vec x_2(\tfrac12,t)$ and
$\vec x_i(\rho,t)\,.\,\vec\ek_1 = 0$ if and only if 
$\rho \in \partial I_i \setminus \{\frac12\}$, $i=1,2$,
for all $t\in[0,T]$.
Throughout $\Gamma_i(t)$ represents the generating curve of a
surface $\mathcal{S}_i(t)$ 
that is axisymmetric with respect to the $x_2$--axis, see
Figure~\ref{fig:sketch}. In particular, on defining
\begin{equation*} 
\vec\Pi_3^3(r, z, \theta) = 
(r\,\cos\theta, z, r\,\sin\theta)^T 
\quad\text{for}\quad r\in \bRgeq\,,\ z \in \bR\,,\ \theta \in [0,2\,\pi]
\end{equation*}
and
$\Pi_2^3(r, z) = \{\vec\Pi_3^3(r, z, \theta) : \theta \in [0,2\,\pi)\}$,
we have that
\begin{equation} \label{eq:SGamma}
\mathcal{S}_i(t) = 
\bigcup_{(r,z)^T \in \Gamma_i(t)} \Pi_2^3(r, z)
= \bigcup_{\rho \in \overline{I}_i} \Pi_2^3(\vec x_i(\rho,t))
\quad\text{and}\quad
\gamma(t) = \Pi_2^3(\vec x_1(\tfrac12,t)) = \Pi_2^3(\vec x_2(\tfrac12,t))
\,.
\end{equation}

On assuming, for $t \in [0,T]$ and $i=1,2$, that
\begin{equation*} 
|[\vec x_i]_\rho| \geq c_0 > 0 \qquad \forall\ \rho \in \overline I_i\,,
\end{equation*}
we introduce the arclength $s$ of the curves, i.e.\ $\partial_s =
|[\vec x_i]_\rho|^{-1}\,\partial_\rho$ in $I_i$, and set
\begin{equation} \label{eq:tau}
\vec\tau_i(\rho,t) = [\vec x_i]_s(\rho,t) = 
\frac{[\vec x_i]_\rho(\rho,t)}{|[\vec x_i]_\rho(\rho,t)|} \quad \text{and}
\quad \vec\nu_i(\rho,t) = -[\vec\tau_i(\rho,t)]^\perp
\quad \text{in } \overline{I}_i \,,
\end{equation}
where $(\cdot)^\perp$ denotes a clockwise rotation by $\frac{\pi}{2}$.
Then the normal velocity $\mathcal{V}_{\mathcal{S}_i}$ 
of $\mathcal{S}_i(t)$ 
in the direction $\unitn_{\mathcal{S}_i}$ is given by
\begin{equation*} 
\mathcal{V}_{\mathcal{S}_i} = [\vec x_i]_t(\rho,t)\,.\,\vec\nu_i(\rho,t) 
\quad\text{on } \Pi_2^3(\vec x_i(\rho,t)) \subset \mathcal{S}_i(t)
\quad \forall\ \rho \in \overline I_i\,,\ t \in [0,T]
\,,\ i = 1,2\,.
\end{equation*}
For the curvature $\varkappa_i$ of $\Gamma_i(t)$ it holds that
\begin{equation} \label{eq:varkappa}
\varkappa_i\,\vec\nu_i = \vec\varkappa_i = [\vec\tau_i]_s =
\frac1{|[\vec x_i]_\rho|} \left[ \frac{[\vec x_i]_\rho}{|[\vec x_i]_\rho|} 
\right]_\rho \quad \text{in }\ \overline I_i
\,,\ i = 1,2\,.
\end{equation}
We recall that the mean curvature and Gaussian curvature of $\mathcal{S}_i(t)$
are then given by
\begin{equation} \label{eq:meanGaussS}
\varkappa_{\mathcal{S}_i} = \varkappa_i - 
\frac{\vec\nu_i\,.\,\vec\ek_1}{\vec x_i\,.\,\vec\ek_1}
\quad\text{and}\quad
\Gauss_{\mathcal{S}_i} = -
\varkappa_i\,\frac{\vec\nu_i\,.\,\vec\ek_1}{\vec x_i\,.\,\vec\ek_1}
= \varkappa_i\,(\varkappa_{\mathcal{S}_i}-\varkappa_i)
\quad\text{in }\ \overline{I}_i\,,\ i = 1,2\,,
\end{equation}
respectively; see \eqref{eq:appmeanGaussS} in Appendix~\ref{sec:B}.
More precisely, if $k_{m,i}$ and $k_{g,i}$
denote the mean and Gaussian curvatures of $\mathcal{S}_i(t)$, then
\begin{equation*} 
k_{m,i} = \varkappa_{\mathcal{S}_i}(\rho,t) \ \text{ and }\
k_{g,i} = \Gauss_{\mathcal{S}_i}(\rho,t) 
\quad\text{on } \Pi_2^3(\vec x_i(\rho,t)) \subset \mathcal{S}_i(t)
\quad \forall\ \rho \in \overline I_i\,,\ t \in [0,T]\,.
\end{equation*}

Clearly, for a smooth surface with bounded curvatures it follows from
(\ref{eq:meanGaussS}) that 
\begin{equation} \label{eq:bcnu}
\vec\nu_i(\rho,t) \,.\,\vec\ek_1 = 0
\qquad \forall\ \rho \in \partial I_i \setminus \{\tfrac12\}
\,,\quad \forall\ t\in[0,T]\,,\ i = 1,2\,,
\end{equation}
which is equivalent to
\begin{equation} \label{eq:bc}
[\vec x_i]_\rho(\rho,t) \,.\,\vec\ek_2 = 0
\qquad \forall\ \rho \in \partial I_i \setminus \{\tfrac12\}\,,
\quad \forall\ t\in[0,T]\,,\ i = 1,2\,.
\end{equation}
We note that for the singular fraction in \eqref{eq:meanGaussS} 
it follows from \eqref{eq:bc} and (\ref{eq:bcnu}), on recalling 
(\ref{eq:varkappa}), that
\begin{align}
\lim_{\rho\to \rho_0} 
\frac{\vec\nu_i(\rho,t)\,.\,\vec\ek_1}{\vec x_i(\rho,t)\,.\,\vec\ek_1} &
= \lim_{\rho\to \rho_0} 
\frac{[\vec\nu_i]_\rho(\rho,t)\,.\,\vec\ek_1}
{[\vec x_i]_\rho(\rho,t)\,.\,\vec\ek_1}
= [\vec\nu_i]_s(\rho_0,t)\,.\,\vec\tau_i(\rho_0,t) 
= -\varkappa_i(\rho_0,t)
\nonumber \\ & \hspace{3cm}
\quad \forall\ \rho_0\in\partial I_i \setminus \{\tfrac12\}\,,\
\forall\ t \in [0,T]\,,\ i = 1,2\,.
\label{eq:bclimit}
\end{align}
Moreover, 
on recalling \eqref{eq:idss}, it is easily seen that
\begin{equation} \label{eq:kdSmu}
k_{\partial\mathcal{S}_i,\rm n} = - 
\frac{\vec\nu_i(\tfrac12,t)\,.\,\vec\ek_1}{\vec x_i(\tfrac12,t)\,.\,\vec\ek_1}
\ \text{ and }\
k_{\partial\mathcal{S}_i,\mu}  = - 
\frac{\vec\mu_i(\tfrac12,t)\,.\,\vec\ek_1}{\vec x_i(\tfrac12,t)\,.\,\vec\ek_1}
\quad\text{on } 
\gamma(t)
\quad \forall \ t \in [0,T]\,,\ i= 1,2\,,
\end{equation}
where $\vec\nu_i(\cdot,t)$ is the unit normal on $\Gamma_i(t)$ 
as defined in (\ref{eq:tau}), and where
\begin{equation} \label{eq:mu}
\vec\mu_1(\tfrac12,t) = \vec\tau_1(\tfrac12,t)\,,\
\vec\mu_2(\tfrac12,t) = -\vec\tau_2(\tfrac12,t)
\quad \forall \ t \in [0,T]\,,
\end{equation}
denotes the corresponding conormals
of $\Gamma_i(t)$ at the endpoint $\vec x_1(\tfrac12,t) = \vec x_2(\tfrac12,t)$.
Here we have recalled that the conormal $\vec\mu_{\partial\mathcal{S}_i}$ 
points out of $\mathcal{S}_i(t)$.

We consider the following axisymmetric energy that is equivalent to 
\eqref{eq:EE} for flows of axisymmetric surfaces without topological changes
\begin{align}
& \widetilde E(\vec x(t))
= E(\mathcal{S}(t)) - 2\,\pi\,\sum_{i=1}^2\alpha^G_i\,m(\mathcal{S}_i(t)) 
\nonumber \\ & \quad
= \sum_{i=1}^2 \left[ \pi\,\alpha_i\,\int_{I_i} 
\vec x_i\,.\,\vec\ek_1
\left[ \varkappa_{\mathcal{S}_i} -\spont_i \right]^2 |[\vec x_i]_\rho| \drho
 \right] 
- 2\,\pi\,\sum_{i=1}^2\alpha^G_i\,\vec\mu_i(\tfrac12)\,.\,\vec\ek_1 
+ \pi\,\varsigma\,\sum_{i=1}^2 \vec x_i(\tfrac12)\,.\,\vec\ek_1\,. 
\label{eq:Ec0c1}
\end{align}
In a similar fashion, we define an axisymmetric analogue of (\ref{eq:areaE})
as
\begin{equation} \label{eq:areaEc0c1}
\widetilde E_\lambda(\vec x(t)) 
= \widetilde E(\vec x(t)) + 
\sum_{i=1}^2 \lambda_{A,i}\,A(\vec x(t)) + \lambda_V\,V(\vec x(t))\,,
\end{equation}
where we have defined, observe \eqref{eq:localarea} in Appendix~\ref{sec:B},
\begin{equation} \label{eq:Area}
A_i(\vec x(t))
= 2\,\pi \,\int_{I_i} \vec x_i\,.\,\vec\ek_1\,|[\vec x_i]_\rho| \drho
= \mathcal{H}^2(\mathcal{S}_i(t))
\end{equation}
and, see e.g.\ \citet[(3.10)]{axisd},
\begin{equation} \label{eq:Volume}
V(\vec x(t)) = \pi \sum_{i=1}^2 
\int_{I_i} (\vec x_i\,.\,\vec\ek_1)^2 \,\vec\nu_i\,.\,\vec\ek_1\,
|[\vec x_i]_\rho| \drho
= \mathcal{L}^3(\Omega(t)) \,.
\end{equation}
For later use we observe that
\begin{equation} \label{eq:dAdt}
\ddt\, A_i(\vec x(t))
= 2\,\pi \int_{I_i} \left[ [\vec x_i]_t\,.\,\vec\ek_1\,|[\vec x_i]_\rho| +
 (\vec x_i\,.\,\vec\ek_1) \,([\vec x_i]_t)_\rho\,.\,\vec\tau_i \right] \drho
\,,\ i=1,2\,,
\end{equation}
and
\begin{equation} \label{eq:dVdt}
\ddt\,V(\vec x(t))
= 2\,\pi\,\sum_{i=1}^2 \int_{I_i} (\vec x_i\,.\,\vec\ek_1)
\,[\vec x_i]_t\,.\,\vec\nu_i\,|[\vec x_i]_\rho|\drho \,.
\end{equation}

The axisymmetric formulation of the gradient flow (\ref{eq:gradflowlambda}) 
is now given by
\begin{align} \label{eq:xtbgnlambda}
(\vec x_i\,.\,\vec\ek_1)\,[\vec x_i]_t\,.\,\vec\nu_i & = 
- \alpha_i\,[\vec x_i\,.\,\vec\ek_1\,(\varkappa_{\mathcal{S}_i})_{s}]_s
+ 2\,\alpha_i\,\vec x_i\,.\,\vec\ek_1 \, (\varkappa_{\mathcal{S}_i} 
- \spont_i) \,\Gauss_{\mathcal{S}_i}
\nonumber \\ & \quad
- \vec x_i\,.\,\vec\ek_1
\left[\tfrac12\,\alpha_i\,(\varkappa_{\mathcal{S}_i}^2 -\spont^2_i)
- \lambda_{A,i} \right] \varkappa_{\mathcal{S}_i}
- \lambda_V \,\vec x_i\,.\,\vec\ek_1
\quad\text{in }\ \overline I_i\,,\ i=1,2\,,
\end{align}
where for the first term on the right hand side of \eqref{eq:xtbgnlambda} 
we have observed \eqref{eq:LBSrad} in Appendix~\ref{sec:B}.
At an interface between the two phases, we require axisymmetric versions
of the boundary conditions (\ref{eq:C0bc}) and (\ref{eq:C1bc}).
First of all, we notice that the geodesic torsion $\mathfrak{t}_i$ of $\gamma(t)$ with respect
to $\mathcal{S}_i$, $i=1,2$, is zero in the axisymmetric setting and hence 
the terms involving the geodesic torsion vanish, see also 
\citet[(2.25)]{axipwf}.
In the $C^0$--case, relating to \eqref{eq:C0bc}, 
we have for the axisymmetric situation the following
conditions at the point $\rho=\frac 12$ and for $t\in[0,T]$:
\begin{subequations} \label{eq:axiC0bc}
\begin{align}
& \alpha_i\,(\varkappa_{\mathcal{S}_i} - \spont_i) - \alpha^G_i\,
\frac{\vec\nu_i\,.\,\vec\ek_1}{\vec x\,.\,\vec\ek_1}
 = 0\,,\quad  i = 1,2\,,  \label{eq:axiC0bc1} \\
& \sum_{i=1}^2 
\left[ 
(-1)^{i-1}\, \alpha_i\,(\varkappa_{\mathcal{S}_i})_s\, \vec\nu_i
-\left(\tfrac 12\,\alpha_i\, ( \varkappa_{\mathcal{S}_i}- \spont_i )^2 
+\alpha_i^G\,\Gauss_{\mathcal{S}_i} +\lambda_{A,i} \right) \vec\mu_i  \right]
- \frac\varsigma{\vec x \,.\,\vec\ek_1}\, \vec\ek_1  =  \vec 0\,,
   \label{eq:axiC0bc2}
\end{align}
\end{subequations}
where we used the notation $\vec x=\vec x_1=\vec x_2$ at
$\rho=\frac 12$.
For the $C^1$--case, and so corresponding to \eqref{eq:C1bc}, 
we obtain at $\rho=\frac 12$ and for $t\in[0,T]$:
\begin{subequations} \label{eq:axiC1bc}
\begin{align}
& [
\alpha_i\,(\varkappa_{\mathcal{S}_i} - \spont_i)]_1^2 - [\alpha^G_i]_1^2\,
\frac{\vec\nu\,.\,\vec\ek_1}{\vec x\,.\,\vec\ek_1}=0
\,, \label{eq:axiC1bc1} \\
& -[\alpha_i\,(\varkappa_{\mathcal{S}_i} )_s]_1^2
-\varsigma\, \frac{\vec\nu\,.\,\vec\ek_1}{\vec x\,.\,\vec\ek_1}=0\,,
\label{eq:axiC1bc2} \\
& [-\tfrac12\,\alpha_i\,(\varkappa_{\mathcal{S}_i}  - \spont_i)^2
+\alpha_i\,(\varkappa_{\mathcal{S}_i} - \spont_i)\,\varkappa_i
- \lambda_{A,i}]_1^2
 -\varsigma\,\frac{\vec\mu\,.\,\vec\ek_1}{\vec x\,.\,\vec\ek_1} = 0
\,, \label{eq:axiC1bc3}
\end{align}
\end{subequations}
where we have defined $\vec\nu=\vec\nu_1=\vec\nu_2$ and 
$\vec\mu=\vec\mu_2=-\vec\mu_1$ at $\rho=\frac12$, and where we
have used (\ref{eq:idss}), (\ref{eq:meanGaussS}) and (\ref{eq:kdSmu}).

Finally, we impose the following boundary conditions at the axis of rotation,
for $t\in[0,T]$:
\begin{subequations}
\label{eq:part0I}
\begin{align}
\vec x_i\,.\,\vec\ek_1 &= 0 \ \text{on } \partial I_i \setminus\{\tfrac12\}\,, 
\label{eq:fixed0} \\
[\vec x_i]_\rho\,.\,\vec\ek_2 & = 0\ \text{on } 
\partial I_i \setminus\{\tfrac12\}\,,
\label{eq:bcbc} \\
[\varkappa_{\mathcal{S}_i}]_\rho &= 0 \ \text{on } 
\partial I_i \setminus\{\tfrac12\}\,.
\label{eq:sdbca} 
\end{align}
\end{subequations}
Here (\ref{eq:sdbca}) ensures that the radially
symmetric functions on $\mathcal{S}_i(t)$ induced by 
$\varkappa_{{\mathcal S}_i}$, $i=1,2$,
are differentiable, while \eqref{eq:bcbc} is the same as \eqref{eq:bc}. 

Clearly, for surface area and volume conserving flows, the Lagrange multipliers
\linebreak
$(\lambda_{A,1}(t),\lambda_{A,2}(t),\lambda_V(t))^T \in \bR^3$ 
in \eqref{eq:xtbgnlambda} need to be chosen such that 
\begin{equation} \label{eq:xsidedtSAV}
\ddt\,\int_{I_i} \vec x_i\,.\,\vec\ek_1\,|[\vec x_i]_\rho| \drho
= 0\,,\ i=1,2\,, \qquad 
\ddt\, \sum_{i=1}^2 
\left( (\vec x_i\,.\,\vec\ek_1)^2 ,\vec\nu_i\,.\,\vec\ek_1\,
|[\vec x_i]_\rho| \right)
= 0\,,
\end{equation}
where we recall \eqref{eq:dAdt} and \eqref{eq:dVdt}.

\setcounter{equation}{0}
\section{Weak formulation} \label{sec:weak}

Using the formal calculus of PDE constrained optimisation, in this section we
derive a weak formulation for the gradient flow \eqref{eq:xtbgnlambda}.
The necessary techniques are described in \citet[\S9.3]{bgnreview}, and details
for the case of a one-phase axisymmetric surface can be found in 
\citet[\S3.1]{axipwf}.
The fact that the obtained weak formulation is indeed consistent with
\eqref{eq:xtbgnlambda} and the boundary conditions \eqref{eq:axiC0bc},
\eqref{eq:axiC1bc} and \eqref{eq:part0I} will be shown in Appendix~\ref{sec:A}.

We begin by defining the following function spaces. Let
\begin{align*} 
\xspace_i &= \{\vec\eta_i \in [H^1(I_i)]^2 :
\vec\eta_i(\rho)\,.\,\vec\ek_1 = 0 \quad \forall\ \rho \in \partial I_i
\setminus\{\tfrac12\}\}\,,\ i=1,2\,, \nonumber \\
\xspace &= \{(\vec\eta_1,\vec\eta_2) \in \mathop{\times}_{i=1}^2 : \xspace_i :
\vec\eta_1(\tfrac12) = \vec\eta_2(\tfrac12)\}\,, 
\end{align*}
as well as $\yspace = \yspace_1 \times \yspace_2$,
with $\yspace_i = \xspace_i$, $i=1,2$, and
\begin{equation} \label{eq:Vyczcoi}
\Vycztest = \{(\vec\eta_1,\vec\eta_2) \in \yspace : 
\vec\eta_1(\tfrac12) = \vec\eta_2(\tfrac12) = \vec 0 \}\,,\quad
\Vycotest = \{(\vec\eta_1,\vec\eta_2) \in \yspace : 
\vec\eta_1(\tfrac12) = \vec\eta_2(\tfrac12)\}\,. 
\end{equation}

For later use, we define the first variation of a differentiable quantity 
$B(\vec x)$, in the direction $\vec\chi$ as
\begin{equation*} 
\left[\deldel{\vec x}\,B(\vec x)\right](\vec\chi)
= \lim_{\varepsilon \rightarrow 0}
\frac{B(\vec x + \varepsilon\, \vec\chi)-B(\vec x)}{\varepsilon}\,,
\end{equation*}
and we recall, for example, the variations of some geometric quantities
from \citet[(3.3)]{axipwf}.

Let $(\cdot,\cdot)$ denote both the $L^2$--inner product on $I_1$ and on $I_2$.
It will always be clear from the integrand which product is meant, and so we
use this abuse of notation throughout the paper.
We now consider the following weak formulation of (\ref{eq:varkappa})
with $\vec x_i \in \xspace_i$ and $\varkappa_i \in L^2(I_i)$ such that 
\begin{equation} \label{eq:varkappaweak}
\left( \varkappa_i\,\vec\nu_i,\vec\eta_i\, |[\vec x_i]_\rho| \right)
+ \left(\vec\tau_i,[\vec\eta_i]_\rho \right) =
 \left[\vec{\rm m}_i\,.\,\vec\eta_i\right](\tfrac12)
\qquad \forall\ \vec\eta_i \in \Vyczcoi\,,\ i = 1,2\,,
\end{equation}
where we recall (\ref{eq:tau}).
We note that (\ref{eq:varkappaweak}) weakly imposes (\ref{eq:bc}).
However, (\ref{eq:varkappaweak}) also yields that  
$\vec{\rm m}_i(\frac12) = \vec\mu_i(\frac12) \in \bR^2$.
This will not be the case under discretisation, where 
$\vec{\rm m}_i(\frac12) \in \bR^2$
is an approximation to the conormal $\vec\mu_i(\frac12)$.   
As $\vec{\rm m}_i$ are only defined at $\rho=\tfrac12$, we simply write
$\vec{\rm m}_i$ for $\vec{\rm m}_i(\tfrac12)$ from now on.
On introducing the parameter $C_1 \in \{0,1\}$, we can easily
model the case of either a $C^0$-- or a $C^1$--junction with the help of the
side constraint 
\begin{equation} \label{eq:mC0C1}
C_1\,(\vec{\rm m}_1 + \vec{\rm m}_2) = \vec 0\,.
\end{equation}
We remark that upon discretisation, \eqref{eq:varkappaweak} leads to an
equidistribution property in the two phases. We refer to the recent review
article \cite{bgnreview}, and to Remark~\ref{rem:equid} below, for more
details.

Now, in order to study the $L^2$--gradient flow of the energy \eqref{eq:Ec0c1},
subject to the side constraints \eqref{eq:varkappaweak} and \eqref{eq:mC0C1},
we consider the Lagrangian
\begin{align} \label{eq:Lag}
& \mathcal{L}((\vec x_i, \varkappa_i^\star, \vec{\rm m}_i, \vec y_i)_{i=1}^2,
\vec\phi) =
\pi\,\sum_{i=1}^2 \left( \alpha_i
\left[ \varkappa_i^\star 
- \frac{\vec\nu_i\,.\,\vec\ek_1}{\vec x_i\,.\,\vec\ek_1}
-\spont_i \right]^2, 
\vec x_i\,.\,\vec\ek_1 \,|[\vec x_i]_\rho|\right) \nonumber \\ & \qquad
+ \pi\,\varsigma\,\sum_{i=1}^2 \vec x_i(\tfrac12)\,.\,\vec\ek_1
- \sum_{i=1}^2 \left(\varkappa_i^\star\,\vec\nu_i,\vec y_i\,|[\vec x_i]_\rho|
\right) 
- \sum_{i=1}^2 \left(\vec\tau_i, [\vec y_i]_\rho \right) 
\nonumber \\ & \qquad
 + \sum_{i=1}^2 \vec{\rm m}_i\,.\left(\vec y_i(\tfrac12) -
2\,\pi\,\alpha^G_i\,\vec\ek_1\right)
+ C_1\,(\vec{\rm m}_1 + \vec{\rm m}_2)\,.\,\vec\phi \,,
\end{align}
for $\vec x = (\vec x_1,\vec x_2) \in \xspace$, 
$\varkappa^\star = (\varkappa_1^\star,\varkappa_2^\star) 
\in L^2(I_1) \times L^2(I_2)$, 
$(\vec{\rm m}_1,\vec{\rm m}_2) \in [\bR^2]^2$, 
$\vec y = (\vec y_1,\vec y_2) \in \yspace$ and $\vec\phi \in \bR^2$.

Upon taking the appropriate variations
$\vec\chi = (\vec\chi_1,\vec\chi_2) \in \xspace$ in $\vec x$,
$\chi_i \in L^2(I_i)$ in $\varkappa_i^\star$, 
$\vec z_i \in \bR^2$ in $\vec{\rm m}_i$, 
$\vec\eta_i\in \Vyczcoi$ in $\vec y_i$ and $\vec w \in \bR^2$ in $\vec\phi$,
we obtain our desired weak formulation,
see also \citet[\S3.1]{axipwf} for more details.
For example, the variations in $\vec {\rm m}_i$ yield that
\begin{equation} \label{eq:yalphaG}
- 2\,\pi\,\alpha^G_i\,\vec\ek_1 + \vec y_i(\tfrac12) + C_1\,\vec\phi = \vec 0
\,,\ i = 1,2\,.
\end{equation}
Moreover, taking variations $\vec\eta_i \in \Vyczcoi$ in $\vec y_i$, 
and setting $\left[\deldel{\vec y_i}\, \mathcal{L}\right](\vec\eta_i) = 0$
gives \eqref{eq:varkappaweak}, with $\varkappa_i$ replaced by
$\varkappa_i^\star$. Thus we obtain $\varkappa^\star_i=\varkappa_i$, $i=1,2$,
and we are going to use these identities from now on.

Taking variations $\chi_i \in L^2(I_i)$ in $\varkappa_i^\star$ and setting
$\left[\deldel{\varkappa_i^\star}\, \mathcal{L}\right](\chi_i) = 0$
we obtain, on using $\varkappa_i^\star = \varkappa_i$, that
\begin{equation*} 
2\,\pi\,\alpha_i\left( 
\varkappa_i - \frac{\vec\nu_i\,.\,\vec\ek_1}{\vec x_i\,.\,\vec\ek_1}
-\spont_i, \vec x_i\,.\,\vec\ek_1\,\chi_i\,|[\vec x_i]_\rho|\right)
- \left(\vec\nu_i\,.\,\vec y_i,\chi_i \,|[\vec x_i]_\rho|\right) = 0
\quad \forall\ \chi_i \in L^2(I_i)\,,\ i=1,2\,,
\end{equation*}
which implies that
\begin{equation} \label{eq:kappaid}
2\,\pi\,\vec x_i\,.\,\vec\ek_1 \,
\alpha_i\left[\varkappa_i - \frac{\vec\nu_i\,.\,\vec\ek_1}
{\vec x_i\,.\,\vec\ek_1} - \spont_i\right]
= \vec y_i\,.\,\vec\nu_i
\quad \text{in }\ \overline I_i\,,\ i = 1,2\,.
\end{equation}
Finally, taking variations in $\vec\phi$ and setting them to zero 
gives \eqref{eq:mC0C1}. Setting $\vec x (\cdot ,t) = (\vec x_1, \vec
x_2)(\cdot ,t) \in \xspace$, the evolution law for $\vec x$ is given as
\begin{equation*}
  2\,\pi\, \sum_{i=1}^2 \left((\vec x_i\,.\, \vec\ek_1)\,[\vec x_i]_t\,.\,
  \vec\nu_i, \vec\chi_i\,.\,\vec\nu_i\, |[\vec x_i]_\rho|\right) 
= - \left[\frac{\delta}{\delta\vec x}\, \mathcal{L}\right](\vec\chi) 
\qquad \forall\ \vec\chi = (\vec \chi_1,\vec \chi_2)\in \xspace\,.
\end{equation*}
Here, the term on the left hand side is the normal part of the
velocity integrated on the surface against the test function, which is
the natural term for a gradient flow formulation. 

Overall we obtain the following weak formulation,
compare with \citet[(3.22)]{axipwf}.
Let $(\vec x_1, \vec x_2)(\cdot,0) \in \xspace$ and $\alpha_i \in \bRplus$,
$\spont_i,\alpha^G_i \in \bR$ be given for $i=1,2$.
For $t \in (0,T]$, find $(\vec x_1,\vec x_2)(\cdot,t) \in \xspace$, 
$(\varkappa_i,\vec{\rm m}_i,\vec y_i) 
\in L^2(I_i) \times \bR^2 \times \Vyczcoi$, $i=1,2$, and
$C_1\,\vec\phi \in \bR^2$ such that 
\begin{subequations} \label{eq:weak3}
\begin{align}
& 2\,\pi\, \sum_{i=1}^2
\left((\vec x_i\,.\,\vec\ek_1)\,[\vec x_i]_t\,.\,\vec\nu_i,
\vec\chi_i\,.\,\vec\nu_i\,|[\vec x_i]_\rho|\right)
- \sum_{i=1}^2
\left([\vec y_i]_\rho\,.\,\vec\nu_i, [\vec\chi_i]_\rho\,.\,\vec\nu_i
\,|[\vec x_i]_\rho|^{-1}\right) 
\nonumber \\ & \quad
= 
- \pi\,\sum_{i=1}^2 \left( \alpha_i
\left[\varkappa_i - \frac{\vec\nu_i\,.\,\vec\ek_1}{\vec x_i\,.\,\vec e_1} 
-\spont_i\right]^2 , 
\vec\chi_i\,.\,\vec\ek_1\,|[\vec x_i]_\rho| + (\vec x_i\,.\,\vec\ek_1)\,
\vec\tau_i\,.\,[\vec\chi_i]_\rho \right) 
\nonumber \\ & \qquad \qquad 
- 2\,\pi\,\sum_{i=1}^2 \alpha_i
\left(\varkappa_i - \frac{\vec\nu_i\,.\,\vec\ek_1}{\vec x_i\,.\,\vec e_1} 
- \spont_i, \frac{\vec\nu_i\,.\,\vec e_1}{\vec x_i\,.\,\vec e_1}\,
\vec\chi_i\,.\,\vec\ek_1 \,|[\vec x_i]_\rho| \right) 
\nonumber \\ & \qquad \qquad 
- 2\,\pi\,\sum_{i=1}^2 \alpha_i \left(
\varkappa_i - \frac{\vec\nu_i\,.\,\vec\ek_1}{\vec x_i\,.\,\vec e_1} - \spont_i,
(\vec\tau_i\,.\,\vec\ek_1)\,[\vec\chi_i]_\rho\,.\,\vec\nu_i \right) 
+ \sum_{i=1}^2 \left( \varkappa_i\, \vec y_i^\perp, [\vec\chi_i]_\rho \right)
\nonumber \\ & \qquad \qquad
- \pi\,\varsigma\,\sum_{i=1}^2 \vec\chi_i(\tfrac12)\,.\,\vec\ek_1
\qquad \forall\ \vec\chi = (\vec\chi_1,\vec\chi_2) \in \xspace \,,
\label{eq:weak3a} \\
&2\,\pi \left(\alpha_i
\left[\varkappa_i - \frac{\vec\nu_i\,.\,\vec\ek_1}{\vec x_i\,.\,\vec\ek_1}
-\spont_i\right],
\vec x_i\,.\,\vec\ek_1\,\chi_i\,|[\vec x_i]_\rho|\right)
- \left(\vec\nu_i\,.\,\vec y_i,\chi_i\,|[\vec x_i]_\rho|\right) = 0
\nonumber \\ & \hspace{9cm}
\qquad \forall\ \chi_i \in L^2(I_i)\,,\ i=1,2\,, \label{eq:weak3b} \\
&\left( \varkappa_i\,\vec\nu_i,\vec\eta_i\, |[\vec x_i]_\rho| \right)
+ \left([\vec x_i]_\rho,[\vec\eta_i]_\rho \,|[\vec x_i]_\rho|^{-1}\right) =
 \vec{\rm m}_i\,.\,\vec\eta_i(\tfrac12)
\qquad \forall\ \vec\eta_i \in \Vyczcoi\,,\ i = 1,2\,,  \label{eq:weak3c} \\
&- 2\,\pi\,\alpha^G_i\,\vec\ek_1 + \vec y_i(\tfrac12) + C_1\,\vec\phi = \vec 0
\,,\ i = 1,2\,,\label{eq:weak3d} \\
& C_1\,(\vec{\rm m}_1 + \vec{\rm m}_2) = \vec 0\,. \label{eq:weak3e} 
\end{align}
\end{subequations}

\begin{rem} \label{rem:C0}
In the case $C_1=0$, the condition \eqref{eq:weak3d} reduces to the
Dirichlet boundary condition $\vec y_i(\tfrac12) =
2\,\pi\,\alpha^G_i\,\vec\ek_1$, $i=1,2$. Moreover, 
the condition \eqref{eq:weak3e} disappears, and so $\vec{\rm m}_i$ can be
eliminated from the formulation by replacing $\vec\eta_i \in \Vyczcoi$ in 
\eqref{eq:weak3c} with test functions such that $\vec\eta_i(\frac12)=\vec0$. 
The resulting formulation is to find
$(\vec x_1,\vec x_2)(\cdot,t) \in \xspace$, 
$(\varkappa_i,\vec y_i) 
\in L^2(I_i) \times \Vyczcoi$
with $\vec y_i(\tfrac12) = 2\,\pi\,\alpha^G_i\,\vec\ek_1$, $i=1,2$, such that 
\eqref{eq:weak3a}, \eqref{eq:weak3b} and 
\begin{equation*} 
\sum_{i=1}^2
\left( \varkappa_i\,\vec\nu_i,\vec\eta_i\, |[\vec x_i]_\rho| \right)
+ \sum_{i=1}^2
\left([\vec x_i]_\rho,[\vec\eta_i]_\rho \,|[\vec x_i]_\rho|^{-1}\right) = 0
\qquad \forall\ (\vec\eta_1,\vec\eta_2) \in \Vycztest\,.
\end{equation*}

In the case $C_1=1$, on the other hand, it follows from 
\eqref{eq:weak3d} and \eqref{eq:weak3e} that 
$\vec y_1(\tfrac12) - \vec y_2(\tfrac12) = 2\,\pi\,[\alpha^G_1 -
\alpha^G_2]\,\vec\ek_1$ and that $\vec{\rm m}_2 = - \vec{\rm m}_1$.
Hence we can again eliminate $\vec{\rm m}_i$, as well as $\vec\phi$,
and reduce the weak formulation to:
Find $(\vec x_1,\vec x_2)(\cdot,t) \in \xspace$, 
$(\varkappa_i,\vec y_i) \in L^2(I_i) \times \Vyczcoi$, 
$i=1,2$, with
$\vec y_1(\tfrac12,t) - \vec y_2(\tfrac12,t) = 2\,\pi\,[\alpha^G_1 -
\alpha^G_2]\,\vec\ek_1$, such that
\eqref{eq:weak3a}, \eqref{eq:weak3b} and 
\begin{equation} 
\sum_{i=1}^2\left( \varkappa_i\,\vec\nu_i,\vec\eta_i\, |[\vec x_i]_\rho| \right)
+ \sum_{i=1}^2
\left([\vec x_i]_\rho,[\vec\eta_i]_\rho \,|[\vec x_i]_\rho|^{-1}\right) = 0
 \quad \forall\ (\vec\eta_1,\vec\eta_2) \in \Vycotest\,,
\label{eq:weak3cc1} 
\end{equation}
where we have used that
$\sum_{i=1}^2 \vec{\rm m}_i\,.\,\vec\eta_i(\tfrac12) = 0$
for $(\vec\eta_1,\vec\eta_2) \in \Vycotest$, recall
{\rm (\ref{eq:Vyczcoi})}.
\end{rem}

\begin{rem} \label{rem:kappaS} 
It is also possible to consider a weak formulation based on
$\varkappa_{\mathcal{S}_i}$ as variables, similarly to 
\citet[\S3.2]{axipwf}. 
In particular, it follows from \eqref{eq:meanGaussS} and \eqref{eq:varkappa} 
that 
\begin{equation*}
\varkappa_{\mathcal{S}_i}\,\vec\nu_i = 
\frac1{|[\vec x_i]_\rho|} \left[ \frac{[\vec x_i]_\rho}{|[\vec x_i]_\rho|} 
\right]_\rho - \frac{\vec\nu_i\,.\,\vec\ek_1}{\vec
x_i\,.\,\vec\ek_1}\,\vec\nu_i \quad \text{in }\ \overline I_i\,,\ i=1,2\,,
\end{equation*}
and so the side constraints \eqref{eq:varkappaweak} are replaced by
\begin{equation} 
 \left(\vec x_i\,.\,\vec\ek_1\,
\varkappa_{\mathcal{S}_i}\,\vec\nu_i + \vec\ek_1,\vec\eta_i 
\,|[\vec x_i]_\rho| \right)
+ \left( (\vec x_i\,.\,\vec\ek_1)\,\vec\tau_i, [\vec\eta_i]_\rho \right) 
= \left[(\vec x_i\,.\,\vec\ek_1)\,
\vec{\rm m}_i\,.\,\vec\eta_i\right](\tfrac12)
\quad \forall\ \vec\eta_i \in \Vyczcoi\,,\ i = 1,2\,.
\label{eq:varkappaSweak}
\end{equation}
Hence the appropriate Lagrangian for the $L^2$--gradient flow of
\eqref{eq:Ec0c1} is
\begin{align*} 
& \mathcal{L}_{\mathcal{S}}
((\vec x_i, \varkappa_{\mathcal{S}_i}^\star, \vec{\rm m}_i, 
\vec y_{\mathcal{S}_i})_{i=1}^2,
\vec\phi) =
\pi\,\sum_{i=1}^2 \left( \alpha_i
\left[ \varkappa_{\mathcal{S}_i}^\star -\spont_i \right]^2 , 
\vec x_i\,.\,\vec\ek_1 \,|[\vec x_i]_\rho|\right) \nonumber \\ & \qquad
+ \pi\,\varsigma\,\sum_{i=1}^2 \vec x_i(\tfrac12)\,.\,\vec\ek_1
- \sum_{i=1}^2 \left(\vec x_i\,.\,\vec\ek_1\,\varkappa^\star_{\mathcal{S}_i}
\,\vec\nu_i  + \vec\ek_1,\vec y_{\mathcal{S}_i}\, |[\vec x_i]_\rho|\right)
 - \sum_{i=1}^2 \left( (\vec x_i\,.\,\vec\ek_1)\,\vec\tau_i,
(\vec y_{\mathcal{S}_i})_\rho\right) \nonumber \\ & \qquad
 + \sum_{i=1}^2 \vec{\rm m}_i\,.\left([(\vec x_i\,.\,\vec\ek_1)\,
\vec y_{\mathcal{S}_i}](\tfrac12) -
2\,\pi\,\alpha^G_i\,\vec\ek_1\right)
+ C_1\,(\vec{\rm m}_1 + \vec{\rm m}_2)\,.\,\vec\phi \,,
\end{align*}
for $(\vec x_1,\vec x_2) \in \xspace$, 
$\varkappa_{\mathcal{S}_i}^\star \in L^2(I_i)$,
$\vec m_i \in \bR^2$, $\vec y_{\mathcal{S}_i} \in \Vyczcoi$ 
and $\vec\phi \in \bR^2$.
As before, upon taking variations in 
$(\vec\chi_1,\vec\chi_2) \in \xspace$ in $\vec x$,
$\chi_i \in L^2(I_i)$ in $\varkappa_i^\star$, 
$\vec z_i \in \bR^2$ in $\vec{\rm m}_i$, 
$\vec\eta_i\in \Vyczcoi$ in $\vec y_i$ and $\vec w \in \bR^2$ in $\vec\phi$,
we obtain a weak formulation.
\end{rem}

\subsection{Conserved flows}\label{sec:LagC}
On writing (\ref{eq:weak3a}) as
\begin{align*} 
& 
2\,\pi\,\sum_{i=1}^2 \left((\vec x_i\,.\,\vec\ek_1)\,[\vec x_i]_t\,.\,\vec\nu_i,
\vec\chi_i\,.\,\vec\nu_i\,|[\vec x_i]_\rho|\right)
- \sum_{i=1}^2 \left([\vec y_i]_\rho\,.\,\vec\nu_i, 
[\vec\chi_i]_\rho \,.\,\vec\nu_i \, |[\vec x_i]_\rho|^{-1}\right) 
\nonumber \\ & \qquad
= \sum_{i=1}^2 \left( \vec f_i, \vec\chi_i\,|[\vec x_i]_\rho|\right)
\quad \forall\ \vec\chi \in \xspace \,,
\end{align*}
a weak formulation of \eqref{eq:xtbgnlambda} and \eqref{eq:xsidedtSAV} 
is given by \eqref{eq:weak3}, with \eqref{eq:weak3a} replaced by
\begin{align} \label{eq:weak4LMa}
& 
2\,\pi\,\sum_{i=1}^2 \left((\vec x_i\,.\,\vec\ek_1)\,[\vec x_i]_t\,.\,\vec\nu_i,
\vec\chi_i\,.\,\vec\nu_i\,|[\vec x_i]_\rho|\right)
- \sum_{i=1}^2 \left([\vec y_i]_\rho\,.\,\vec\nu_i, 
[\vec\chi_i]_\rho \,.\,\vec\nu_i \, |[\vec x_i]_\rho|^{-1}\right) 
\nonumber \\ & \quad
= \sum_{i=1}^2 \left( \vec f_i, \vec\chi_i\,|[\vec x_i]_\rho|\right)
- 2\,\pi\, \sum_{i=1}^2 \lambda_{A,i} 
\left[
\left( \vec\ek_1, \vec\chi_i\,|[\vec x_i]_\rho|\right) 
+ \left( (\vec x_i\,.\,\vec\ek_1)\,\vec\tau_i
, [\vec\chi_i]_\rho\right) \right]
\nonumber \\ & \qquad
- 2\,\pi\, \lambda_V \sum_{i=1}^2 \left((\vec x_i\,.\,\vec\ek_1)\,\vec\nu_i, 
\vec\chi_i\,|[\vec x_i]_\rho|\right)\quad \forall\ \vec\chi \in \xspace
\,,
\end{align}
where $(\lambda_{A,1}(t), \lambda_{A,2}(t),\lambda_V(t))^T \in \bR^3$ 
are chosen such that \eqref{eq:xsidedtSAV} holds, which is equivalent to
\begin{equation} \label{eq:xsideSAV}
A_i(\vec x(t)) = A_i(\vec x(0))\,,\ i=1,2\,,
\qquad V(\vec x(t)) = V(\vec x(0))\,.
\end{equation}
We note that for the second term on the right hand side of \eqref{eq:weak4LMa} 
we have observed that
$\left[\deldel{\vec x}\, A_i(\vec x) \right](\vec\chi) = 2\,\pi
\left( \vec\ek_1, \vec\chi_i\,|[\vec x_i]_\rho|\right) 
+ 2\,\pi \left( (\vec x_i\,.\,\vec\ek_1)\,\vec\tau_i
, [\vec\chi_i]_\rho\right)$, $i=1,2$, similarly to \eqref{eq:dAdt},
compare also with \eqref{eq:varkappaSweak}. The advantage of the formulation
\eqref{eq:weak4LMa} over one with $\varkappa_{\mathcal{S}_i}$ is that
mimicking \eqref{eq:weak4LMa} on the discrete level will allow for a stability
estimate.

\setcounter{equation}{0}
\section{Semidiscrete approximation} \label{sec:sd}

Let $\overline I_i=\bigcup_{j=1}^{J_i} I_{i,j}$, $J_i\geq3$, be
decompositions of $\overline I_i$ into intervals given by the nodes $q_{i,j}$,
$I_{i,j}=[q_{i,j-1},q_{i,j}]$. 
For simplicity, and without loss of generality,
we assume that the subintervals form equipartitionings of $\overline I_i$,
i.e.\ that 
\begin{equation} \label{eq:Jequi}
q_{i,j} = \tfrac12\,(i-1) + j\,h_i\,,\quad \text{with}\quad 
h_i = (2\,J_i)^{-1}\,, \qquad j=0,\ldots, J_i\,.
\end{equation}
The necessary finite element spaces are defined as follows:
\begin{align*}
V^h_i & = \{\chi_i \in C(\overline I_i) : \chi_i\!\mid_{I_{i,j}} 
\text{ is linear}\ \forall\ j=1,\ldots, J_i\}\,,\ i=1,2\,,\\
\quad\text{and}\quad
\Vhi & = [V^h_i]^2\,,\ i=1,2\,.
\end{align*}
We also define $\yspace^h_i = \yspace_i \cap \Vhi$, $i=1,2$, as well as
\begin{align*}
\xspace^h & = \xspace \cap \mathop{\times}_{i=1}^2 \Vhi\,,\quad
\Vhycztest = \Vycztest \cap \mathop{\times}_{i=1}^2 \Vhi\,,\quad
\Vhycotest = \Vycotest \cap \mathop{\times}_{i=1}^2 \Vhi\,,\nonumber \\
W^h_1 & = \{ \chi_1 \in V^h_1 : \chi_1(0) = 0 \}\,,\
W^h_2 = \{ \chi_2 \in V^h_2 : \chi_2(1) = 0 \}\,,\
W^h = W^h_1 \times W^h_2\,.
\end{align*}
Let $\{\chi_{i,j}\}_{j=0}^{J_i}$ denote the standard basis of $V^h_i$.
For later use, we let $\pi^h_i:C(\overline I_i)\to V^h_i$ 
be the standard interpolation operator at the nodes $\{q_{i,j}\}_{j=0}^{J_i}$,
and similarly $\vec\pi^h_i:[C(\overline I_i)]^2 \to \Vhi$.
Let the mass lumped $L^2$--inner product $(f,g)^h$,
for two piecewise continuous functions on $I_i$, with possible jumps at the 
nodes $\{q_{i,j}\}_{j=1}^{J_i-1}$, 
be defined as 
\begin{equation}
( f, g )^h = \tfrac12\,h\,\sum_{j=1}^{J_i}
\left[(f\,g)(q_{i,j}^-) + (f\,g)(q_{i,j-1}^+)\right], \label{eq:ip0}
\end{equation}
where we define
$f(q^\pm)=\underset{\delta\searrow 0}{\lim}\ f(q\pm\delta)$.
The definition (\ref{eq:ip0}) naturally extends to vector valued functions.

Let $(\vec X^h_i(t))_{t\in[0,T]}$, with 
$(\vec X^h_1(t),\vec X^h_2(t))\in \xspace^h$,
be approximations to $(\vec x_i(t))_{t\in[0,T]}$ and define
$\Gamma^h_i(t) = \vec X^h_i(t)(\overline I_i)$. 
From now on we use the shorthand notation
$\vec X^h = (\vec X^h_1,\vec X^h_2)$, and similarly for all the other finite
element functions.

\begin{ass} \label{ass:Ah1}
Let
\begin{equation} \label{eq:Xhpos}
\vec X^h_i(\rho,t) \,.\,\vec\ek_1 > 0 \quad 
\forall\ \rho \in \overline I_i \setminus \{0,1\}
\qquad \forall\ t \in [0,T]\,,\ i=1,2\,.
\end{equation}
In addition, let $\vec X^h_i(q_{i,j},t) \ne \vec X^h_i(q_{i,j+1},t)$,
$j=0,\ldots,J_i-1$, for all $t\in [0,T]$, $i=1,2$.
\end{ass}

Then, similarly to (\ref{eq:tau}), we set
\begin{equation} \label{eq:tauh}
\vec\tau^h_i = [\vec X^h_i]_s = \frac{[\vec X^h_i]_\rho}
{|[\vec X^h_i]_\rho|} 
\quad \text{and} \quad \vec\nu^h_i = -(\vec\tau^h_i)^\perp
\quad \text{in } \overline I_i\,,
\end{equation}
which is well-defined if Assumption~\ref{ass:Ah1} holds.
We note that \eqref{eq:Xhpos} implies $\vec\tau^h_i\,.\,\vec\ek_1 \not=0$
on elements touching the $x_2$--axis, and so
\begin{equation*} 
\vec\nu^h_i\,.\,\vec\ek_2 \not=0 \quad\text{ on } \partial I_i
\setminus\{\tfrac12\}\,,
\end{equation*}
compare also with \eqref{eq:bcnu} and \eqref{eq:bc}.

\begin{ass} \label{ass:Ah2}
Let {\rm Assumption~\ref{ass:Ah1}} hold and let
$\vec X^h_i(q_{i,j-1},t) \ne \vec X^h_i(q_{i,j+1},t)$,
$j=1,\ldots,J_i-1$, for all $t\in [0,T]$.
\end{ass}

For later use, we let 
$\vec\omega^h_i \in \Vhi$ be the mass-lumped 
$L^2$--projection of $\vec\nu^h_i$ onto $\Vhi$, $i=1,2$, i.e.\
\begin{equation} \label{eq:omegah}
\left(\vec\omega_i^h, \vec\varphi_i \, |[\vec X^h_i]_\rho| \right)^h 
= \left( \vec\nu^h_i, \vec\varphi_i \, |[\vec X^h_i]_\rho| \right)
= \left( \vec\nu^h_i, \vec\varphi_i \, |[\vec X^h_i]_\rho| \right)^h
\qquad \forall\ \vec\varphi_i\in\Vhi\,.
\end{equation}
Assumption~\ref{ass:Ah2} yields that $|\vec\omega^h_i| > 0$ in 
$\overline I_i$, $i=1,2$.
It follows that $\vec v_{i}^h \in \Vhi$, $i=1,2$, defined by
\begin{equation} \label{eq:vh0}
\vec v^h_{i} = \vec\pi^h_i\left[\frac{\vec\omega^h_{i}}
{|\vec\omega^h_{i}|}\right],
\end{equation}
is well-defined if Assumption~\ref{ass:Ah2} holds. 
We also define $\mat Q^h_i \in [V^h_i]^{2 \times 2}$ by
\begin{equation} \label{eq:Qh}
\mat Q^h_i(q_{i,j}) = \begin{cases}
\mat\Id & q_{i,j} = \tfrac12\,,\\
\vec v^h_{i} \otimes \vec v^h_{i} & q_{i,j} \not= \tfrac12\,,
\end{cases}
\quad j = 0,\ldots,J_i\,,\ i=1,2\,.
\end{equation}
Later on we will describe the evolution of $\Gamma^h_i(t)$ through 
$\vec\pi^h_i[\mat Q^h_i\,[\vec X^h_i]_t]$, 
for $\partial_t\,(\vec X^h_1,\vec X^h_2) \in \xspace^h$. This will
allow tangential motion for interior nodes, which together with a
discretisation of \eqref{eq:varkappaweak} will lead to equidistribution in each
phase.
But crucially, we will specify the full velocity
at the junction point, $\rho = \tfrac12$. This is because 
the tangential motion of the junction
cannot be allowed to be arbitrary, as this would affect the evolution of
the two phases, and not just the evolution of their
parameterisations $\vec X^h_i$, $i=1,2$. 
A similar strategy has been pursued by the authors
in \citet[(4.7)]{axipwf} and in \citet[(4.8)]{pwfc0c1}. 

As the discrete analogue of (\ref{eq:varkappaweak}), we let 
$(\vec X^h_1,\vec X^h_2) \in \xspace^h$,
$\kappa^h_i \in V^h_i$ and $\vec{\rm m}^h_i \in \bR^2$ be such that
\begin{equation} \label{eq:sideh}
\left( \kappa^h_i\,\vec\nu^h_i, \vec\eta_i \,|[\vec X^h_i]_\rho|\right)^h
+ \left( \vec\tau^h_i, [\vec\eta_i]_\rho \right) = 
\vec{\rm m}^h_i\,.\,\vec\eta_i(\tfrac12)
\qquad \forall\ \vec\eta_i \in \yspace^h_i\,,\ i=1,2\,,
\end{equation}
where we recall (\ref{eq:tauh}).
In the case of a $C^1$--junction, it will turn out that \eqref{eq:sideh}
can influence the tangential motion of the junction in a way that only depends
on the discretisation parameters, rather than on the actual physics of the
problem. To avoid this from happening, we need to add more flexibility for the
tangential motion of the junction. In particular, on recalling \eqref{eq:tauh}, 
we amend \eqref{eq:sideh} to
\begin{align} \label{eq:sidehtang}
&
\left( \kappa^h_i\,\vec\nu^h_i, \vec\eta_i \,|[\vec X^h_i]_\rho|\right)^h
+ C_1\,\beta^h \left(\chi_{i,(2-i)\,J_i}\,[\vec X^h_i]_\rho, \vec\eta_i
\right)^h
+ \left( \vec\tau^h_i, [\vec\eta_i]_\rho \right) = 
\vec{\rm m}^h_i\,.\,\vec\eta_i(\tfrac12) \nonumber \\ & \hspace{9cm}
\qquad \forall\ \vec\eta_i \in \yspace^h_i\,,\ i=1,2\,,
\end{align}
where $\beta^h\in\bR$ is an additional degree of freedom, and where
we observe that $\chi_{i,(2-i)\,J_i}$ is the basis function of $V^h_i$ with
$\chi_{i,(2-i)\,J_i}(\tfrac12)=1$, $i=1,2$. 
The effect of the new term in \eqref{eq:sidehtang}, analogously to
\citet[(3.49)]{pwftj}, is to allow for an additional degree of freedom
avoiding that meshes are equidistributing across the junction, compare also
\citet[Remark~3.2]{pwftj}. 

We would like to mimic on the discrete level the procedure in
Section~\ref{sec:weak}. However, a naive discretisation of (\ref{eq:Lag}) will
not give a well-defined Lagrangian, since a discrete variant of
(\ref{eq:bclimit}) will in general not hold. To overcome the arising
singularity in a discretisation of (\ref{eq:Lag}), we now introduce
the following discrete approximation of $\varkappa_{\mathcal{S}_i}$, which will
be based on $\kappa^h_i$.
In particular, on recalling (\ref{eq:bclimit}) and (\ref{eq:omegah}), 
we introduce,
given $\vec X^h_i \in \xspace^h_i$ and $\kappa^h_i \in V^h_i$, the function 
$\doctorkappa^h_i(\vec X^h_i, \kappa^h_i) \in V^h_i$ such that
\begin{equation} \label{eq:calKh}
[\doctorkappa^h_i (\vec X^h_i,\kappa^h_i)](q_{i,j}) = \begin{cases}
\kappa^h_i(q_{i,j}) - 
\dfrac{\vec\omega_{i}^h(q_{i,j})\,.\,\vec\ek_1}{\vec X^h_i(q_{i,j})\,.\,\vec\ek_1}
& q_{i,j} \in \overline I_i \setminus \{0,1\}\,, \\
2\, \kappa^h_i(q_{i,j}) & q_{i,j} \in \{0,1\}\,,
\end{cases}
\end{equation}
compare with \citet[(4.11)]{axipwf}. 
This allows us to define the discrete analogue of the energy 
(\ref{eq:Ec0c1}) as
\begin{align} 
\widehat E^h(t) &= 
\pi\,\sum_{i=1}^2
\left(\alpha_i
\left[ \doctorkappa^h_i(\vec X^h_i, \kappa^h_i) - \spont_i \right]^2,  
\vec X^h_i\,.\,\vec\ek_1 \,|[\vec X^h_i]_\rho|\right)^h \nonumber \\
& \qquad
- 2\,\pi\,\sum_{i=1}^2 \alpha^G_i \,\vec{\rm m}^h_i\,.\,\vec\ek_1
+ \pi\,\varsigma\sum_{i=1}^2 \vec X^h_i(\tfrac12)\,.\,\vec\ek_1\,.
\label{eq:Eh}
\end{align}

\begin{rem} \label{rem:kappahpartial0}
We observe that the energy $\widehat E^h(t)$ does not depend
on the values $\kappa^h_1(0,t)$ and $\kappa^h_2(1,t)$.
We will thus fix these values to be zero from now on,
by seeking $\kappa^h_i \in W^h_i$, $i=1,2$. 
A welcome side effect of this procedure is that choosing
$\vec\eta_1 = \chi_{1,0}\,\vec\ek_2$ and
$\vec\eta_2 = \chi_{2,J_2}\,\vec\ek_2$ in \eqref{eq:sidehtang} yields that
\begin{equation} \label{eq:Xhq1q0}
(\vec X^h_1 (q_{1,1}) - \vec X^h_1 (q_{1,0}) ) \,.\,\vec\ek_2 = 
(\vec X^h_2 (q_{2,J_2}) - \vec X^h_2 (q_{2,J_2-1}) ) \,.\,\vec\ek_2 = 
0\,,
\end{equation}
which can be viewed as exact discretisations of the $90^\circ$ 
contact angle conditions \eqref{eq:bcbc}. 
\end{rem}

Similarly to (\ref{eq:Lag}), we define the discrete Lagrangian
\begin{align*} 
& \mathcal{L}^h((\vec X^h_i, \kappa^h_i, \vec{\rm m}^h_i, \vec Y^h_i)_{i=1}^2,
\beta^h,\vec\phi^h) \nonumber \\ & \quad =
\pi \sum_{i=1}^2 \left(\alpha_i 
\left[ \doctorkappa^h_i(\vec X^h_i, \kappa^h_i) - \spont_i \right]^2 , 
\vec X^h_i\,.\,\vec\ek_1 \,|[\vec X^h_i]_\rho|\right)^h 
+ \pi\,\varsigma\,\sum_{i=1}^2 \vec X^h_i(\tfrac12)\,.\,\vec\ek_1
\nonumber \\ & \qquad 
- \sum_{i=1}^2 \left( \kappa^h_i\,\vec\nu^h_i, \vec Y^h_i 
\,|[\vec X^h_i]_\rho|\right)^h
- C_1\,\beta^h \sum_{i=1}^2 
\left(\chi_{i,(2-i)\,J_i}\,[\vec X^h_i]_\rho, \vec Y^h_i \right)^h 
- \sum_{i=1}^2 \left( \vec\tau^h_i, [\vec Y^h_i]_\rho \right)
\nonumber \\ & \qquad 
 + \sum_{i=1}^2 \vec{\rm m}_i^h\,.\left(\vec Y^h_i(\tfrac12) -
2\,\pi\,\alpha^G_i\,\vec\ek_1\right)
+ C_1\,(\vec{\rm m}^h_1 + \vec{\rm m}^h_2)\,.\,\vec\phi^h \,,
\end{align*}
for the minimisation of the energy \eqref{eq:Eh} subject to the side 
constraint \eqref{eq:sidehtang} and a discrete variant of \eqref{eq:mC0C1}, 
where $(\vec X^h_1,\vec X^h_2) \in \xspace^h$, 
$\kappa^h_i \in \Whi$, $C_1\,\beta^h\in\bR$,
$\vec{\rm m}^h_i \in \bR^2$, $\vec Y^h_i \in \yspace^h_i$ and
$\vec\phi^h \in \bR^2$.

Taking variations $\vec\eta_i \in \yspace^h_i$ in $\vec Y^h_i$, and setting
$\left[\deldel{\vec Y^h_i}\, \mathcal{L}^h\right](\vec\eta_i) = 0$
we obtain (\ref{eq:sidehtang}).
Taking variations $\chi_i \in \Whi$ in $\kappa^h_i$ and setting
$\left[\deldel{\kappa^h_i}\, \mathcal{L}^h\right](\chi_i) = 0$ we obtain 
\begin{align} 
&2\,\pi\left( \vec X^h_i\,.\,\vec\ek_1 \,\left(\alpha_i\,[\doctorkappa^h_i(
\vec X^h_i,\kappa^h_i) - \spont_i ] \right),\chi_i
\,|[\vec X^h_i]_\rho|\right)^h
- \left( \vec Y^h_i, \chi_i\,\vec\nu^h_i \,|[\vec X^h_i]_\rho|\right)^h = 0
\nonumber \\ & \hspace{9cm}
\qquad \forall\ \chi_i \in \Whi\,, \ i = 1,2 \label{eq:kappahYh}
\end{align}
where we have recalled (\ref{eq:calKh}). 
Taking variations in $\vec{\rm m}^h_i \in \bR^2$, $i=1,2$, 
and setting them to zero, 
yields, similarly to (\ref{eq:yalphaG}), that
\begin{equation} \label{eq:mhy}
\vec Y^h_i(\tfrac12) = 2\,\pi\,\alpha^G_i\,\vec\ek_1 
- C_1\,\vec\phi^h\,,\ i = 1,2\,.
\end{equation}
Similarly, taking variations in $\vec\phi \in \bR^2$, 
and setting them to zero, yields
\begin{equation} \label{eq:mhC0C1}
C_1\,(\vec{\rm m}^h_1 + \vec{\rm m}^h_2) = \vec 0\,.
\end{equation}
Taking variations in $\beta^h\in\bR$, and setting them to zero, implies
\begin{align} \label{eq:betavariation}
C_1\,\sum_{i=1}^2 \left(\chi_{i,(2-i)\,J_i}\,[\vec X^h_i]_\rho, \vec Y^h_i
\right)^h & = 0 \nonumber \\ 
\quad\iff\quad
C_1\,\left[ 
(\vec X^h_1(q_{1,J_1}) - \vec X^h_1(q_{1,J_1-1}))\,.\,
\vec Y^h_1(\tfrac12) +
(\vec X^h_2(q_{2,1}) - \vec X^h_2(q_{2,0}))\,.\,
\vec Y^h_2(\tfrac12) \right] & = 0 \,.
\end{align}
 
Taking variations $\vec\chi = (\vec\chi_1,\vec\chi_2) \in \xspaceh$ in 
$\vec X^h = (\vec X^h_1,\vec X^h_2)$, and 
setting \linebreak 
$2\,\pi\,\sum_{i=1}^2 ((\vec X^h_i\,.\,\vec\ek_1)\,\mat Q^h_i\,
[\vec X^h_i]_t,\vec\chi_i\,|[\vec X^h_i]_\rho|)^h = 
- \left[\deldel{\vec X^h}\, \mathcal{L}^h\right](\vec\chi)$
we obtain  
\begin{align}
& 2\,\pi\,\sum_{i=1}^2 
\left((\vec X^h_i\,.\,\vec\ek_1)\,\mat Q^h_i\,[\vec X^h_i]_t,
\vec\chi_i\,|[\vec X^h_i]_\rho|\right)^h \nonumber \\ &\qquad = 
- \pi \sum_{i=1}^2 \left( \alpha_i
\left[ \doctorkappa^h_i(\vec X^h_i,\kappa^h_i) - \spont_i \right]^2,
\left[\deldel{\vec X^h}\,(\vec X^h_i\,.\,\vec\ek_1)\,
|[\vec X^h_i]_\rho|\right] (\vec\chi) \right)^h 
\nonumber \\ & \qquad \qquad
- 2\,\pi\,\sum_{i=1}^2 \alpha_i \left( 
\left[ \doctorkappa^h_i(\vec X^h_i,\kappa^h_i) - \spont_i \right], 
\left[\deldel{\vec X^h}\,\doctorkappa^h_i(\vec X^h_i,\kappa^h_i)\right]
(\vec\chi)\,(\vec X^h_i\,.\,\vec\ek_1)\,|[\vec X^h_i]_\rho| \right)^h 
\nonumber \\ & \qquad \qquad 
+ \sum_{i=1}^2 \left( \kappa^h_i\, \vec Y^h_i ,
\left[\deldel{\vec X^h}\,\vec\nu^h_i\,|[\vec X^h_i]_\rho|\right]
(\vec\chi)\right)^h
+ C_1\,\beta^h \sum_{i=1}^2
\left(\chi_{i,(2-i)\,J_i}\,[\vec\chi_i]_\rho, \vec Y^h_i \right)^h
\nonumber \\ & \qquad \qquad 
+ \sum_{i=1}^2 \left([\vec Y^h_i]_\rho , 
\left[\deldel{\vec X^h}\,\vec\tau^h_i\right](\vec\chi)\right) 
- \pi\,\varsigma\, \sum_{i=1}^2 \vec\chi_i(\tfrac12)\,.\,\vec\ek_1
\qquad \forall\ \vec\chi \in \xspaceh\,.
\label{eq:Pxth}
\end{align}
Choosing $\vec\chi = \vec X^h_t$ in (\ref{eq:Pxth}) yields
\begin{align}
& 2\,\pi\,\sum_{i=1}^2 
\left(\vec X^h_i\,.\,\vec\ek_1\,|\mat Q^h_i\,[\vec X^h_i]_t|^2,
|[\vec X^h_i]_\rho|\right)^h 
\nonumber \\ & \qquad = 
- \pi \,\sum_{i=1}^2 \left( 
\alpha_i \left[ \doctorkappa^h_i(\vec X^h_i,\kappa^h_i) - \spont_i \right]^2,
\left[(\vec X^h_i\,.\,\vec\ek_1)\,|[\vec X^h_i]_\rho|\right]_t \right)^h 
\nonumber \\ & \qquad \qquad 
- 2\,\pi\,\sum_{i=1}^2 \alpha_i \left( 
 \doctorkappa^h_i(\vec X^h_i,\kappa^h_i) - \spont_i , 
\left[\dfrac{\vec\omega_{i}^h\,.\,\vec\ek_1}{\vec X^h_i\,.\,\vec\ek_1}
\right]_t (\doctorZ^h_i - 2)
\,(\vec X^h_i\,.\,\vec\ek_1)\,|[\vec X^h_i]_\rho| \right)^h 
\nonumber \\ & \qquad \qquad 
+ \sum_{i=1}^2 
\left( \kappa^h_i\, \vec Y^h_i,
\left[\vec\nu^h_i\,|[\vec X^h_i]_\rho|\right]_t \right)^h
+ C_1\,\beta^h \sum_{i=1}^2 
\left(\chi_{i,(2-i)\,J_i}\,([\vec X^h_i]_\rho)_t, \vec Y^h_i \right)^h
\nonumber \\ & \qquad \qquad 
+ \sum_{i=1}^2 \left([\vec Y^h_i]_\rho , [\vec\tau^h_i]_t \right)  
- \pi\,\varsigma\, \sum_{i=1}^2 [\vec X^h_i]_t(\tfrac12)\,.\,\vec\ek_1 \,,
\label{eq:Pxtxth}
\end{align}
where we have defined $\doctorZ^h_i \in V^h_i$ such that
\begin{equation*} 
\doctorZ^h_i (q_{i,j}) = \begin{cases}
1 & q_{i,j} \in \overline I_i \setminus \{0,1\}\,, \\
2 & q_{i,j} \in \{0,1\}\,.
\end{cases}
\end{equation*}

Differentiating (\ref{eq:sidehtang}) with respect to $t$, and then choosing
$\vec\eta_i = \vec Y^h_i \in \yspace^h_i$ and noting \eqref{eq:betavariation}, 
yields that 
\begin{align}
&  \left( [\kappa^h_i]_t, 
\vec Y^h_i\,.\,\vec\nu^h_i\,|[\vec X^h_i]_\rho|\right)^h
+ \left( \kappa^h_i\,\vec Y^h_i, \left[\vec\nu^h_i\,|[\vec X^h_i]_\rho|
\right]_t\right)^h
+ C_1\,\beta^h 
\left(\chi_{i,(2-i)\,J_i}\,([\vec X^h_i]_\rho)_t, \vec Y^h_i \right)^h
\nonumber \\ & \qquad
+ \left([\vec\tau^h_i]_t,[\vec Y^h_i]_\rho\right)
= [\vec{\rm m}^h_i]_t\,.\,\vec Y^h_i(\tfrac12)\,,\ i = 1,2\,.
\label{eq:dtyyh}
\end{align}
It follows from (\ref{eq:dtyyh}), (\ref{eq:mhy}) and (\ref{eq:kappahYh}) 
with $\chi_i = [\kappa^h_i]_t \in W^h_i$ that
\begin{align}
&\left( \kappa^h_i\,\vec Y^h_i, \left[\vec\nu^h_i\,|[\vec X^h_i]_\rho|
\right]_t\right)^h
+ C_1\,\beta^h 
\left(\chi_{i,(2-i)\,J_i}\,([\vec X^h_i]_\rho)_t, \vec Y^h_i \right)^h
+ \left([\vec\tau^h_i]_t,[\vec Y^h_i]_\rho\right)
\nonumber \\ & \qquad 
= - \left( [\kappa^h_i]_t, 
\vec Y^h_i\,.\,\vec\nu^h_i\,|[\vec X^h_i]_\rho|\right)^h
+ 2\,\pi\, \alpha^G_i [\vec{\rm m}^h_i]_t\,.\,\vec\ek_1 
- C_1\,[\vec{\rm m}^h_i]_t\,.\,\vec\phi^h
\nonumber \\ & \qquad 
= -2\,\pi\,\alpha_i \left( \vec X^h_i\,.\,\vec\ek_1\, [\doctorkappa^h_i(
\vec X^h_i,\kappa^h_i) - \spont_i ], [\kappa^h_i]_t 
\,|[\vec X^h_i]_\rho| \right)^h 
\nonumber \\
& \qquad \qquad + 
2\,\pi\, \alpha^G_i [\vec{\rm m}^h_i]_t\,.\,\vec\ek_1 
- C_1\,[\vec{\rm m}^h_i]_t\,.\,\vec\phi^h\,,\ i=1,2\,.
\label{eq:mykappah}
\end{align}
Combining (\ref{eq:Pxtxth}) and (\ref{eq:mykappah}) 
yields, on recalling \eqref{eq:mhC0C1} and (\ref{eq:Eh}), that
\begin{align}
& 2\,\pi\,\sum_{i=1}^2 
\left(\vec X^h_i\,.\,\vec\ek_1\,|\mat Q^h_i\,[\vec X^h_i]_t|^2,
|[\vec X^h_i]_\rho|\right)^h 
\nonumber \\ & \qquad 
= - \pi\,\sum_{i=1}^2 \left( \alpha_i
\left[ \doctorkappa^h_i(\vec X^h_i,\kappa^h_i) - \spont_i \right]^2,
\left[(\vec X^h_i\,.\,\vec\ek_1)\,|[\vec X^h_i]_\rho|\right]_t \right)^h 
\nonumber \\ & \qquad \qquad 
- 2\,\pi\,\sum_{i=1}^2 \alpha_i \left( 
 \doctorkappa^h_i(\vec X^h_i,\kappa^h_i) - \spont_i , 
\left[\dfrac{\vec\omega_{i}^h\,.\,\vec\ek_1}{\vec X^h_i\,.\,\vec\ek_1}
\right]_t (\doctorZ^h_i - 2)
\,(\vec X^h_i\,.\,\vec\ek_1)\,|[\vec X^h_i]_\rho| \right)^h 
\nonumber \\ & \qquad \qquad 
-2\,\pi\,\sum_{i=1}^2 \alpha_i
\left( \vec X^h_i\,.\,\vec\ek_1\,[\doctorkappa^h_i(
\vec X^h_i,\kappa^h_i) - \spont_i ] , 
[\kappa^h_i]_t \,|[\vec X^h_i]_\rho| \right)^h
\nonumber \\ & \qquad \qquad  
+ 2\,\pi\,\sum_{i=1}^2 \alpha^G_i\,[\vec{\rm m}^h_i]_t\,.\,\vec\ek_1 
- \pi\,\varsigma\,\sum_{i=1}^2 [\vec X^h_i]_t(\tfrac12)\,.\,\vec\ek_1 
\nonumber \\ & \qquad 
= - \ddt \,\widehat E^h(t)\, .
\label{eq:Pxtstabh}
\end{align}

We now return to (\ref{eq:Pxth}) which, similarly to \citet[(4.22)]{axipwf}, 
can be rewritten as
\begin{align}
& 2\,\pi\,\sum_{i=1}^2 
\left((\vec X^h_i\,.\,\vec\ek_1)\,\mat Q^h_i\,[\vec X^h_i]_t,
\vec\chi_i\,|[\vec X^h_i]_\rho|\right)^h = \nonumber \\ & \quad 
- \pi\,\sum_{i=1}^2 \left( \alpha_i
\left[ \doctorkappa^h_i(\vec X^h_i,\kappa^h_i) - \spont_i \right]^2 ,
\vec\chi_i\,.\,\vec\ek_1\,|[\vec X^h_i]_\rho| 
+ (\vec X^h_i\,.\,\vec\ek_1)\,\vec\tau^h_i\,.\,[\vec\chi_i]_\rho \right)^h
\nonumber \\ & \quad 
+ 2\,\pi\,\sum_{i=1}^2\alpha_i 
\left( \left[\doctorkappa^h_i(\vec X^h_i,\kappa^h_i) - \spont_i\right] 
 (\doctorZ^h_i - 2), 
\frac{\vec\omega_{i}^h\,.\,\vec\ek_1}{\vec X^h_i\,.\,\vec\ek_1}\,
\vec\chi_i\,.\,\vec\ek_1\,|[\vec X^h_i]_\rho| \right)^h 
\nonumber \\ & \quad 
+ 2\,\pi\,\,\sum_{i=1}^2 \alpha_i
\left( \left[\doctorkappa^h_i(\vec X^h_i,\kappa^h_i) - \spont_i\right]
(\doctorZ^h_i - 2)\,\vec\ek_1, 
(\vec\nu^h_i\,.\,[\vec\chi_i]_\rho) \,\vec\tau^h_i + 
(\vec\tau^h_i\,.\,[\vec\chi_i]_\rho) \,(\vec\omega_{i}^h-\vec\nu^h_i) 
\right)^h 
\nonumber \\ & \quad 
+ C_1\,\beta^h \sum_{i=1}^2
\left(\chi_{i,(2-i)\,J_i}\,[\vec\chi_i]_\rho, \vec Y^h_i \right)^h
- \sum_{i=1}^2 \left( \kappa^h_i\,\vec Y^h_i,
[\vec\chi_i]_\rho^\perp \right)^h
\nonumber \\ & \quad 
+ \sum_{i=1}^2
\left([\vec Y^h_i]_\rho\,.\,\vec\nu^h_i , 
[\vec\chi_i]_\rho\,.\,\vec\nu^h_i\,|[\vec X^h_i]_\rho|^{-1}\right) 
- \pi\,\varsigma\,\sum_{i=1}^2 \vec\chi_i(\tfrac12)\,.\,\vec\ek_1
\qquad \forall\ \vec\chi \in \xspaceh\,.
\label{eq:Pxt2h}
\end{align}
Combining (\ref{eq:Pxt2h}), (\ref{eq:kappahYh}), (\ref{eq:sidehtang}),
(\ref{eq:mhy}) and (\ref{eq:mhC0C1}), our semidiscrete approximation is given, 
on noting $\vec a \,.\,\vec b^\perp= -\vec a^\perp .\,\vec b$
and (\ref{eq:tauh}), as follows. 

\noindent
$(\BGNpwf^h)^{h}$
Let $\vec X^h(\cdot,0) \in \xspace^h$ be given.
Then, for $t \in (0,T]$ find
$\vec X^h(\cdot,t) \in \xspace^h$, 
$(\kappa^h_i(\cdot,t),\vec Y^h_i(\cdot,t),$ $\vec{\rm m}^h_i(t))
\in W^h_i \times \yspace^h_i \times \bR^2$, $i=1,2$, 
$C_1\,\beta^h(t) \in \bR$ and $C_1\,\vec\phi^h(t) \in \bR^2$ such that
\begin{subequations} \label{eq:sd}
\begin{align}
& 2\,\pi\,\sum_{i=1}^2\left((\vec X^h_i\,.\,\vec\ek_1)\,\mat Q^h_i\,
[\vec X^h_i]_t,
\vec\chi_i\,|[\vec X^h_i]_\rho|\right)^h 
- \sum_{i=1}^2 
\left([\vec Y^h_i]_\rho\,.\,\vec\nu^h_i , [\vec\chi_i]_\rho\,.\,\vec\nu^h_i\,
|[\vec X^h_i]_\rho|^{-1} \right) = 
\nonumber \\ & \quad 
- \pi\,\sum_{i=1}^2 \left( 
\alpha_i\,\left[ \doctorkappa^h_i(\vec X^h_i,\kappa^h_i) - \spont_i \right]^2,
\vec\chi_i\,.\,\vec\ek_1\,|[\vec X^h_i]_\rho| 
+ (\vec X^h_i\,.\,\vec\ek_1)\,\vec\tau^h_i\,.\,[\vec\chi_i]_\rho \right)^h
\nonumber \\ & \quad 
+ 2\,\pi\,\sum_{i=1}^2 \alpha_i
\left( \left[\doctorkappa^h_i(\vec X^h_i,\kappa^h_i) - \spont_i\right] 
 (\doctorZ^h_i - 2), 
\frac{\vec\omega_{i}^h\,.\,\vec\ek_1}{\vec X^h\,.\,\vec\ek_1}\,
\vec\chi_i\,.\,\vec\ek_1 \,|[\vec X^h_i]_\rho| \right)^h 
\nonumber \\ & \quad 
+ 2\,\pi\,\sum_{i=1}^2 \alpha_i 
\left( \left[\doctorkappa^h_i(\vec X^h_i,\kappa^h_i) - \spont_i\right]
(\doctorZ^h_i - 2)\,\vec\ek_1, 
(\vec\nu^h_i\,.\,[\vec\chi_i]_\rho) \,\vec\tau^h_i 
+ (\vec\tau^h_i\,.\,[\vec\chi_i]_\rho) \,
(\vec\omega_{i}^h-\vec\nu^h_i) \right)^h 
\nonumber \\ & \quad 
+ \sum_{i=1}^2 \left( \kappa^h_i\,(\vec Y^h_i)^\perp,
[\vec\chi_i]_\rho \right)^h 
+ C_1\,\beta^h \sum_{i=1}^2 
\left(\chi_{i,(2-i)\,J_i}\,[\vec\chi_i]_\rho, \vec Y^h_i \right)^h
- \pi\,\varsigma\,\sum_{i=1}^2 \vec\chi_i(\tfrac12)\,.\,\vec\ek_1
\nonumber \\ & \hspace{11cm}
\qquad \forall\ \vec\chi \in \xspaceh\,, \label{eq:sda} \\
& 
2\,\pi\sum_{i=1}^2
\left( \vec X^h_i\,.\,\vec\ek_1 \left( \alpha_i\,[\doctorkappa^h_i(
\vec X^h_i,\kappa^h_i) - \spont_i ]\right)
,\chi_i \,|[\vec X^h_i]_\rho|\right)^h
- \sum_{i=1}^2
\left( \vec Y^h_i, \chi_i\,\vec\nu^h_i \,|[\vec X^h_i]_\rho|\right)^h = 0
\nonumber \\
& \hspace{11cm} \qquad \forall\ \chi \in W^h\,,\label{eq:sdb} \\
&\left( \kappa^h_i\,\vec\nu^h_i, \vec\eta_i \,|[\vec X^h_i]_\rho|\right)^h
+ C_1\,\beta^h \left(\chi_{i,(2-i)\,J_i}\,[\vec X^h_i]_\rho, \vec\eta_i
\right)^h
+ \left( [\vec X^h_i]_\rho, [\vec\eta_i]_\rho\,|[\vec X^h_i]_\rho|^{-1} \right)
 = \vec{\rm m}^h_i\,.\,\vec\eta_i(\tfrac12)
\nonumber \\ & \hspace{9cm}
 \qquad \forall\ \vec\eta_i \in \yspace^h_i\,,\ i=1,2\,, \label{eq:sdc} \\
&- 2\,\pi\,\alpha^G_i\,\vec\ek_1 + \vec Y^h_i(\tfrac12) + C_1\,\vec\phi^h 
= \vec 0 \,,\ i = 1,2\,,\label{eq:sdd} \\
& 
C_1\,\sum_{i=1}^2 \left(\chi_{i,(2-i)\,J_i}\,[\vec X^h_i]_\rho, \vec Y^h_i
\right)^h = 0\,, \label{eq:sde} \\ 
& C_1\,(\vec{\rm m}_1^h + \vec{\rm m}_2^h) = \vec 0\,. \label{eq:sdf} 
\end{align}
\end{subequations}

\begin{thm} \label{thm:stab}
Let {\rm Assumption~\ref{ass:Ah2}} be satisfied and 
let $(\vec X^h(t),\kappa^h(t),\vec Y^h(t),\vec{\rm m}^h(t),$ \linebreak
$C_1\,\beta^h(t),\vec\phi^h(t))_{t\in (0,T]}$ 
be a solution to \mbox{\rm (\ref{eq:sd})}. 
Then the solution satisfies the stability bound
\[
\ddt \,\widehat E^h(t)
+ 2\,\pi\,\sum_{i=1}^2 \left(\vec X^h_i\,.\,\vec\ek_1\,
|\mat Q^h_i\,[\vec X^h_i]_t|^2, |[\vec X^h_i]_\rho|\right)^h=0\,. 
\]
\end{thm}
\begin{proof}
The desired result follows as (\ref{eq:sd}) is just a rewrite of 
(\ref{eq:Pxth}), (\ref{eq:kappahYh}), (\ref{eq:sidehtang}), \eqref{eq:mhy} 
and (\ref{eq:mhC0C1}), and then noting (\ref{eq:Pxtxth})--(\ref{eq:Pxtstabh}). 
\end{proof}

\begin{rem} \label{rem:equid}
We note that on choosing 
$\vec\eta_i = \chi_{i,j}\,[\vec\omega_{i}^h(q_{i,j})]^\perp$, for 
$j \in \{1,\ldots,J_i-1\}$ so that $\vec\eta_i \in \yspace^h_i$
with $\vec\eta_i(\tfrac12)=\vec 0$, 
in {\rm (\ref{eq:sdc})}, $i=1,2$, we obtain that
\begin{align} \label{eq:equidphases}
|\vec X^h_i(q_{i,j}) - \vec X^h_i(q_{i,j-1})| & = 
|\vec X^h_i(q_{i,j+1}) - \vec X^h_i(q_{i,j})| \nonumber \\
\ \text{or} \quad
\vec X^h_i(q_{i,j}) - \vec X^h_i(q_{i,j-1}) & \parallel
\vec X^h_i(q_{i,j+1}) - \vec X^h_i(q_{i,j})
\,,\quad j=1,\ldots,J_i-1\,,\ i=1,2\,. 
\end{align}
See {\rm \citet[Remark 2.4]{triplej}} 
for details. Hence the curves $\Gamma^h_i(t)$, $i=1,2$,
will each be equidistributed where-ever two neighbouring 
elements are not parallel.

We now highlight why the term involving $\beta^h$ is crucial in
\eqref{eq:sidehtang} in order to avoid undesirable tangential motion of the
junction when $C_1=1$. To this end, let us assume for now that $\beta^h=0$. 
Then we can choose
$\vec\eta_1 = \chi_{1,J_1}\,\vec\mu^h_2$ and
$\vec\eta_2 = \chi_{2,0}\,\vec\mu^h_2$
in \eqref{eq:sdc}, where we note that $\vec\mu^h_2 = -
[\vec\omega_2^h(q_{2,0})]^\perp$ is the true conormal to $\Gamma^h_2(t)$ at
$\vec X^h_2(\tfrac12,t)$.
On noting
$\vec\eta_1(\tfrac12) =  \vec\eta_2(\tfrac12) = \vec\mu^h_2$
and \eqref{eq:sdf} it follows that
\begin{align}
& \left( \kappa^h_1\,\vec\nu^h_1, \chi_{1,J_1}\,\vec\mu^h_2
\,|[\vec X^h_1]_\rho|\right)^h
+ \left( \vec\tau^h_1, [\chi_{1,J_1}]_\rho\, \vec\mu^h_2 \right)
+ \left( \vec\tau^h_2, [\chi_{2,0}]_\rho\, \vec\mu^h_2 \right)
= 0 \nonumber \\ \Rightarrow &
\left( \kappa^h_1, \chi_{1,J_1} \,|[\vec X^h_1]_\rho|\right)^h 
\vec\nu^h_1(\tfrac12)\,.\,\vec\mu^h_2
+ \left( 1 , [\chi_{1,J_1}]_\rho \right) \vec\tau^h_1(\tfrac12)\,.\,\vec\mu^h_2
- \left( 1, [\chi_{2,0}]_\rho \right) = 0\,.
\label{eq:C1equid}
\end{align}
Now, similarly to \citet[Remark~3.2]{pwftj}, it can be argued that
\eqref{eq:C1equid}, enforces some tangential motion of the junction that is
determined by the discretisation. In particular, in the case that the two
elements meeting at the junction are parallel, which implies that
$\vec\nu^h_1(\tfrac12)\,.\,\vec\mu^h_2 = 0$ and
$\vec\tau^h_1(\tfrac12)\,.\,\vec\mu^h_2 = -1$, then \eqref{eq:C1equid} 
enforces
\begin{equation} \label{eq:C1equid2}
\left( 1 , [\chi_{1,J_1}]_\rho \right)
+ \left( 1, [\chi_{2,0}]_\rho \right) = 0\,,
\end{equation}
which means that the two elements next to the $C^1$--junction will have 
the same length. Together with \eqref{eq:equidphases} this would imply a global
equidistribution property, across the two phases. Even though in general
\eqref{eq:C1equid2} will not hold exactly, in practice some undesirable
tangential motion can be expected, and is observed in our numerical
experiments. It is for this reason that we only consider the scheme 
\eqref{eq:sd} as stated.
\end{rem}

\begin{rem} \label{rem:C0h}
In accordance with {\rm Remark~\ref{rem:C0}}, it is possible to eliminate the
discrete conormal vectors $\vec{\rm m}^h_i$, $i=1,2$, 
as well as $\vec\phi^h$, from \eqref{eq:sd}. In particular, 
$(\vec X^h(t),\kappa^h(t),$ $\vec Y^h(t), C_1\,\beta^h(t))_{t\in (0,T]}$ form 
part of a solution to \eqref{eq:sd} if and only if
$\vec X^h(t) \in \xspace^h$, $\kappa^h(t) \in W^h$,
$\vec Y^h(t) \in \yspace^h$ and $C_1\,\beta^h(t) \in \bR$ with
\[
\begin{cases}
\vec Y^h_i(\tfrac12,t) = 2\,\pi\,\alpha^G_i\,\vec\ek_1\,,\ i=1,2
& C_1 = 0\,,\\
\vec Y^h_1(\tfrac12,t) - \vec Y^h_2(\tfrac12,t) = 2\,\pi\,[\alpha^G_1 -
\alpha^G_2]\,\vec\ek_1 \quad \text{and}\ \eqref{eq:betavariation} 
& C_1 = 1\,,
\end{cases}
\]
are such that \eqref{eq:sda}, \eqref{eq:sdb} and
\begin{align*} 
& \sum_{i=1}^2
\left( \kappa^h_i\,\vec\nu^h_i, \vec\eta_i \,|[\vec X^h_i]_\rho|\right)^h
+ C_1\,\beta^h\,\sum_{i=1}^2 
\left(\chi_{i,(2-i)\,J_i}\,[\vec X^h_i]_\rho, \vec\eta_i
\right)^h
+ \sum_{i=1}^2
\left( [\vec X^h_i]_\rho, [\vec\eta_i]_\rho\,|[\vec X^h_i]_\rho|^{-1} 
\right) \nonumber \\ & \hspace{8cm}
= 0
\quad \forall\ \vec\eta \in 
\begin{cases} 
\Vhycztest & C_1 = 0\,,\\
\Vhycotest & C_1 = 1\,,
\end{cases} 
\end{align*}
hold.
\end{rem}

\subsection{Conserved flows} \label{sec:new51}

We rewrite \eqref{eq:sda} as
\begin{align*} 
& 2\,\pi\,\sum_{i=1}^2\left((\vec X^h_i\,.\,\vec\ek_1)\,\mat Q^h_i\,
[\vec X^h_i]_t,
\vec\chi_i\,|[\vec X^h_i]_\rho|\right)^h 
-\sum_{i=1}^2
 \left([\vec Y^h_i]_\rho\,.\,\vec\nu^h_i , [\vec\chi_i]_\rho\,.\,\vec\nu^h_i\,
|[\vec X^h_i]_\rho|^{-1} \right) 
\nonumber \\ & \qquad 
= \sum_{i=1}^2 \left( \vec f^h_i, \vec\chi_i \,|[\vec X^h_i]_\rho|\right)^h 
\quad  \forall\ \vec\chi \in \xspaceh\,.
\end{align*}
Then the natural generalisation of $(\BGNpwf^h)^{h}$, \eqref{eq:sd}, that
approximates the weak formulation \eqref{eq:weak4LMa}, 
\eqref{eq:weak3b}--\eqref{eq:weak3e} and \eqref{eq:xsideSAV} is given by 
\eqref{eq:sd}, with \eqref{eq:sda} replaced by
\begin{align}
& 2\,\pi\,\sum_{i=1}^2\left((\vec X^h_i\,.\,\vec\ek_1)\,\mat Q^h_i\,
[\vec X^h_i]_t,
\vec\chi_i\,|[\vec X^h_i]_\rho|\right)^h 
- \sum_{i=1}^2
\left([\vec Y^h_i]_\rho\,.\,\vec\nu^h_i , [\vec\chi_i]_\rho\,.\,\vec\nu^h_i\,
|[\vec X^h_i]_\rho|^{-1} \right) 
\nonumber \\ & \qquad 
= \sum_{i=1}^2 \left( \vec f^h_i, \vec\chi_i \,|[\vec X^h_i]_\rho|\right)^h 
- 2\,\pi\,\sum_{i=1}^2 \lambda^h_{A,i} 
\left[
\left( \vec\ek_1, \vec\chi_i\,|[\vec X^h_i]_\rho|\right) 
+ \left( (\vec X^h_i\,.\,\vec\ek_1)\,\vec\tau^h_i
, [\vec\chi_i]_\rho\right) \right]
 \nonumber \\ & \hspace{5cm}
- 2\,\pi\,\lambda_V^h \sum_{i=1}^2 \left((\vec X^h_i\,.\,\vec\ek_1)\, 
\vec\nu_i^h, \vec\chi_i \,|[\vec X^h_i]_\rho|\right)
 \qquad \forall\ \vec\chi \in \xspaceh\,, \label{eq:Q0hXtLMa}
\end{align}
where $(\lambda_{A,1}^h(t), \lambda_{A,2}^h(t), \lambda_V^h(t))^T \in \bR^3$ 
are such that
\begin{equation} \label{eq:sdwfb}
A_i(\vec X^h(t)) = A_i(\vec X^h(0))\,,\ i=1,2\,,\quad\text{and}\quad
V(\vec X^h(t)) = V(\vec X^h(0))\,.
\end{equation}
Here, on recalling \eqref{eq:Area} and \eqref{eq:Volume}, we note that
$A_i(\vec X^h(t))$ denotes the surface area of $\mathcal{S}^h_i(t)$,
where, similarly to (\ref{eq:SGamma}), we set
\begin{equation*} 
\mathcal{S}^h_i(t) = \bigcup_{\rho \in \overline{I_i}} 
\Pi_2^3(\vec X^h_i(\rho,t))\,,\ i=1,2\,.
\end{equation*}
Moreover, $V(\vec X^h(t))$ is the volume of the domain $\Omega^h(t)$ with
$\partial\Omega^h(t) = \cup_{i=1}^2 \mathcal{S}^h_i(t)$. We remark that
\begin{equation} \label{eq:Areah}
A_i(\vec Z^h) = 2\,\pi\left(\vec Z^h_i\,.\,\vec\ek_1 ,
|[\vec Z^h_i]_\rho|\right) \quad \vec Z^h \in \xspace^h
\end{equation}
and
\begin{equation}  \label{eq:Volumeh}
V(\vec Z^h) = -\pi \sum_{i=1}^2 \left( (\vec Z^h_i\,.\,\vec\ek_1)^2, 
[[\vec Z^h_i]_\rho]^\perp\,.\,\vec\ek_1\right) \quad 
\vec Z^h \in \xspace^h\,,
\end{equation}
recall \eqref{eq:Area}, \eqref{eq:Volume} and \eqref{eq:tauh}.

\begin{thm} \label{thm:stabC}
Let {\rm Assumption~\ref{ass:Ah2}} be satisfied and 
let $(\vec X^h(t),\kappa^h(t),\vec Y^h(t),\vec{\rm m}^h(t),$ \linebreak 
$C_1\,\beta^h(t),\vec\phi^h(t),
\lambda_{A,1}(t),\lambda_{A,2}(t),\lambda_V(t))_{t\in (0,T]}$ 
be a solution to \eqref{eq:Q0hXtLMa}, 
\eqref{eq:sdb}, \eqref{eq:sdc}, \eqref{eq:sdwfb}.
Then the solution satisfies the stability bound
\[
\ddt \,\widehat E^h(t)
+ 2\,\pi\,\sum_{i=1}^2 \left(\vec X^h_i\,.\,\vec\ek_1\,
|\mat Q^h_i\,[\vec X^h_i]_t|^2, |[\vec X^h_i]_\rho|\right)^h=0\,. 
\]
\end{thm}
\begin{proof}
Differentiating the three equations in \eqref{eq:sdwfb} with respect to $t$,
recalling \eqref{eq:dAdt}, \eqref{eq:dVdt}, and choosing
$\vec\chi = \vec X^h_t$ in \eqref{eq:Q0hXtLMa} yields
\begin{align*}
&
2\,\pi\,\sum_{i=1}^2\left((\vec X^h_i\,.\,\vec\ek_1)\,|\mat Q^h_i\,
[\vec X^h_i]_t|^2,|[\vec X^h_i]_\rho|\right)^h 
- \left([\vec Y^h_i]_\rho\,.\,\vec\nu^h_i , [\vec X^h_i]_{t,\rho}
\,.\,\vec\nu^h_i\,|[\vec X^h_i]_\rho|^{-1} \right) 
\nonumber \\ & \qquad 
= \sum_{i=1}^2 \left( \vec f^h_i, [\vec X^h_i]_t \,|[\vec X^h_i]_\rho|\right)^h ,
\end{align*}
which is equivalent to \eqref{eq:Pxtxth}. Hence the stability result follows as
in the proof of Theorem~\ref{thm:stab}.
\end{proof}

\setcounter{equation}{0}
\section{Fully discrete scheme}  \label{sec:fd}

Let $0= t_0 < t_1 < \ldots < t_{M-1} < t_M = T$ be a
partitioning of $[0,T]$ into possibly variable time steps 
$\ttau_m = t_{m+1} - t_{m}$, $m=0\to M-1$. 
For $\vec X^m  = (\vec X^m_1, \vec X^m_2) \in \xspace^h$, we let 
$\vec\tau^m_i$ and $\vec\nu^m_i$ be the natural fully discrete analogues
of $\vec\tau^h_i$ and $\vec\nu^h_i$, recall (\ref{eq:tauh}).  
In addition, let 
$\vec\omega_{i}^m \in \Vhi$ 
and $\vec v_{i}^m \in \Vhi$, $i=1,2$, 
be the natural fully discrete analogues of
(\ref{eq:omegah}) and \eqref{eq:vh0}. 
Finally, let $\mat Q^m_i \in [V^h_i]^{2\times 2}$ be the natural fully discrete 
analogue of $\mat Q^h_i$, recall \eqref{eq:Qh}. 

We propose the following fully discrete approximation of $(\BGNpwf^h)^h$,
where we make use of the reformulation in Remark~\ref{rem:C0h}.

\noindent
$(\BGNpwf^m)^{h}$
Let $\vec X^0 \in \xspace^h$, $\kappa^0 \in W^h$, $\vec Y^0 \in \yspace^h$ 
and $C_1\,\beta^0 \in \bR$ be given.
For $m=0,\ldots,M-1$, find $\delta\vec X^{m+1} \in \xspace^h$, 
with $\vec X^{m+1} = \vec X^m + \delta\vec X^{m+1}$, 
$\kappa^{m+1} \in W^h$, $C_1\,\beta^{m+1} \in \bR$,
$\vec Y^{m+1} \in \yspace^h$ with
\begin{equation} \label{eq:Ycond}
\begin{cases}
\vec Y^{m+1}_i(\tfrac12) = 2\,\pi\,\alpha^G_i\,\vec\ek_1\,,\ i=1,2
& C_1 = 0\,,\\
\vec Y^{m+1}_1(\tfrac12) - \vec Y^{m+1}_2(\tfrac12) = 2\,\pi\,[\alpha^G_1 -
\alpha^G_2]\,\vec\ek_1\,,\
\displaystyle
\sum_{i=1}^2 \left(\chi_{i,(2-i)\,J_i}\,[\vec X^m_i]_\rho, \vec Y^{m+1}_i
\right)^h = 0 & C_1 = 1\,,
\end{cases}
\end{equation}
such that
\begin{subequations} \label{eq:fd}
\begin{align}
& 2\,\pi\,\sum_{i=1}^2 \left(\vec X^m_i\,.\,\vec\ek_1\,\mat Q^m_i\,
\frac{\vec X^{m+1}_i - \vec X^m_i}{\ttau_m}, \vec\chi_i\,|[\vec X^m_i]_\rho|
\right)^h 
- \sum_{i=1}^2
\left([\vec Y^{m+1}_i]_\rho , [\vec\chi_i]_\rho\,|[\vec X^m_i]_\rho|^{-1} 
\right) 
 \nonumber \\ & \
=-\sum_{i=1}^2
\left( [\vec Y^{m}_i]_\rho\,.\,\vec\tau^m_i, 
[\vec\chi_i]_\rho\,.\,\vec\tau^m_i\,|[\vec X^m_i]_\rho|^{-1} \right)
+ C_1\,\beta^{m} \sum_{i=1}^2 
\left(\chi_{i,(2-i)\,J_i}\,[\vec\chi_i]_\rho, \vec Y^m_i \right)^h
\nonumber \\ & \quad
 - \pi\, \sum_{i=1}^2
 \left( \alpha_i \left[ \doctorkappa^h_i(\vec X^m_i,\kappa^m_i) - \spont_i 
\right]^2, \vec\chi_i\,.\,\vec\ek_1\,|[\vec X^m_i]_\rho| 
+ (\vec X^m_i\,.\,\vec\ek_1)\,\vec\tau^m_i\,.\,[\vec\chi_i]_\rho \right)^h
\nonumber \\ & \quad
+ 2\,\pi\,\sum_{i=1}^2
\alpha_i\left( \left[\doctorkappa^h_i(\vec X^m_i,\kappa^m_i) - \spont_i\right] 
 (\doctorZ^h_i - 2), 
\frac{\vec\omega_{i}^m\,.\,\vec \ek_1}{\vec X^m_i\,.\,\vec \ek_1}\,
\vec\chi_i\,.\,\vec\ek_1  \,|[\vec X^m_i]_\rho| \right)^h 
\nonumber \\ & \quad
+ 2\,\pi\,\sum_{i=1}^2
\alpha_i \left( \left[\doctorkappa^h_i(\vec X^m_i,\kappa^m_i) - \spont_i\right]
(\doctorZ^h_i - 2)\,\vec\ek_1, 
(\vec\nu^m_i\,.\,[\vec\chi_i]_\rho) \,\vec\tau^m_i + 
(\vec\tau^m_i\,.\,[\vec\chi_i]_\rho) \,
(\vec\omega_{i}^m-\vec\nu^m_i) \right)^h 
\nonumber \\ & \quad
+ \sum_{i=1}^2 \left( \kappa^m_i\,(\vec Y^m_i)^\perp ,
[\vec\chi_i]_\rho \right)^h
- \pi\,\varsigma\, \sum_{i=1}^2 \vec\chi_i(\tfrac12)\,.\,\vec\ek_1
\qquad \forall\ \vec\chi \in \xspaceh\,, \label{eq:fda} \\
&2\,\pi\,\sum_{i=1}^2
 \left( \vec X^m_i\,.\,\vec\ek_1 \,\left( \alpha_i\,[\doctorkappa^h_i(
\vec X^m_i,\kappa^{m+1}_i) - \spont_i ] \right)
,\chi_i \,|[\vec X^m_i]_\rho|\right)^h
- \sum_{i=1}^2
\left( \vec Y^{m+1}_i, \chi_i\,\vec\nu^m_i \,|[\vec X^m_i]_\rho|\right)^h 
\nonumber \\ & \hspace{8cm}
= 0 \qquad \forall\ \chi \in W^h\,,\label{eq:fdb} \\
& \sum_{i=1}^2
\left( \kappa^{m+1}_i\,\vec\nu^m_i, \vec\eta_i \,|[\vec X^m_i]_\rho|\right)^h
+ C_1\,\beta^{m+1}\,\sum_{i=1}^2 
\left(\chi_{i,(2-i)\,J_i}\,[\vec X^m_i]_\rho, \vec\eta_i \right)^h
\nonumber \\ & \hspace{3cm}
+ \sum_{i=1}^2
\left( [\vec X^{m+1}_i]_\rho, [\vec\eta_i]_\rho\,|[\vec X^m_i]_\rho|^{-1} 
\right) = 0  \quad \forall\ \vec\eta \in 
\begin{cases} 
\Vhycztest & C_1 = 0\,,\\
\Vhycotest 
& C_1 = 1\,.
\end{cases} \label{eq:fdc}
\end{align}
\end{subequations}

The linear system \eqref{eq:fd} in practice can be solved similarly to the
techniques employed by the authors in \cite{triplej,triplejMC,clust3d}. That
is, we assemble the linear systems on each curve separately, and then use
projections to enforce the matching conditions in $\xspace^h$ and
$\yspace^h_{C^1}$ for the test and trial spaces. The resulting systems of 
linear equations can be solved with preconditioned Krylov subspace iterative
solvers. Here independent direct solvers for the linear systems on 
each curve act as efficient preconditioners, where for the direct
factorisations we employ the UMFPACK package, see \cite{Davis04}.

\begin{ass} \label{ass:spanvm0}
Let $\vec X^m$ satisfy {\rm Assumption~\ref{ass:Ah2}} with $\vec X^h$ 
replaced by $\vec X^m$.
In the case $C_1=1$, we also assume that
$\vec X^m_1(q_{1,J_1-1}) \not= \vec X^m_2(q_{2,1})$, and that
$\spa\{\vec\omega^m_1(q_{1,j})\}_{j = 1}^{J_1} = 2$ or
$\spa\{\vec\omega^m_2(q_{2,j})\}_{j = 0}^{J_2-1} = 2$.
\end{ass}

\begin{lem} \label{lem:ex}
Let {\rm Assumption~\ref{ass:spanvm0}} hold.
Let $\vec X^m \in \xspace^h$, $\vec Y^m \in \yspace^h$,
$\kappa^m \in W^h$, $C_1\,\beta^m\in\bR$
and $\alpha_1,\alpha_2 \in \bRplus$,
$\spont_1,\spont_2,\alpha^G_1,\alpha^G_2 \in \bR$ be given.
Then there exists a unique solution to $(\BGNpwf^m)^h$, \eqref{eq:fd}.
\end{lem}
\begin{proof}
Let $\ell = C_1 \in \{0,1\}$.
As we have a linear system of equations, with the same number of equations as
unknowns, existence follows from uniqueness. Hence we consider a solution
to the homogeneous equivalent of \eqref{eq:fd}, and need to show that this
solution is in fact zero. In particular, let
$\delta\vec X \in \xspaceh$, 
$\kappa \in W^h$, $C_1\,\beta\in\bR$,
$\vec Y \in \yspace^h_{C^\ell}$ be such that
\begin{subequations} 
\begin{align}
& 2\,\pi\,\sum_{i=1}^2 \left( (\vec X^m_i\,.\,\vec\ek_1)\,\mat Q^m_i\,
\delta\vec X_i, \vec\chi_i\,|[\vec X^m_i]_\rho|\right)^h 
- \ttau_m\,\sum_{i=1}^2 
 \left([\vec Y_i]_\rho , [\vec\chi_i]_\rho\,|[\vec X^m_i]_\rho|^{-1} 
\right) = 0  \nonumber \\ & 
\hspace{12cm} \forall\ \vec\chi \in \xspace^h\,, \label{eq:fdproofa} \\
&2\,\pi\,\sum_{i=1}^2 \alpha_i \left( 
\vec X^m_i\,.\,\vec\ek_1\,\kappa_i ,
\chi_i \,|[\vec X^m_i]_\rho|\right)^h
- \sum_{i=1}^2 \left( \vec Y_i, \chi_i\,\vec\nu^m_i \,|[\vec X^m_i]_\rho|
\right)^h = 0
\qquad \forall\ \chi \in W^h\,,\label{eq:fdproofb} \\
&\sum_{i=1}^2 \left( \kappa_i\,\vec\nu^m_i, \vec\eta_i \,|[\vec X^m_i]_\rho|
\right)^h
+ C_1\,\beta\,\sum_{i=1}^2 
\left(\chi_{i,(2-i)\,J_i}\,[\vec X^m_i]_\rho, \vec\eta_i \right)^h
+ \sum_{i=1}^2 \left( (\delta\vec X_i)_\rho, [\vec\eta_i]_\rho\,
|[\vec X^m_i]_\rho|^{-1} \right) \nonumber \\ & \qquad
= 0
\qquad \forall\ \vec\eta \in \yspace^h_{C^\ell}\,, \label{eq:fdproofc} \\
& 
C_1\,
\sum_{i=1}^2 \left(\chi_{i,(2-i)\,J_i}\,[\vec X^m_i]_\rho, \vec Y_i \right)^h 
= 0\,. \label{eq:fdproofd}
\end{align}
\end{subequations}
Choosing $\vec\chi=\delta\vec X$ in \eqref{eq:fdproofa}, 
$\chi = \kappa$ in \eqref{eq:fdproofb} and
$\vec\eta = \vec Y$ in \eqref{eq:fdproofc} yields, that
\begin{equation} \label{eq:fdproof1}
2\,\pi\,\sum_{i=1}^2\left( \vec X^m_i\,.\,\vec\ek_1\,
|\mat Q^m_i\,\delta\vec X_i|^2,
|[\vec X^m_i]_\rho|\right)^h
+ 2\,\pi\,\sum_{i=1}^2\alpha_i\,\ttau_m \left( \vec X^m_i\,.\,\vec\ek_1\,
\kappa^2_i , |[\vec X^m_i]_\rho|\right)^h = 0\,.
\end{equation}
It follows from \eqref{eq:fdproof1}, Assumption~\ref{ass:spanvm0} 
and $\kappa \in W^h$ that $\kappa=0$.
Similarly, it follows from \eqref{eq:fdproof1}, 
$\delta\vec X \in \xspace^h$ and \eqref{eq:Qh} that 
$\delta\vec X_1(\tfrac12) = \delta\vec X_2(\tfrac12) = \vec 0$.
Hence we can choose
$\vec\eta =\delta\vec X \in \yspace^h_{C^\ell}$ 
in \eqref{eq:fdproofc} to yield 
\begin{equation*} 
\sum_{i=1}^2 \left( |(\delta\vec X_i)_\rho|^2, |[\vec X^m_i]_\rho|^{-1}\right)
= 0\,,
\end{equation*}
which implies that $\delta\vec X = \vec 0$.
In addition, if $C_1 = 1$, we recall from \eqref{eq:betavariation} that
choosing $\vec\eta = (\chi_{1,J_1}\,\vec\ek_k,\chi_{2,0}\,\vec\ek_k)$
in \eqref{eq:fdproofc}, for $k = 1,2$, yields that
\begin{equation*} 
\beta\,(\vec X^m_1(q_{1,J_1}) - \vec X^m_1(q_{1,J_1-1}) + 
\vec X^m_2(q_{2,1}) - \vec X^m_2(q_{2,0})) = \vec 0 \,.
\end{equation*}
Hence Assumption~\ref{ass:spanvm0} yields, on noting
$\vec X^m_1(q_{1,J_1}) = \vec X^m_2(q_{2,0})$, that $\beta = 0$.
Moreover, choosing $\vec\chi = \vec Y \in \yspace^h_{C^\ell} \subset \xspace^h$
in \eqref{eq:fdproofa} shows that $\vec Y_i$ is constant on $\overline I_i$,
$i=1,2$. If $C_1 = 0$, then this constant must be zero. 
If $C_1=1$, we observe from \eqref{eq:fdproofb} and \eqref{eq:omegah} that 
$\vec Y_1(q_{1,j})\,.\,\vec \omega^m_1(q_{1,j}) = 0$ for $j=1,\ldots,J_1$ and
$\vec Y_2(q_{2,j})\,.\,\vec \omega^m_2(q_{2,j}) = 0$ for $j=0,\ldots,J_2-1$.
Hence Assumption~\ref{ass:spanvm0} yields that $\vec Y = \vec 0$.
Thus we have shown the existence of a unique solution to $(\BGNpwf^m)^h$.
\end{proof}

\subsection{Conserved flows} \label{sec:fdC}

Here, following the approach in \citet[\S4.3.1]{axisd}, we consider fully
discrete conserving approximations. 
In particular, on rewriting \eqref{eq:fda} as
\begin{align*}
& 2\,\pi\,\sum_{i=1}^2 \left(\vec X^m_i\,.\,\vec\ek_1\,\mat Q^m_i\,
\frac{\vec X^{m+1}_i - \vec X^m_i}{\ttau_m}, \vec\chi_i\,|[\vec X^m_i]_\rho|
\right)^h 
- \sum_{i=1}^2
\left([\vec Y^{m+1}_i]_\rho , [\vec\chi_i]_\rho\,|[\vec X^m_i]_\rho|^{-1} 
\right) \\ & \qquad
 = \sum_{i=1}^2\left( \vec f^m_i, \vec\chi_i \, |[\vec X^m_i]_\rho| \right)^h ,
\end{align*}
we can formulate our surface area and volume conserving variant for
$(\BGNpwf^m)^h$ as follows. 

$(\BGNpwf^m_{A,V})^h$:
Let $\vec X^0 \in \xspace^h$, $\kappa^0 \in W^h$, $\vec Y^0 \in \yspace^h$ 
and $C_1\,\beta^0 \in \bR$ be given.
For $m=0,\ldots,M-1$, find $\delta\vec X^{m+1} \in \xspace^h$, 
with $\vec X^{m+1} = \vec X^m + \delta\vec X^{m+1}$, 
$\kappa^{m+1} \in W^h$, $C_1\,\beta^{m+1} \in \bR$,
$\vec Y^{m+1} \in \yspace^h$ with \eqref{eq:Ycond},  
and $\lambda_{A,1}^{m+1}, \lambda_{A,2}^{m+1},\lambda_V^{m+1} \in \bR$ 
such that \eqref{eq:fdb}, \eqref{eq:fdc} and
\begin{subequations} \label{eq:fdwf}
\begin{align}
& 2\,\pi\,\sum_{i=1}^2 \left(\vec X^m_i\,.\,\vec\ek_1\,\mat Q^m_i\,
\frac{\vec X^{m+1}_i - \vec X^m_i}{\ttau_m}, \vec\chi_i\,|[\vec X^m_i]_\rho|
\right)^h 
- \sum_{i=1}^2
\left([\vec Y^{m+1}_i]_\rho , [\vec\chi_i]_\rho\,|[\vec X^m_i]_\rho|^{-1} 
\right) 
\nonumber \\ & \qquad
= \sum_{i=1}^2\left( \vec f^m_i, \vec\chi_i \, |[\vec X^m_i]_\rho| \right)^h
- 2\,\pi\,\sum_{i=1}^2\lambda_{A,i}^{m+1}
\left[
\left( \vec\ek_1, \vec\chi_i\,|[\vec X^m_i]_\rho|\right) 
+ \left( (\vec X^m_i\,.\,\vec\ek_1)\,\vec\tau^m_i
, [\vec\chi_i]_\rho\right) \right]
\nonumber \\ & \qquad\qquad
- 2\,\pi\,\lambda_V^{m+1} \sum_{i=1}^2 
\left((\vec X^m_i\,.\,\vec\ek_1)\,\vec\nu^m,
\vec\chi_i \, |[\vec X^m_i]_\rho| \right)
\qquad \forall\ \vec\chi \in \xspaceh\,, \label{eq:fdwfa} \\
&{\rm (i)}\ A_i(\vec X^{m+1}) = A_i(\vec X^0)\,,\ i=1,2\,,\quad
{\rm (ii)}\ V(\vec X^{m+1}) = V(\vec X^0) \label{eq:fdwfb} 
\end{align}
\end{subequations}
hold.
Here we have recalled \eqref{eq:Areah} and \eqref{eq:Volumeh}.

The nonlinear system of equations arising at each time level of
$(\BGNpwf^m_{A,V})^h$ can be solved with a suitable iterative solution method, 
see below.
In the simpler case of phase area conserving flow, we need to find
$(\delta\vec X^{m+1}, \kappa^{m+1}, \vec Y^{m+1}, C^1\,\beta^{m+1},
\lambda_{A,1}^{m+1},\lambda_{A,2}^{m+1}, \lambda_V^{m+1}) 
\in \xspace^h \times W^h \times \yspace^h \times \bR \times \bR^2 \times \{0\}$
such that \eqref{eq:fdb}, \eqref{eq:fdc}, \eqref{eq:fdwfa} and 
\eqref{eq:fdwfb}(i) hold. 
Similarly, for volume conserving flow, we need to find
$(\delta\vec X^{m+1}, \kappa^{m+1}, \vec Y^{m+1}, C^1\,\beta^{m+1},
\lambda_{A,1}^{m+1},\lambda_{A,2}^{m+1}, \lambda_V^{m+1}) 
\in \xspace^h \times W^h \times \yspace^h \times \bR
\times \{0\}^2 \times \bR$
such that \eqref{eq:fdb}, \eqref{eq:fdc}, \eqref{eq:fdwfa} and 
\eqref{eq:fdwfb}(ii) hold. 

Adapting the strategy in \citet[\S4.3.1]{axisd}, 
we now describe a Newton method for solving the nonlinear system 
(\ref{eq:fdwf}), \eqref{eq:fdb} and \eqref{eq:fdc},
where for ease of presentation we suppress the dependence on $\beta^{m+1}$.
The linear system (\ref{eq:fdwfa}), \eqref{eq:fdb} and \eqref{eq:fdc}, 
with $(\lambda_{A,1}^{m+1}, \lambda_{A,2}^{m+1}, \lambda_V^{m+1})$ in
(\ref{eq:fdwfa}) replaced by 
$\lambda = (\lambda_{A,1}, \lambda_{A,2}, \lambda_V)$, 
can be written as:
Find $(\delta\vec X^{m+1}(\lambda), 
\kappa^{m+1}(\lambda), \vec Y^{m+1}(\lambda)) 
\in \xspace^h\times W^h \times \yspace^h$ such that
\begin{equation*} 
\mathbb{T}^m\,\begin{pmatrix}
\vec Y^{m+1}(\lambda)\\[1mm]
\delta\vec X^{m+1}(\lambda)\\[1mm]
\kappa^{m+1}(\lambda)
\end{pmatrix}
= \begin{pmatrix} \vec{\underline{\mathfrak g}}^m \\[1mm] 0 \\[1mm] \vec 0
\end{pmatrix}
+ \sum_{\ell=1}^2 \lambda_{A,\ell}\, \begin{pmatrix} \vec{\underline{\mathcal K}}^m_\ell 
\\[1mm] 0 \\[1mm]\vec 0
\end{pmatrix}
+ \lambda_V\, \begin{pmatrix} \vec{\underline{\mathcal N}}^m 
\\[1mm] 0 \\[1mm]\vec 0
\end{pmatrix}.
\end{equation*}
Assuming the linear operator $\mathbb{T}^m$ is invertible, we obtain that
\begin{align}
\begin{pmatrix}
\vec Y^{m+1}(\lambda)\\[1mm]
\delta\vec X^{m+1}(\lambda)\\[1mm]
\kappa^{m+1}(\lambda)
\end{pmatrix} &
= (\mathbb{T}^m)^{-1}
\left[\begin{pmatrix} \vec{\underline{\mathfrak g}}^m \\[1mm]
 0 \\[1mm] \vec 0  \end{pmatrix}
+ \sum_{\ell=1}^2 \lambda_{A,\ell} 
\begin{pmatrix} \vec{\underline{\mathcal K}}^m_\ell \\[1mm] 0 \\[1mm] \vec 0
\end{pmatrix}
+ \lambda_V\, 
\begin{pmatrix} \vec{\underline{\mathcal N}}^m \\[1mm] 0 \\[1mm] \vec 0
\end{pmatrix}\right] \nonumber \\ & 
=: (\mathbb{T}^m)^{-1} 
\begin{pmatrix} \vec{\underline{\mathfrak g}}^m \\[1mm] 0 \\[1mm] \vec 0
\end{pmatrix}
+ \sum_{\ell=1}^2\lambda_{A,\ell} \begin{pmatrix} \vec{\underline s}^m_{\ell,1} \\[1mm] 
\vec{\underline s}^m_{\ell,2} \\[1mm] {\underline s}^m_{\ell,3}
\end{pmatrix}
+ \lambda_V\, \begin{pmatrix} \vec{\underline q}^m_1 \\[1mm] 
\vec{\underline q}^m_2 \\[1mm] {\underline q}^m_3
\end{pmatrix} .
\label{eq:lmsysinverse}
\end{align}
It immediately follows from (\ref{eq:lmsysinverse}) that
\begin{equation*}
\partial_{\lambda_{A,\ell}} \vec X^{m+1}(\lambda) 
= \vec{\underline s}^m_{\ell,2}\,,\ \ell=1,2\,,\quad
\partial_{\lambda_V} \vec X^{m+1}(\lambda)
= \vec{\underline q}^m_2\,,
\end{equation*}
where $\vec X^{m+1}(\lambda) = \vec X^m + \delta\vec X^{m+1}(\lambda)$.
Hence
\begin{align*}
\partial_{\lambda_{A,\ell}} A_1(\vec X^{m+1}(\lambda))
  & = \left[\deldel{\vec X^{m+1}}\, A_1(\vec X^{m+1}(\lambda))\right]
(\vec s^m_{\ell,2})\,, \\
\partial_{\lambda_{A,\ell}} A_2(\vec X^{m+1}(\lambda))
  & = \left[\deldel{\vec X^{m+1}}\, A_2(\vec X^{m+1}(\lambda))\right]
(\vec s^m_{\ell,2})\,, \\
\partial_{\lambda_{A,\ell}} V(\vec X^{m+1}(\lambda))
 & = \left[\deldel{\vec X^{m+1}}\, V(\vec X^{m+1}(\lambda))\right]
(\vec s^m_{\ell,2})\,,
\end{align*}
for $\ell=1,2$, and similarly for 
$\partial_{\lambda_V} A_i(\vec X^{m+1}(\lambda))$, $i=1,2$,
and $\partial_{\lambda_V} V(\vec X^{m+1}(\lambda))$.
Here $\vec s^m_{\ell,2} \in \xspace^h$ 
is the finite element function corresponding to the
coefficients in $\vec{\underline s}^m_{\ell,2}$ for the standard basis of 
$\xspace^h$.
Moreover, on recalling \eqref{eq:dAdt} and \eqref{eq:dVdt}, 
we have defined the first variations of $A_i(\vec Z^h)$,
for any $\vec Z^h \in \xspace^h$, as
\begin{align*}
\left[\deldel{\vec Z^h}\, A_i(\vec Z^h)\right](\vec\eta)
 & = \lim_{\epsilon\to0} \frac1\epsilon\left(
A_i(\vec Z^h + \epsilon\,\vec\eta) - A_i(\vec Z^h)\right) \\ &
= 2\,\pi
\left(\vec\eta_i\,.\,\vec\ek_1,|[\vec Z^h_i]_\rho|\right) + 2\,\pi \left(
(\vec Z^h_i\,.\,\vec\ek_1)\,
[\vec\eta_i]_\rho,[\vec Z^h_i]_\rho\, |[\vec Z^h_i]_\rho|^{-1} \right)
\quad \forall\ \vec\eta \in \xspace^h \,,
\end{align*}
and similarly
\begin{align*}
\left[\deldel{\vec Z^h}\, V(\vec Z^h)\right](\vec\eta)
& = \lim_{\epsilon\to0} \frac1\epsilon\left(
V(\vec Z^h + \epsilon\,\vec\eta) - V(\vec Z^h)\right) \nonumber \\ & 
= - 2\,\pi \sum_{i=1}^2
\left( \vec Z^h_i\,.\,\vec\ek_1, \vec\eta_i\,.\,
[[\vec Z^h_i]_\rho]^\perp\right) 
\quad \forall\ \vec\eta \in \xspace^h \,.
\end{align*}
We can then proceed as in \citet[(4.13)]{axisd} to define a Newton iteration
for finding a solution to the nonlinear system $(\BGNpwf^m_{A,V})^h$.
In practice this Newton iteration always converged within a couple of 
iterations.

\clearpage
\setcounter{equation}{0}
\section{Numerical results} \label{sec:nr}

As the fully discrete energy for the scheme $(\BGNpwf^m)^{h}$,
on recalling (\ref{eq:Eh}), we define
\begin{align*} 
\widehat E^{m+1} & = \pi\, \sum_{i=1}^2 \left(\alpha_i
\left[ \doctorkappa^h_i(\vec X^m_i, \kappa^{m+1}_i) - \spont_i \right]^2 ,
\vec X^m_i\,.\,\vec\ek_1 \,|[\vec X^m_i]_\rho|\right)^h 
\nonumber \\ & \qquad
- 2\,\pi\,\sum_{i=1}^2 \alpha^G_i\, \vec{\rm m}^{m+1}_i\,.\,\vec\ek_1
+ \pi\,\varsigma\,\sum_{i=1}^2 \vec X^{m+1}_i(\tfrac12)\,.\,\vec\ek_1\,, 
\end{align*}
where, e.g., 
\begin{align*} 
\vec{\rm m}^{m+1}_2 & = 
\left( 1, \chi_{2,0}\,|[\vec X^m_2]_\rho|\right) 
\left[\kappa^{m+1}_2\,\vec\omega^m_2\right](\tfrac12)
+ C_1\,\beta^{m+1}\left( 1 , \chi_{2,0} \right) [\vec X^m_2]_\rho (\tfrac12)
\nonumber \\ & \quad
+ \left( 1 , (\chi_{2,0})_\rho \,|[\vec X^m_2]_\rho|^{-1}\right) 
[\vec X^{m+1}_2]_\rho (\tfrac12) 
\end{align*}
is a fully discrete approximation to $\vec{\rm m}_2^h$ defined in 
\eqref{eq:sdc}, recall \eqref{eq:fdc}. 

Given $\vec X^0$, we set $\beta^0=0$ and define the following initial data.
First, we let $\vec\kappa^0_i\in \Vhi$ be such that
\begin{equation*} 
\left( \vec\kappa^{0}_i,\vec\eta_i \,|[\vec X^0_i]_\rho|\right)^h
+ \left( \vec\tau^{0}_i , [\vec\eta_i]_\rho \right)
= 0 \quad \forall\ \vec\eta_i \in \Vhi\,,
\end{equation*}
and then define $\kappa^0_{\star,i} = \pi^h_i
[\vec\kappa^{0}_i\,.\,\vec v_{i}^0]$, $i=1,2$. Now $\kappa^0 \in W^h$ is
defined as the orthogonal projection of $\kappa^0_\star$ onto $W^h$.
Moreover, we let $\vec Y^0_{\star,i} \in \Vhi$ be such that
\begin{equation*} 
\vec Y^0_{\star,i} = 2\,\pi\,\alpha_i\,\vec\pi^h_i
\left[|\vec\omega_{i}^0|^{-1}
\vec X^0_i\,.\,\vec\ek_1 \left[ \doctorkappa^h_i(\vec X^0_i, 
\kappa^0_i) - \spont_i\right] \vec v_{i}^0 \right],
\end{equation*}
and then define
$\vec Y^0_\dag \in \yspace^h_{C^0}$ as 
the orthogonal projection of $\vec Y^0_\star$ onto $\yspace^h_{C^0}$.
Finally, we let $\vec Y^0 \in \yspace^h$ via
\begin{equation*} 
\vec Y^0_i(q_{i,j}) = \begin{cases}
2\,\pi\,\alpha^G_i\,\vec\ek_1 & q_{i,j} = \tfrac12\,,\\
\vec Y^0_{\dag,i}(q_{i,j}) & q_{i,j} \in \overline I_i \setminus \{\tfrac12\}\,,
\end{cases}
\quad j = 0,\ldots,J_i\,,\ i = 1,2\,.
\end{equation*}

Unless otherwise stated, we use $\alpha_1 = \alpha_2 = 1$,
$\spont_1 = \spont_2 = \varsigma = \alpha^G_1 = \alpha^G_2 = 0$ and 
compute simulations of the unconstrained gradient flow. 
We will always use uniform time steps,
$\ttau_m = \ttau$, $m=0,\ldots,M-1$.
For the visualisations, we will display phase 1 in red, and phase 2 in yellow.

\subsection{$C^0$--junctions}

The evolution in Figure~\ref{fig:test1} starts from two symmetric surfaces that
meet at a $C^0$--junction line. For the first four experiments in this 
subsection, we use the discretisation parameters $\ttau = 10^{-3}$ and 
$J_1=J_2=65$.
The evolution appears to show that the fastest way to
reduce the overall energy to zero is to flatten and to enlarge the surfaces.
We conjecture that the surfaces are going to converge
to two flat disks with their radius converging to infinity. 
By adding a non-zero line energy, the growth to infinity is prevented. In fact,
repeating the simulation for any positive $\varsigma$ will lead to the surfaces
shrinking to a point. An example is seen in Figure~\ref{fig:test1line}, 
where we used $\varsigma=0.02$.
\begin{figure}
\center
\mbox{
\includegraphics[angle=-90,width=0.25\textwidth]{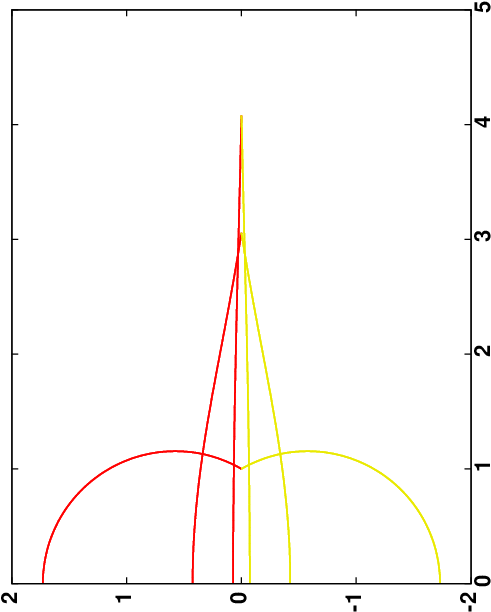} 
\includegraphics[angle=-90,width=0.25\textwidth]{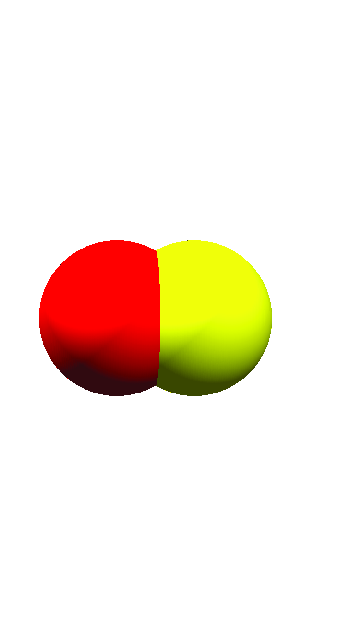} 
\includegraphics[angle=-90,width=0.25\textwidth]{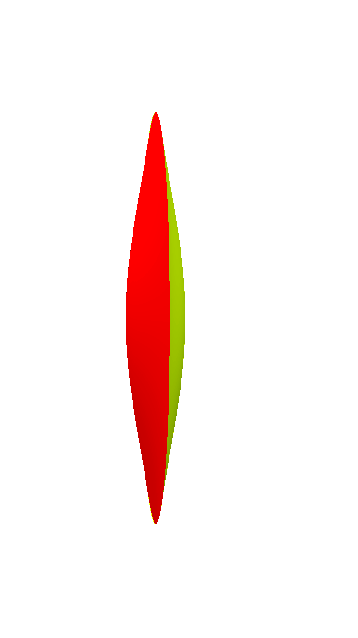} 
\includegraphics[angle=-90,width=0.25\textwidth]{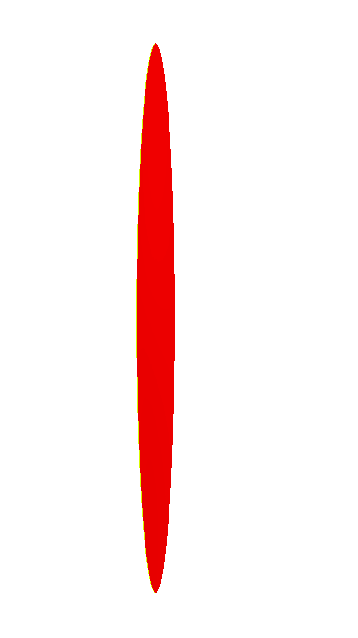} 
}
\caption{($C^0$)
Plots at times $t=0,1,10$.
}
\label{fig:test1}
\end{figure}%
\begin{figure}
\center
\mbox{
\includegraphics[angle=-90,width=0.25\textwidth]{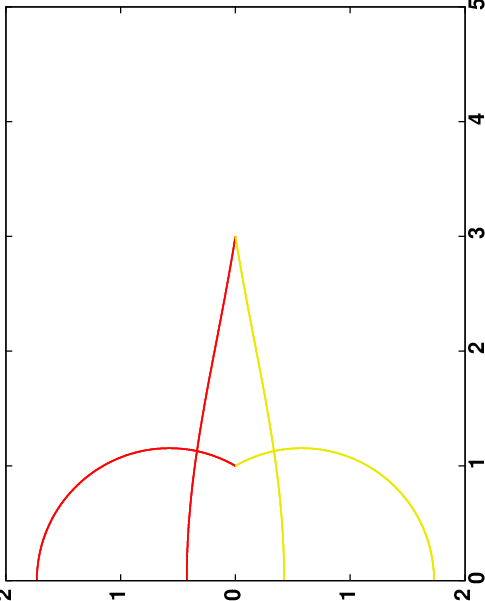} 
\includegraphics[angle=-90,width=0.25\textwidth]{figures/c0c1_c0longt0S} 
\includegraphics[angle=-90,width=0.25\textwidth]{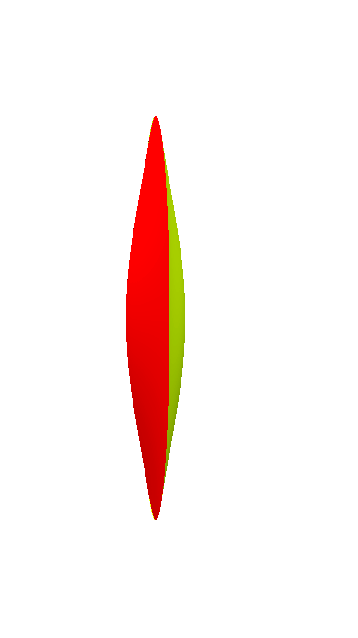} 
\includegraphics[angle=-90,width=0.25\textwidth]{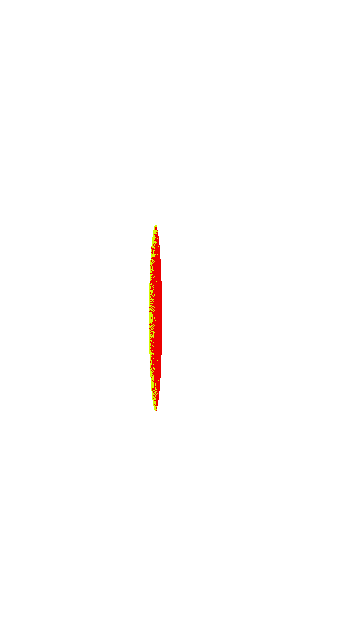} 
}
\caption{($C^0$: $\varsigma = 0.02$)
Plots at times $t=0,1,10$.
}
\label{fig:test1line}
\end{figure}%
To conclude this subsection, we show an experiment for phase area and 
volume conserving flow in Figure~\ref{fig:test3c3}.
\begin{figure}
\center
\includegraphics[angle=-90,width=0.2\textwidth]{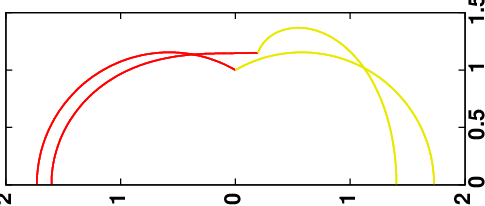} \qquad
\includegraphics[angle=-90,width=0.3\textwidth]{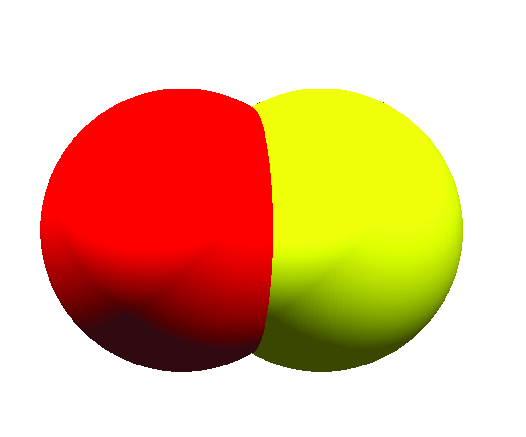} 
\includegraphics[angle=-90,width=0.3\textwidth]{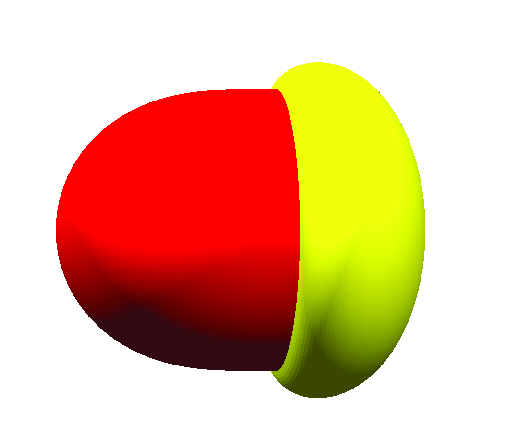} 
\caption{($C^0$ with phase area and volume conservation, 
$\spont_1 = -0.5$, $\spont_2 = -4$)
Plots at times $t=0,1$.
}
\label{fig:test3c3}
\end{figure}%

In Figure~\ref{fig:new1cons3} we show a simulation for a flat disc separated
into two phases, where phase 2 has two connected components. We note that the
model and theory presented in this paper, for simplicity, only considered the
case of a single junction being present. But it is a straightforward matter to
extend the ideas, and the approximations, to more than one junction. Clearly,
in the example in Figure~\ref{fig:new1cons3} two junctions are present.
We let $\ttau=10^{-4}$ and $(J_1,J_2) = (47,84)$.
\begin{figure}
\center
\mbox{
\includegraphics[angle=-90,width=0.25\textwidth]{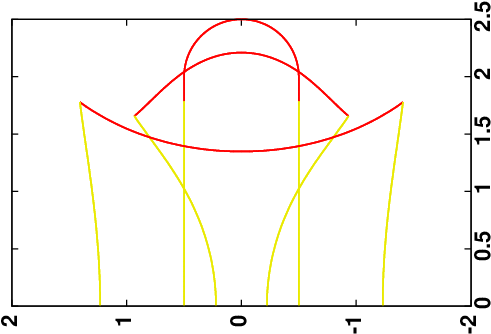} 
\includegraphics[angle=-90,width=0.25\textwidth]{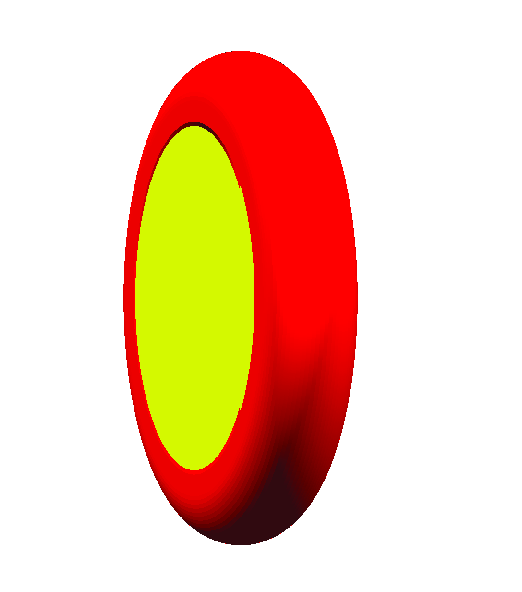} 
\includegraphics[angle=-90,width=0.25\textwidth]{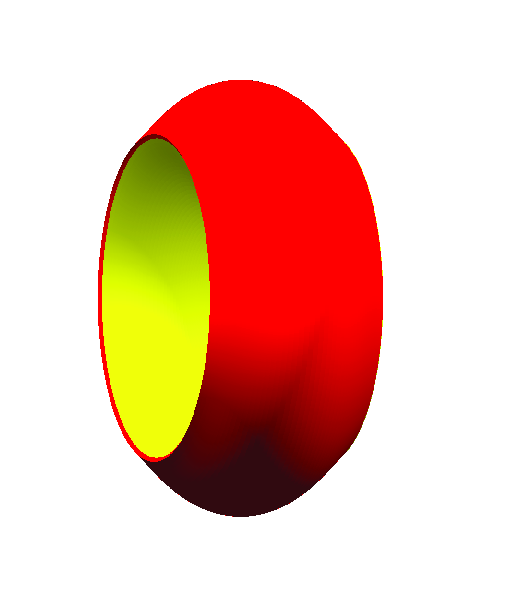} 
\includegraphics[angle=-90,width=0.25\textwidth]{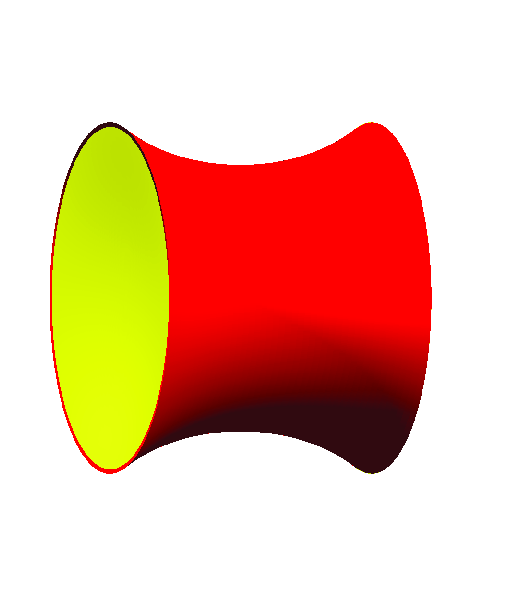} 
}
\caption{($C^0$ with phase area and volume conservation)
Plots at times $t=0,0.1,1$.
}
\label{fig:new1cons3}
\end{figure}%

\subsection{$C^1$--junctions}

We begin with a study of the tangential motion at the junction, recall
Remark~\ref{rem:equid}. 
To this end, we compare the results from our
scheme \eqref{eq:fd} to the ones from an alternative fully discrete
approximation that is based on \eqref{eq:sideh} in place of 
\eqref{eq:sidehtang}. For the experiments in Figure~\ref{fig:tm_test}
we start with each phase represented by a quarter of a unit circle.
As discretisation parameters we use $\ttau=10^{-4}$ and
$(J_1,J_2) = (65,9)$, so that the upper phase is much finer discretised than
the lower phase. On the continuous level, the initial data is a steady state
solution. However, the scheme based on \eqref{eq:sideh} induces a tangential
motion of the junction point that is based purely on the discretisation. As a
side effect, the whole surface moves up, which is not physical.
In contrast, the evolution for our scheme \eqref{eq:fd} is nearly stationary.
We note that the condition \eqref{eq:Xhq1q0} leads to some change at the lower 
boundary, and we observe a small tangential motion of the junction point.
\begin{figure}
\center
\includegraphics[angle=-90,width=0.2\textwidth]{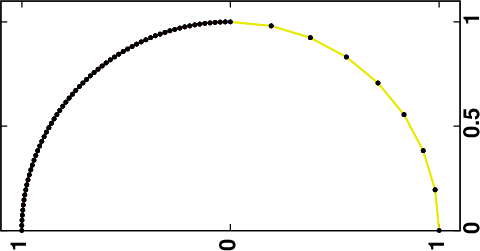} \qquad
\includegraphics[angle=-90,width=0.2\textwidth]{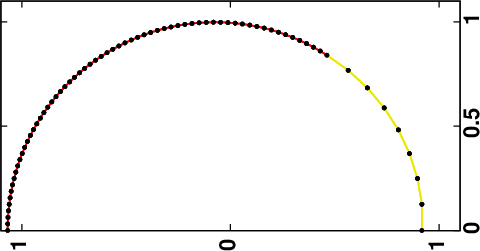}\qquad
\includegraphics[angle=-90,width=0.2\textwidth]{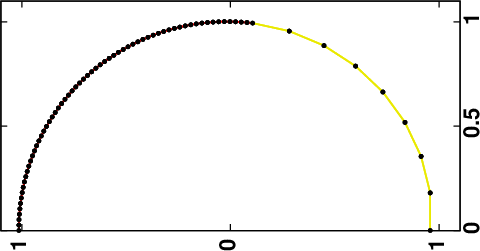} 
\caption{($C^1$)
The plots show the initial data (left), the solution of the scheme based on 
\eqref{eq:sideh} at time $t=1$ (middle), 
and the solution from \eqref{eq:fd} at time $t=1$ (right).
}
\label{fig:tm_test}
\end{figure}%

As another comparison, which highlights the rather subtle effects of changing
\eqref{eq:sidehtang} to \eqref{eq:sideh}, we repeat the experiment in
Figure~\ref{fig:test3c3}, but now for a $C^1$--junction with only
phase area preservation. 
As the discretisation parameters we use $J_1=J_2=65$ and $\ttau=10^{-4}$.
While our scheme \eqref{eq:fd} shows a monotonically decreasing discrete 
energy, see Figure~\ref{fig:c1test3c1}, the fully
discrete approximation based on \eqref{eq:sideh} exhibits a highly oscillatory
energy plot, and some non-trivial tangential motion at the junction point that
leads to rather large elements near the junction.
\begin{figure}
\center
\mbox{
\includegraphics[angle=-90,width=0.2\textwidth]{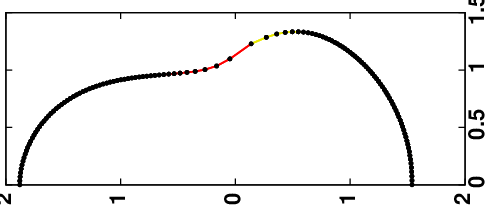} 
\includegraphics[angle=-90,width=0.3\textwidth]{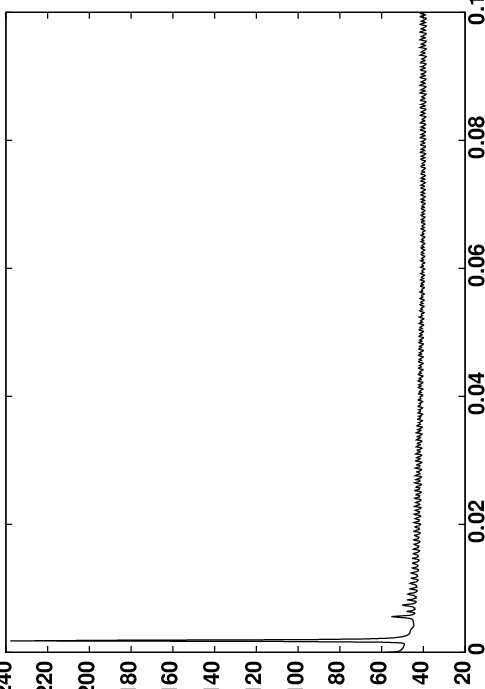} 
\includegraphics[angle=-90,width=0.2\textwidth]{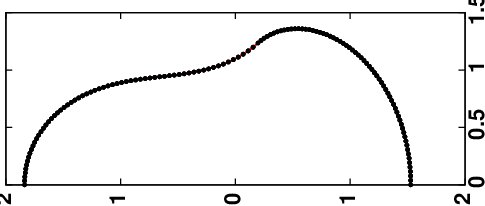} 
\includegraphics[angle=-90,width=0.3\textwidth]{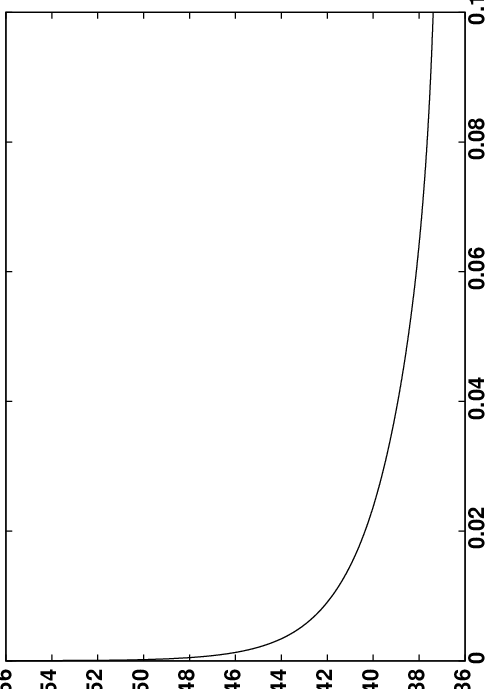} 
}
\caption{($C^1$ with phase area conservation, 
$\spont_1 = -0.5$, $\spont_2 = -4$) 
On the left we show the solution at time $t=0.1$, and a plot of the
discrete energy, for a fully discrete approximation based on \eqref{eq:sideh}.
On the right we display the same for our scheme \eqref{eq:fd}.
}
\label{fig:c1test3c1}
\end{figure}%
This in turn leads to bad curvature approximations at the junction. We
visualise this in Figure~\ref{fig:c1kappajump}, where for the final solution of
both schemes we
plot the approximations $\doctorkappa^h_i(\vec X^M_i, \kappa^M_i)$
of $\varkappa_{\mathcal{S}_i}$, $i=1,2$, against arclength.
Clearly, the curvature approximations from the scheme based on \eqref{eq:sideh}
are completely unphysical. The discretisations from our scheme \eqref{eq:fd},
on the other hand , 
approximately satisfy \eqref{eq:axiC1bc1} and \eqref{eq:axiC1bc2}, which yield
$\varkappa_{\mathcal{S}_1} - \varkappa_{\mathcal{S}_2} = 3.5$ and
$(\varkappa_{\mathcal{S}_1})_s = (\varkappa_{\mathcal{S}_2})_s$, respectively,
for the continuous solution at the junction.
\begin{figure}
\center
\includegraphics[angle=-90,width=0.4\textwidth]{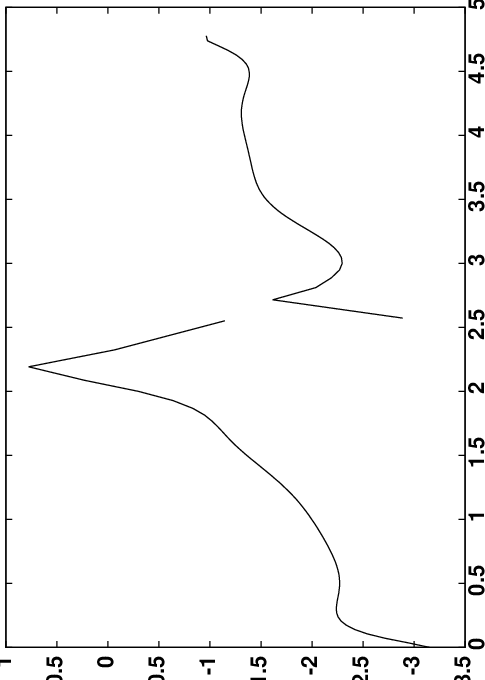}  \qquad
\includegraphics[angle=-90,width=0.4\textwidth]{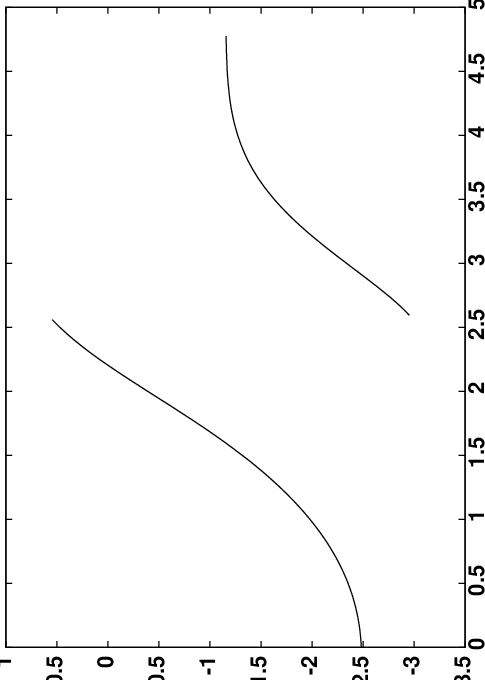} 
\caption{($C^1$ with phase area conservation, 
$\spont_1 = -0.5$, $\spont_2 = -4$) 
A plot of $\doctorkappa^h_i(\vec X^M_i, \kappa^M_i)$, $i=1,2$, against
arclength of $\overline{\Gamma^M_1} \cup \overline{\Gamma^M_2}$, for the two
experiments in Figure~\ref{fig:c1test3c1}.
}
\label{fig:c1kappajump}
\end{figure}%

Hence, from now on, we only consider simulations for the scheme \eqref{eq:fd}.
To begin, we perform a convergence experiment for the special case that
the two phases have identical physical properties, with
$\spont_1 = \spont_2 = \spont = -1$. 
Then a sphere of radius $R(t)$, where $R(t)$ satisfies
\begin{equation} \label{eq:ODE}
R'(t) = - \tfrac\spont{R(t)}\,(\tfrac2{R(t)} + \spont)\,,
\quad R(0) = 1\,,
\end{equation}
is a solution to (\ref{eq:gradflowlambda}) with
$\lambda_{A,1}=\lambda_{A,2}=\lambda_V = 0$.
The nonlinear ODE (\ref{eq:ODE}) is solved by 
$R(t) = z(t) - \tfrac2\spont$, where $z(t)$ is such that 
$\tfrac12\,(z^2(t) - z_0^2) - \tfrac4\spont\,(z(t)-z_0) + \tfrac4{\spont^2}\,
\ln \tfrac{z(t)}{z_0} + \spont^2\,t = 0$, with $z_0 = 1 + \tfrac2\spont$.
We use the solution to (\ref{eq:ODE}), with $\spont = -1$, and 
a sequence of approximations for the unit sphere to compute the error 
\[
\errorXx = \max_{m=1,\ldots,M} \max_{i=1,2} \max_{j=0,\ldots,J_i} \left| 
|\vec X^m_i(q_{i,j})| - R(t_m) \right|
\]
over the time interval $[0,T]$, for $T=1$, between
the true solution and the discrete solutions for the scheme \eqref{eq:fd}. 
This error only measures the accuracy of the normal motion of the interface,
accounting for the fact that the continuous problem has a whole family of
solutions, with the tangential motion essentially arbitrary. Nevertheless, 
in the absence of tangential energetic forcings, any numerical method should 
ensure that the phase boundary does not move tangentially during the evolution.
In order to measure this property, we also compute the quantity
$|\vec X^M(\tfrac12) - R(T)\,\vec\ek_1|$
for the solutions of the scheme \eqref{eq:fd}. 
As initial data we choose $\vec X^0 \in \xspace^h$ with
\begin{align*} 
\vec X^0_1(q_{1,j}) & = \begin{pmatrix} 
\cos[(\tfrac12 - q_{1,j})\,\pi + 0.1\,\cos((\tfrac12 - 2\,q_{1,j})\,\pi)] \\
\sin[(\tfrac12 - q_{1,j})\,\pi + 0.1\,\cos((\tfrac12 - 2\,q_{1,j})\,\pi)]
\end{pmatrix}, \quad j = 0,\ldots,J_1\,,\nonumber \\
\vec X^0_2(q_{2,j}) & = \begin{pmatrix} 
\cos[(\tfrac12 - q_{2,j})\,\pi + 0.1\,\cos((\tfrac12 - 2\,q_{2,j})\,\pi)] \\
\sin[(\tfrac12 - q_{2,j})\,\pi + 0.1\,\cos((\tfrac12 - 2\,q_{2,j})\,\pi)]
\end{pmatrix}, \quad j = 0,\ldots,J_2\,,
\end{align*}
recall (\ref{eq:Jequi}), which ensures that the evolutions for \eqref{eq:fd}  
will exhibit some tangential motion within each phase.
We use the time step size $\ttau= 10^{-3}\,h^2_{\Gamma^0}$,
where $h_{\Gamma^0}$ is the maximal edge length of $\Gamma^0 = 
(\Gamma^0_1,\Gamma^0_2)$, and report the computed errors in 
Table~\ref{tab:spont-1adapt2}. The reported errors appear to indicate
an at least linear convergence rate for the two error quantities.
We remark that the final element ratios
\begin{equation*} 
r^M_i = \dfrac{\max_{j=1,\ldots, J_i} 
|\vec X^M_i(q_{i,j}) - \vec X^M_i(q_{i,j-1})|}
{\min_{j=1,\ldots,J_i} |\vec X^M_i(q_{i,j}) - \vec X^M_i(q_{i,j-1})|}\,,\
i=1,2\,,
\end{equation*}
have the value $1$ for each of the runs displayed in 
Table~\ref{tab:spont-1adapt2}. Of course, this is to be expected from the 
equidistribution results in Remark~\ref{rem:equid}.
\begin{table}
\center
\begin{tabular}{|cr|c|c|c|c|}
\hline
$(J_1-1,J_2-1)$ & $h_{\Gamma^0}$ & $\errorXx$ & EOC 
& $|\vec X^M(\tfrac12) - R(T)\,\vec\ek_1|$ & EOC
\\ \hline
(16,8)   & 2.3408e-01 & 4.4399e-02 & ---  & 3.9101e-02 & ---  \\  
(32,16)  & 1.1762e-01 & 1.3277e-02 & 1.75 & 1.8489e-02 & 1.09 \\  
(64,32)  & 5.8881e-02 & 3.8599e-03 & 1.79 & 9.1529e-03 & 1.02 \\  
(128,64) & 2.9449e-02 & 1.0863e-03 & 1.83 & 4.5772e-03 & 1.00 \\ 
(256,128)& 1.4726e-02 & 3.8711e-04 & 1.49 & 2.2921e-03 & 1.00 \\ 
\hline
\end{tabular}
\caption{Errors for the convergence test with $\spont_1 = \spont_2 = -1$
for the scheme $(\BGNpwf^m)^h$.}
\label{tab:spont-1adapt2}
\end{table}%

In the next experiments we approximate well-known equilibrium shapes from
\citet[Fig.\ 8]{JulicherL96}, see also the experiments in 
\citet[Fig.~7.21]{pwfc0c1}.
To this end, we consider the volume and phase area conserving flow for 
initial surfaces with reduced volumes
$v_r \in \{0.95,\, 0.91,\, 0.9,\ 0.885,\ 0.84,\ 0.8\}$,
where
\[
v_r = \frac{3\,V(\vec X^0)}
{4\,\pi\,(\frac{A_1(\vec X^0) + A_2(\vec X^0)}{4\,\pi})^\frac32}
= \frac{6\,\pi^\frac12\,V(\vec X^0)}
{(A_1(\vec X^0) + A_2(\vec X^0))^\frac32}\,.
\]
In addition, the surface areas are fixed so that 
$A_1(\vec X^0) + A_2(\vec X^0) = 4\,\pi$ and so that
the two phases have a surface area ratio of 
$\frac{A_1(\vec X^0)}{A_1(\vec X^0) + A_2(\vec X^0)} = 0.1$.
See Figure~\ref{fig:c1ESinit} for the initial shapes,
where the spatial discretisation parameters are given by
$(J_1,J_2) = (93,421)$, $(91,423)$, $(90,424)$, $(92,422)$, $(95,419)$
and $(97,417)$, respectively.
For these experiments we set $\varsigma = 9$.
Choosing a time step size of $\ttau = 10^{-5}$, we integrate the volume and
phase area conserving flow until the discrete energy becomes stationary,
and we report on the obtained shapes in Figure~\ref{fig:c1ES}. 
These configurations appear to
agree well with the computed shapes in \citet[Fig.\ 8]{JulicherL96}.
\begin{figure}
\center
\mbox{
\includegraphics[angle=-90,width=0.16\textwidth]{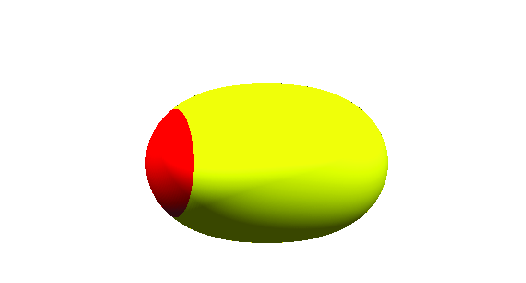} 
\includegraphics[angle=-90,width=0.16\textwidth]{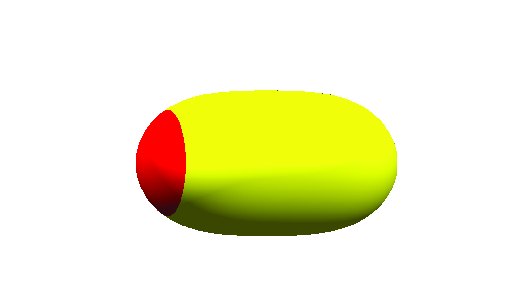} 
\includegraphics[angle=-90,width=0.16\textwidth]{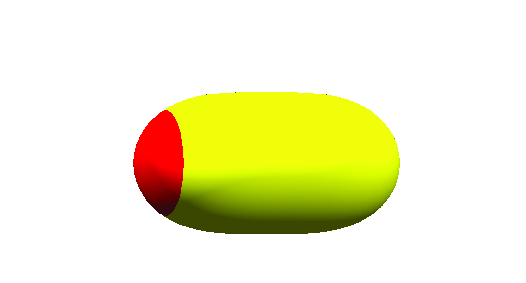} 
\includegraphics[angle=-90,width=0.16\textwidth]{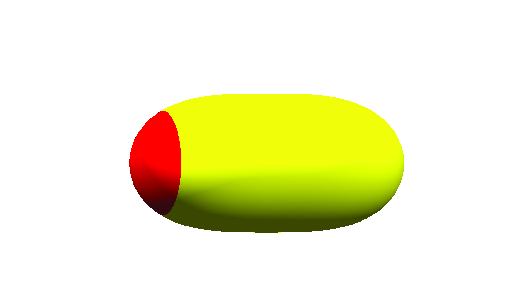} 
\includegraphics[angle=-90,width=0.16\textwidth]{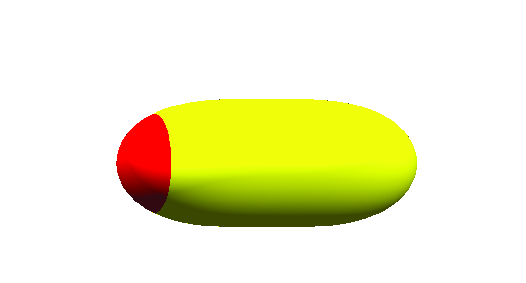} 
\includegraphics[angle=-90,width=0.16\textwidth]{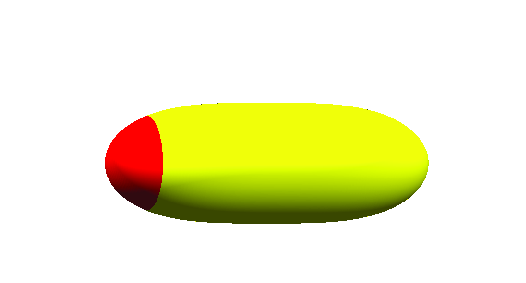} 
}
\caption{
The initial shapes for $v_r = 0.95$, $0.91$, $0.9$, $0.885$, $0.84$ and $0.8$, 
respectively.
}
\label{fig:c1ESinit}
\end{figure}%
\begin{figure}
\center
\mbox{
\includegraphics[angle=-90,width=0.15\textwidth]{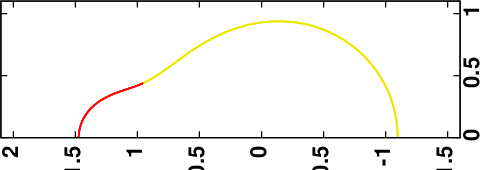} 
\includegraphics[angle=-90,width=0.15\textwidth]{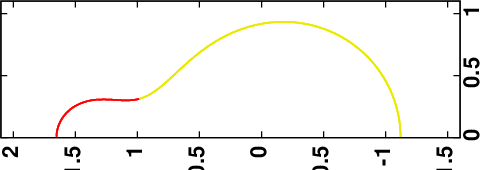} 
\includegraphics[angle=-90,width=0.15\textwidth]{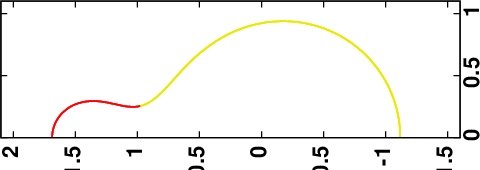} 
\includegraphics[angle=-90,width=0.15\textwidth]{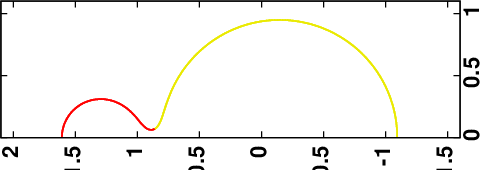} 
\includegraphics[angle=-90,width=0.15\textwidth]{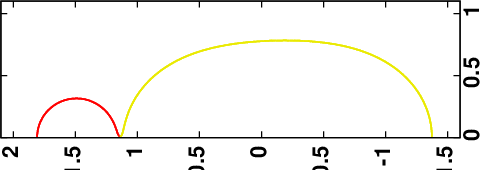} 
\includegraphics[angle=-90,width=0.15\textwidth]{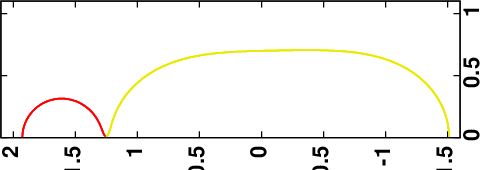} 
}
\mbox{
\includegraphics[angle=-90,width=0.15\textwidth]{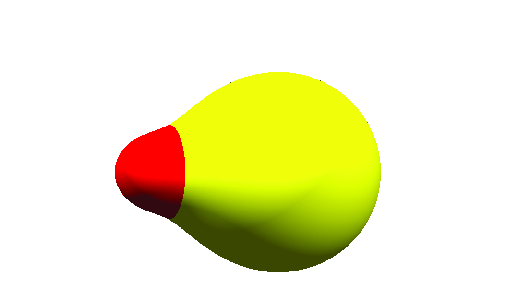} 
\includegraphics[angle=-90,width=0.15\textwidth]{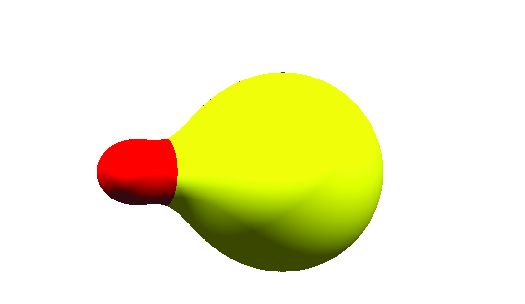} 
\includegraphics[angle=-90,width=0.15\textwidth]{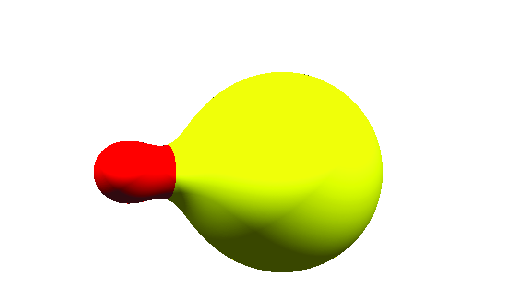} 
\includegraphics[angle=-90,width=0.15\textwidth]{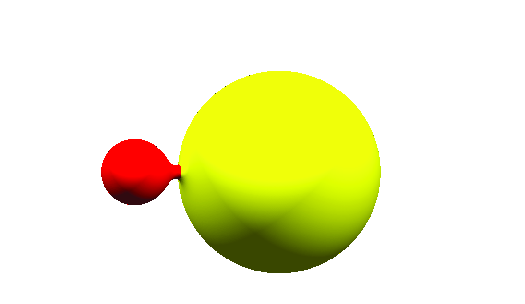} 
\includegraphics[angle=-90,width=0.15\textwidth]{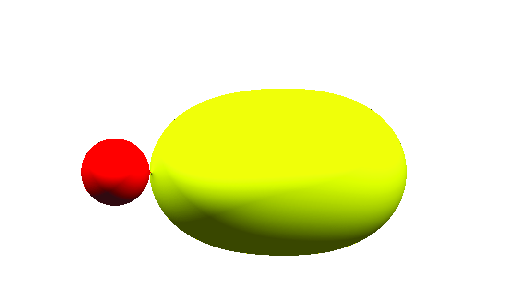} 
\includegraphics[angle=-90,width=0.15\textwidth]{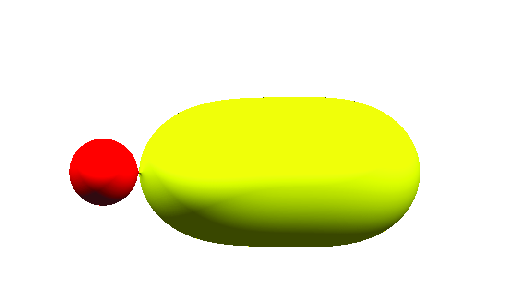} 
}
\caption{($C^1$ with phase area and volume conservation, $\varsigma=9$)
Approximations of the equilibrium shapes 
for $v_r = 0.95$, $0.91$, $0.9$, $0.885$, $0.84$ and $0.8$, 
respectively. 
}
\label{fig:c1ES}
\end{figure}%

Next 
we vary the Gaussian bending rigidity $\alpha^G_1$ for the equilibrium shape in
Figure~\ref{fig:c1ES} with $v_r=0.9$, and report on the new equilibrium shapes
in Figure~\ref{fig:c1ES09alpha}.
It can clearly be observed, that the interface between the two phases
moves away from the neck position, if $|\alpha^G_1|$ increases.
This can be explained with the help of the axisymmetric formulation of
the Gaussian curvature contribution in the energy.
In fact, in the $C^1$--case, when $\vec\mu_2(\tfrac12)=-\vec\mu_1(\tfrac12)$, 
we obtain, compare (\ref{eq:Ec0c1}),
\[
2\,\pi\,(\alpha^G_1-\alpha^G_2)\, \vec\mu_2(\tfrac 12)\,.\,\vec \ek_1
\]
as the Gaussian curvature contribution. This implies that the first
component of $\vec\mu_2(\frac 12)$ prefers to be positive if
$\alpha_1^G-\alpha_2^G<0$, and prefers to be negative if
$\alpha_1^G-\alpha_2^G>0$. We observe this behaviour in
Figure~\ref{fig:c1ES09alpha}, and in particular observe that phase 2 is
in the neck region if $\alpha^G_1$ is negative and phase 1 is in the
neck region if $\alpha^G_1$ is positive, compare also 
\citet[Fig.\ 5]{BaumgartDWJ05}. 
For the numerical results in Figure~\ref{fig:c1ES09alpha} we remark that the
condition \eqref{eq:alphaGbound} is only satisfied if $\alpha^G_1\in[-2,2]$.
Yet also for values outside this interval, our numerical method is able to
integrate the evolution, and the movement of the phase boundary becomes ever
more pronounced.
In addition, we show some equilibrium shapes for $\alpha^G_1\in[-2,2]$ when
the surface has a reduced volume of $v_r=0.885$. In this case, we observe
an induced pinch-off for $\alpha^G_1=2$, see Figure~\ref{fig:c1ES0885alpha}.
\begin{figure}
\center
\includegraphics[angle=-90,width=0.2\textwidth]{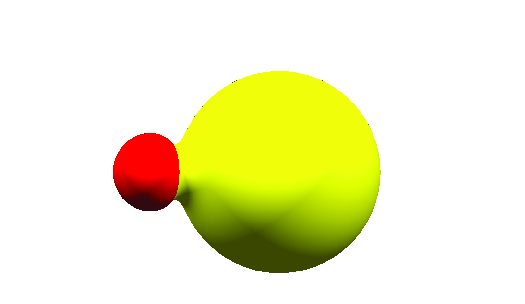} 
\includegraphics[angle=-90,width=0.1\textwidth]{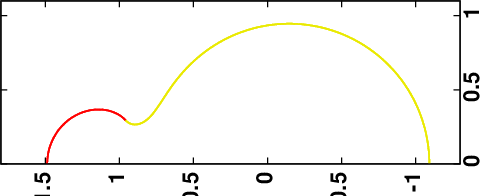} 
\includegraphics[angle=-90,width=0.1\textwidth]{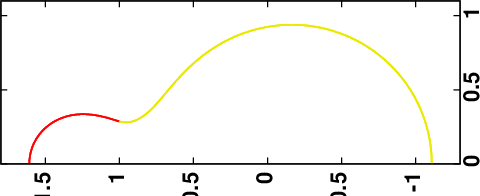} 
\includegraphics[angle=-90,width=0.1\textwidth]{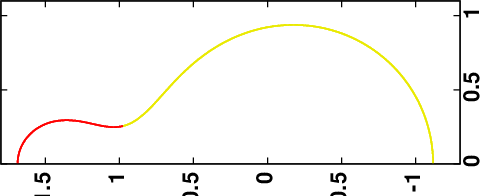} 
\includegraphics[angle=-90,width=0.1\textwidth]{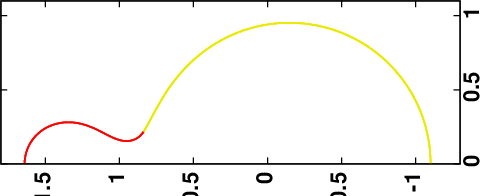} 
\includegraphics[angle=-90,width=0.1\textwidth]{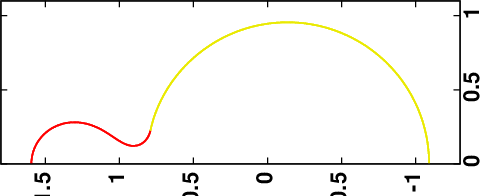} 
\quad
\includegraphics[angle=-90,width=0.2\textwidth]{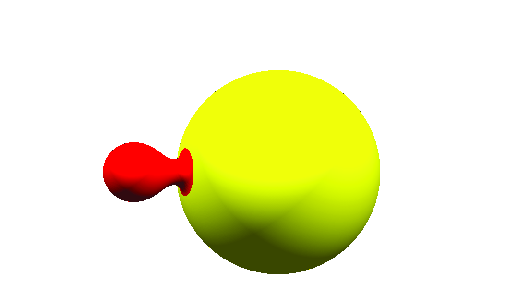} 

\caption{($C^1$ with phase area and volume conservation, $\varsigma=9$)
Approximations of the equilibrium shapes for $v_r=0.9$,
when $\alpha_1^G = -8,-2,0,2,8$. Apart from the cuts, we also show the surfaces
for $\alpha_1^G=-8$ (left) and for $\alpha_1^G=8$ (right).
}
\label{fig:c1ES09alpha}
\end{figure}%
\begin{figure}
\center
\includegraphics[angle=-90,width=0.2\textwidth]{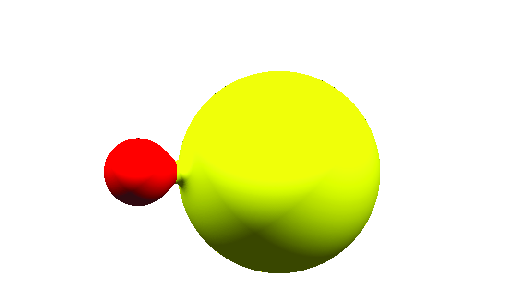} 
\includegraphics[angle=-90,width=0.16\textwidth]{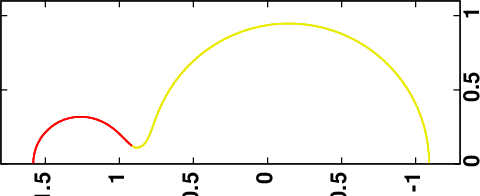} 
\includegraphics[angle=-90,width=0.16\textwidth]{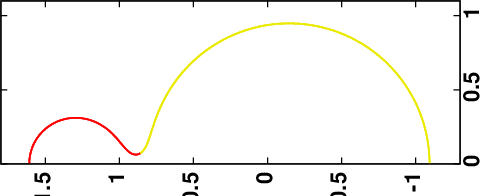} 
\includegraphics[angle=-90,width=0.16\textwidth]{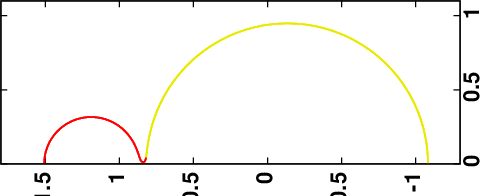} 
\includegraphics[angle=-90,width=0.2\textwidth]{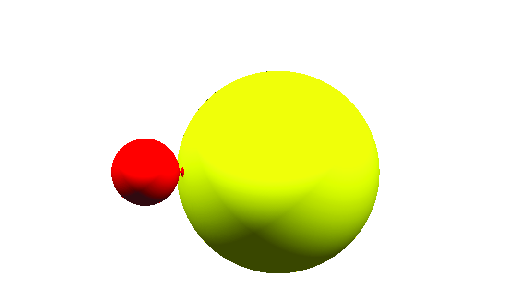} 

\caption{($C^1$ with phase area and volume conservation, $\varsigma=9$)
Approximations of the equilibrium shapes for $v_r=0.885$, when
$\alpha_1^G = -2, 0, 2$. Apart from the cuts, we also show the surfaces
for $\alpha_1^G=-2$ (left) and for $\alpha_1^G=2$ (right).
Note that for $\alpha_1^G=2$ the gradient flow encounters pinch-off.
}
\label{fig:c1ES0885alpha}  
\end{figure}%

In the next set of numerical results, 
we consider the case that one of the phases has two
connected components. These results are inspired by the vesicle shapes found in
experiments. First we consider a surface with reduced volume
$v_r = 0.956$, total surface area $A_1 + A_2 = 4\,\pi$ and
with a phase area ratio of $A_1 / (A_1+A_2) = 0.46$.
Our numerical results in Figure~\ref{fig:wangdu08_fig4_3}
show some resemblance with \citet[Fig.\ 1d]{BaumgartHW03}, see also
\citet[Fig.\ 4]{WangD08}.
\begin{figure}
\center
\mbox{
\includegraphics[angle=-90,width=0.3\textwidth]{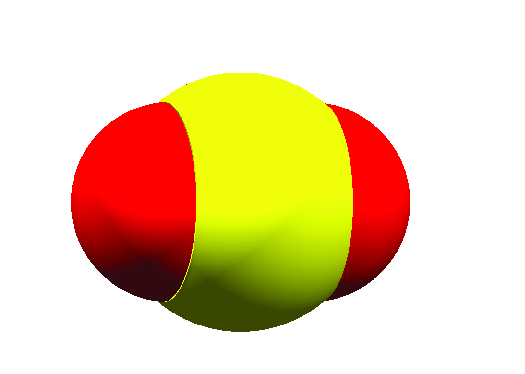} \quad
\includegraphics[angle=-90,width=0.2\textwidth]{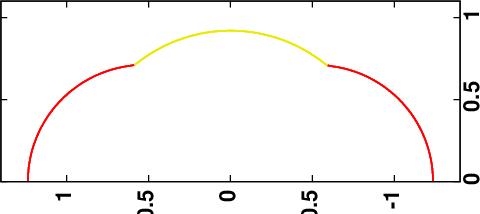} 
\includegraphics[angle=-90,width=0.2\textwidth]{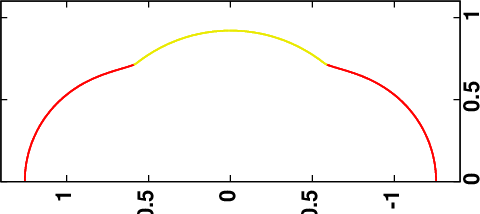} 
\includegraphics[angle=-90,width=0.2\textwidth]{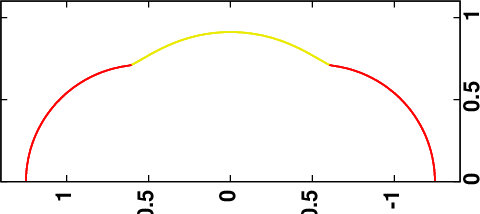} 
}
\caption{($C^1$ with phase area and volume conservation, 
 $\varsigma=50$)
Approximations of the equilibrium shapes for $v_r= 0.956$.
The surface for $(\alpha_1,\alpha_2) = (0.01,0.01)$, as well as the cuts
for $(\alpha_1,\alpha_2) = (0.01,0.01)$ (left), 
$(\alpha_1,\alpha_2) = (1,0.01)$ (middle) and 
$(\alpha_1,\alpha_2) = (0.01,1)$ (right).
}
\label{fig:wangdu08_fig4_3}
\end{figure}%
Next we consider the shape in \citet[Fig.\ 2f]{BaumgartHW03}, 
see also the final simulated surface in \citet[Fig.\ 3]{WangD08}.
We consider a surface with reduced volume
$v_r = 0.8$, total surface area $A_1 + A_2 = 4\,\pi$ and
with a phase area ratio of $A_1 / (A_1+A_2) = 0.09$.
Our numerical results are shown in Figure~\ref{fig:wangdu08_fig3_redvol08}
and the results resemble the situation in the neck region of the
experiments of \citet[Fig.\ 2f]{BaumgartHW03}.
\begin{figure}
\center
\mbox{
\includegraphics[angle=-90,width=0.2\textwidth]{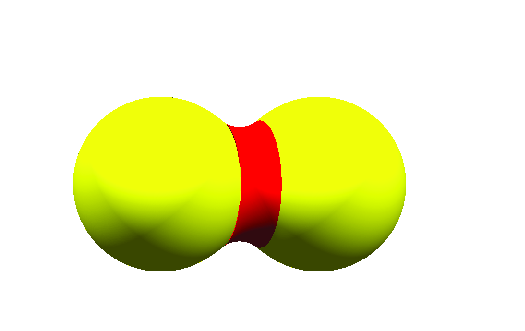} 
\includegraphics[angle=-90,width=0.1\textwidth]{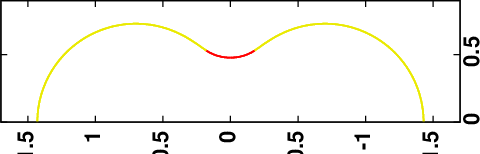} 
\includegraphics[angle=-90,width=0.1\textwidth]{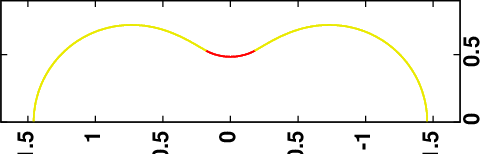} 
\includegraphics[angle=-90,width=0.1\textwidth]{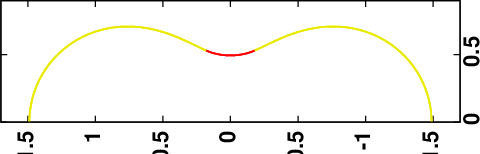} 
\includegraphics[angle=-90,width=0.1\textwidth]{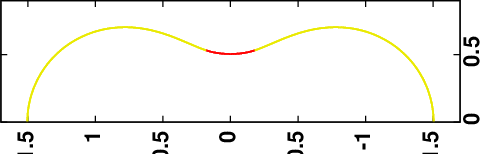} 
\includegraphics[angle=-90,width=0.1\textwidth]{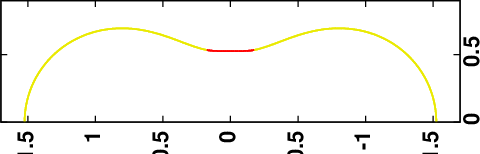} 
\includegraphics[angle=-90,width=0.2\textwidth]{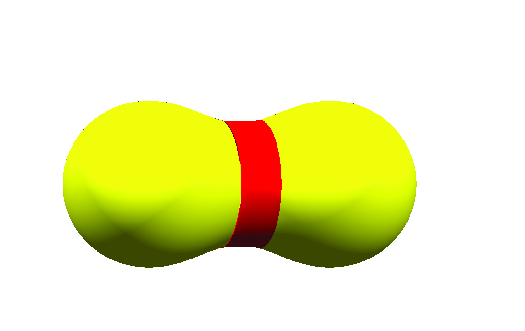} 
}
\mbox{
\includegraphics[angle=-90,width=0.2\textwidth]{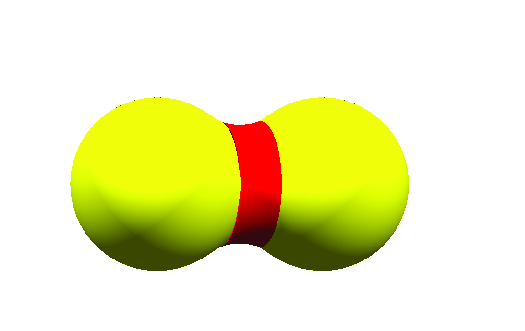} 
\includegraphics[angle=-90,width=0.1\textwidth]{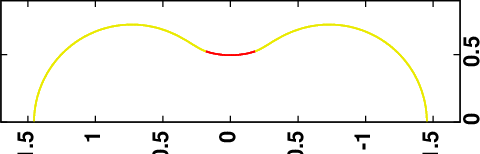} 
\includegraphics[angle=-90,width=0.1\textwidth]{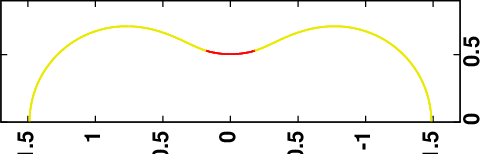} 
\includegraphics[angle=-90,width=0.1\textwidth]{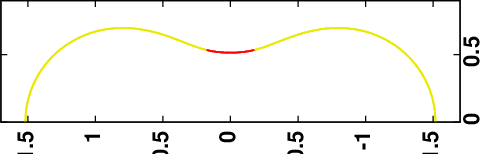} 
\includegraphics[angle=-90,width=0.1\textwidth]{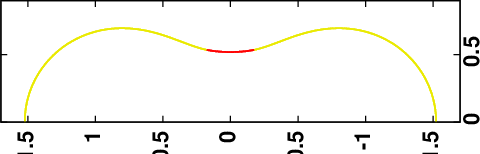} 
\includegraphics[angle=-90,width=0.1\textwidth]{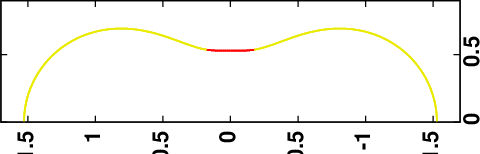} 
\includegraphics[angle=-90,width=0.2\textwidth]{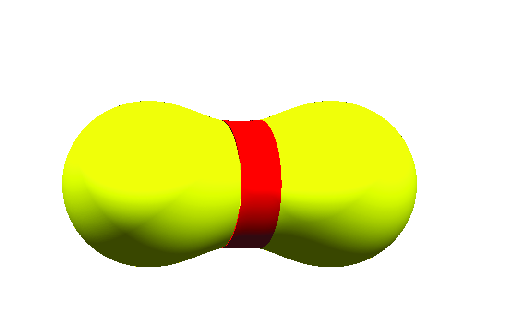} 
}
\caption{($C^1$ with phase area and volume conservation, 
$\varsigma=9$)
Approximations of the equilibrium shapes for $v_r= 0.8$, when
$(\alpha_1,\alpha_2) = (1,0.1)$, $(1,0.5)$, $(1,1)$, $(0.5,1)$, $(0.1,1)$
and $\spont_1=2$ (top), as well as for $\spont_1=0$ (bottom).
On the sides we show the surfaces for $(\alpha_1,\alpha_2) = (1,0.1)$ (left)
and $(\alpha_1,\alpha_2) = (0.1,1)$ (right).
}
\label{fig:wangdu08_fig3_redvol08}
\end{figure}%

\begin{appendix}
\renewcommand{\theequation}{A.\arabic{equation}}
\setcounter{equation}{0}
\section{Consistency of the weak formulations} \label{sec:A}

Starting from our weak formulations, \eqref{eq:weak3}
with \eqref{eq:weak3a} replaced by \eqref{eq:weak4LMa}, in this appendix
we derive the strong form for the $L^2$--gradient flow of (\ref{eq:areaEc0c1}),
together with the boundary conditions that need to hold on $\partial I_i$, 
for $i=1,2$. Here we will make extensive use of \citet[Appendix~A]{axipwf}, 
and for ease of exposition we will often suppress the dependence on time.
We begin by writing \eqref{eq:weak4LMa} as
\begin{equation*} 
2\,\pi\,\sum_{i=1}^2 \left((\vec x_i\,.\,\vec\ek_1)\,[\vec x_i]_t\,.\,\vec\nu_i,
\vec\chi_i\,.\,\vec\nu_i\,|[\vec x_i]_\rho|\right)
= \sum_{i=1}^2 D_i(\vec\chi)
\quad \forall\ \vec\chi \in \xspace\,,
\end{equation*}
where
\begin{align}
D_i(\vec\chi) & = 
 \left([\vec y_i]_\rho\,.\,\vec\nu_i, 
[\vec\chi_i]_\rho \,.\,\vec\nu_i \, |[\vec x_i]_\rho|^{-1}\right) 
+ \left( \vec f_i, \vec\chi_i\,|[\vec x_i]_\rho|\right)
- 2\,\pi\, \lambda_V \left((\vec x_i\,.\,\vec\ek_1)\,\vec\nu_i, 
\vec\chi_i\,|[\vec x_i]_\rho|\right)\nonumber \\ & \quad
- 2\,\pi\,\lambda_{A,i} \left[
\left( \vec\ek_1, \vec\chi_i\,|[\vec x_i]_\rho|\right) 
+ \left( (\vec x_i\,.\,\vec\ek_1)\,\vec\tau_i
, [\vec\chi_i]_\rho\right) \right], \ i=1,2\,.
\label{eq:Di}
\end{align}
On noting that the right hand side of \eqref{eq:Di} corresponds 
to the right hand side of \citet[(A.2)]{axipwf} for a single curve, 
we can apply the results from \citet[Appendix~A]{axipwf} to show that
the strong formulations for the flows in the interior are given by 
\eqref{eq:xtbgnlambda}, while the boundary conditions on
$\partial I_i \setminus \{\tfrac12\}$, for $i=1,2$, are \eqref{eq:part0I}.
Hence it only remains to derive the conditions that need to hold at the
junction, i.e.\ on $\{\tfrac12\}$. 
Collecting the contributions that arise from the boundary terms
$B_1,\ldots,B_5$ in \citet[Appendix~A.1]{axipwf} at the junction point 
for each of the two curves, 
which altogether arise from the first, second and last term on the right hand
side of \eqref{eq:Di}, 
we obtain that the weak formulation enforces
\begin{align} 
& \sum_{i=1}^2 
\left\{(-1)^{i+1}([\vec y_i]_s\,.\,\vec\nu_i)\,\vec\chi_i\,.\,\vec\nu_i
  -\pi\,\varsigma\, \vec\chi_i\,.\,\vec\ek_1 
-\pi\,(-1)^{i+1}\,\vec x_i\,.\,\vec\ek_1\,
(\alpha_i\,[\varkappa_{\mathcal{S}_i}-\spont_i]^2 +2\,\lambda_{A,i})\,
\vec\chi_i\,.\,\vec\tau_i
\right. \nonumber \\ & \qquad \left.
 -2\,\pi\,\alpha_i\,(-1)^{i+1}\,(\varkappa_{\mathcal{S}_i} -   \spont_i)\, 
 (\vec\tau_i\,.\,\vec\ek_1)\, \vec\chi_i\,.\,\vec\nu_i
 +(-1)^{i+1} \left[\varkappa_i\,\vec\chi_i\,.\,\vec y_i^\perp\right]\right\} 
= 0  \label{eq:bc1div2}
\end{align}
at the junction. 
We note from (\ref{eq:kappaid}) and (\ref{eq:meanGaussS}) that
\begin{equation} \label{eq:A1:kappay}
\vec y_i\,.\,\vec\nu_i = 
2\,\pi\,\vec x_i\,.\,\vec\ek_1\,
\alpha_i\,(\varkappa_{\mathcal{S}_i} - \spont_i) \,,
\quad \text{where}\quad
\varkappa_{\mathcal{S}_i} = \varkappa_i - 
\frac{\vec\nu_i\,.\,\vec\ek_1}{\vec x_i\,.\,\vec\ek_1}
\quad\text{in } \overline{I}_i\,,\ i=1,2\,.
\end{equation}
Moreover, we recall from \citet[(3.24)]{axipwf} that it can be shown that
\begin{equation} \label{eq:A1:324}
\varkappa_i\,\vec y_i^\perp + ([\vec y_i]_s\,.\,\vec\nu_i)\, \vec\nu_i =
\varkappa_i\,(\vec y_i\,.\,\vec\nu_i)\,\vec\tau_i + 
(\vec y_i\,.\,\vec\nu_i)_s\,\vec\nu_i
\quad\text{in } \overline I_i\,,\ i=1,2\,.
\end{equation}
It follows from \eqref{eq:A1:324} and \eqref{eq:A1:kappay} 
that we can combine the first and last term on the left hand 
side of (\ref{eq:bc1div2}) to give
\[
 2\,\pi\,\sum_{i=1}^2 (-1)^{i+1} \left\{
\varkappa_i\,\vec x_i\,.\,\vec\ek_1
\,\alpha_i\,(\varkappa_{\mathcal{S}_i} -\spont_i)\,\vec\chi_i\,.\,\vec\tau_i 
  + \alpha_i\, [\vec x_i\,.\,\vec\ek_1\,(\varkappa_{\mathcal{S}_i})_s
    +\vec\tau_i\,.\,\vec\ek_1\,(\varkappa_{\mathcal{S}_i}-\spont_i)]
\,\vec\chi_i\,.\,\vec\nu_i \right\} .
\]
Hence, on using the notations $\vec x = \vec x_1 = \vec x_2$ and
$\vec\chi = \vec\chi_1 = \vec\chi_2$ at the
point $\frac12$, and on recalling (\ref{eq:mu}), it follows from 
(\ref{eq:bc1div2}) 
that
\begin{align*}
& 2\,\pi\,\vec x\,.\,\vec\ek_1\,
\sum_{i=1}^2
\left[
(-1)^{i+1}\,\alpha_i\,(\varkappa_{\mathcal{S}_i})_s\,\vec\nu_i
-(\tfrac12\,\alpha_i\,(\varkappa_{\mathcal{S}_i}-\spont_i)^2+\lambda_{A,i})\,
\vec\mu_i 
+\alpha_i\,(\varkappa_{\mathcal{S}_i}-\spont_i)\,\varkappa_i\,\vec\mu_i 
\right. \nonumber \\ & \qquad \qquad \left.
-\tfrac12\,\frac{\varsigma}{\vec x\,.\,\vec\ek_1}
\,\vec\ek_1\right] . \,\vec\chi = 0\,.
\end{align*}
As $\vec\chi \in \xspace$ is arbitrary, we obtain from the above identity that
\begin{align} 
& \sum_{i=1}^2
\left[
(-1)^{i+1}\,\alpha_i\,(\varkappa_{\mathcal{S}_i})_s\,\vec\nu_i
-(\tfrac12\,\alpha_i\,(\varkappa_{\mathcal{S}_i}-\spont_i)^2+\lambda_{A,i}
-\alpha_i\,(\varkappa_{\mathcal{S}_i}-\spont_i)\,\varkappa_i)\,\vec\mu_i 
\right]
-\frac{\varsigma}{\vec x\,.\,\vec\ek_1}\,\vec\ek_1
\nonumber \\ & \qquad
= \vec 0 \quad\text{on } \overline I_1 \cap \overline I_2\,.
\label{eq:bc1div2b}
\end{align}
We first consider the case of a $C^0$--junction, i.e.\ $C_1 = 0$.
Then it follows from (\ref{eq:weak3d}) and \eqref{eq:A1:kappay} that
\begin{equation*}
\alpha_i\,(\varkappa_{\mathcal{S}_i} - \spont_i) = \alpha_i^G\,
\frac{\vec\nu_i\,.\,\vec\ek_1}{\vec x\,.\,\vec\ek_1}
\quad \text{on } \partial I_i \setminus \{0,1\}\,,\ i=1,2\,,
\end{equation*}
which is (\ref{eq:axiC0bc1}). Using this identity in (\ref{eq:bc1div2b}),
we obtain with the help of (\ref{eq:meanGaussS}) that
\begin{equation*}
\sum_{i=1}^2
\left[
(-1)^{i+1}\,\alpha_i\,(\varkappa_{\mathcal{S}_i})_s\,\vec\nu_i
-(\tfrac12\,\alpha_i\,(\varkappa_{\mathcal{S}_i}-\spont_i)^2+\lambda_{A,i}
+\alpha_i^G\,\Gauss_{\mathcal{S}_i})\,\vec\mu_i \right]
-\frac{\varsigma}{\vec x\,.\,\vec\ek_1}\,\vec\ek_1 
= \vec 0 
\,,
\end{equation*}
which is (\ref{eq:axiC0bc2}). This shows that the weak formulation
implies the boundary conditions at the junction in the $C^0$--case.

In the $C^1$--case, i.e.\ for $C_1=1$, we recall from 
Remark~\ref{rem:C0} that
\begin{equation}\label{eq:y1miny2}
\vec y_1-\vec y_2 = 2\,\pi\,[\alpha_1^G-\alpha_2^G]\,\vec\ek_1
\quad\text{on } \overline I_1 \cap \overline I_2\,,
\end{equation}
and that \eqref{eq:weak3cc1} holds. Applying integration by parts 
to the two second order terms in \eqref{eq:weak3cc1}, and
observing the fact that $\vec\eta_1(\tfrac12) = \vec\eta_2(\tfrac12)$
as $(\vec\eta_1,\vec\eta_2) \in \Vycotest$, yields
\begin{equation*}
  \frac{[\vec x_1]_\rho}{|[\vec x_1]_\rho|} 
- \frac{[\vec x_2]_\rho}{|[\vec x_2]_\rho|} = \vec 0
\quad\text{ on } \overline{I}_1 \cap \overline I_2\,,
\end{equation*}
which, on using (\ref{eq:tau}) and (\ref{eq:mu}), implies that
\begin{equation}\label{eq:numuid}
  \vec\nu:=\vec\nu_2 = \vec\nu_1\quad \text{ and } \quad 
  \vec\mu:=\vec\mu_2 = -\vec\mu_1
\quad\text{ on } \overline{I}_1 \cap \overline I_2\,.
\end{equation}
On combining (\ref{eq:y1miny2}) and (\ref{eq:weak3b}), which states
that $2\,\pi\,\alpha_i\,\vec x_i \,.\, \vec\ek_1\,
(\varkappa_{\mathcal{S}_i}-\spont_i) = \vec y_i\,.\,\vec\nu_i$
in $\overline{I}_i$, we obtain, on recalling the first definition in
\eqref{eq:numuid}, that
\begin{equation*}
  [\alpha_i\,(\varkappa_{\mathcal{S}_i} -\spont_i)]^2_1 - [\alpha_i^G]^2_1\,
  \frac{\vec\nu\,.\,\vec\ek_1}{\vec x\,.\,\vec\ek_1} = 0
\quad\text{ on } \overline{I}_1 \cap \overline I_2\,,
\end{equation*}
which is (\ref{eq:axiC1bc1}). Moreover, substituting (\ref{eq:numuid}) into 
(\ref{eq:bc1div2b}) gives
\[
\sum_{i=1}^2
(-1)^{i+1}
\left[ \alpha_i\,(\varkappa_{\mathcal{S}_i})_s\,\vec\nu
+(\tfrac12\,\alpha_i\,(\varkappa_{\mathcal{S}_i}-\spont_i)^2+\lambda_{A,i}
-\alpha_i\,(\varkappa_{\mathcal{S}_i}-\spont_i)\,\varkappa_i)\,\vec\mu
\right]
-\frac{\varsigma}{\vec x\,.\,\vec\ek_1}\,\vec\ek_1 = \vec 0 
\]
at the junction, and taking the inner products with $\vec\nu$ and $\vec\mu$
leads to
\begin{equation*}
-[\alpha_i\,(\varkappa_{\mathcal{S}_i})_s]^2_1 
- \varsigma\,\frac{\vec\nu\,.\,\vec\ek_1}{\vec x\,.\,\vec\ek_1} = 0
\quad\text{on } \overline I_1 \cap \overline I_2
\end{equation*}
and
\begin{equation*}
[-\tfrac12\,\alpha_i\,(\varkappa_{\mathcal{S}_i}-\spont_i)^2 +
\alpha_i\,(\varkappa_{\mathcal{S}_i}-\spont_i)\, \varkappa_i 
- \lambda_{A,i}]^2_1 -
 \varsigma\, \frac{\vec\mu\,.\,\vec\ek_1}{\vec x\,.\,\vec\ek_1} = 0
\quad\text{on } \overline I_1 \cap \overline I_2\,.
\end{equation*}
The last two equations coincide with (\ref{eq:axiC1bc2}) and
(\ref{eq:axiC1bc3}), respectively. Hence we have shown that in the 
$C^1$--case, the weak formulation implies the correct boundary conditions 
(\ref{eq:axiC1bc}).

\renewcommand{\theequation}{B.\arabic{equation}}
\setcounter{equation}{0}
\section{Some axisymmetric differential geometry} \label{sec:B}
In this appendix, we review some material on the geometry of surfaces
from Chapter~2 in the recent review article \cite{bgnreview}, 
and apply it to axisymmetric surfaces.

Let $\vec x : I \to \bR^2$, $I\subset \bR$, 
be a local parameterisation of the curve $\Gamma$,
with tangent $\vec\tau = |\vec x_\rho|^{-1}\,\vec x_\rho = \vec x_s$,
unit normal $\vec\nu$ 
and curvature vector $\varkappa\,\vec\nu = \vec\tau_s$,
where we have defined $\partial_s = |\vec{x}_\rho|^{-1}\,\partial_\rho$.
Let $\Gamma$ be
the generating curve of an axisymmetric surface $\mathcal{S}$ in $\bR^3$. 
Then $\vec y : I \times [0,2\,\pi) \to \bR^3$ is a local parameterisation of
$\mathcal{S}$, where
\begin{equation*} 
\vec y(\rho,\theta) = (\vec x(\rho)\,.\,\vec\ek_1\,\cos\theta, 
\vec x(\rho)\,.\,\vec\ek_2, \vec x(\rho)\,.\,\vec\ek_1\,\sin\theta)^T 
\,.
\end{equation*}
The tangent space of $\mathcal{S}$ at $\vec y(\rho,\theta)$ is spanned by
the two tangent vectors
\begin{align} \label{eq:yrho}
\vec y_\rho (\rho,\theta) & = |\vec x_\rho|\,
(\vec\tau\,.\,\vec\ek_1\,\cos\theta, 
\vec\tau\,.\,\vec\ek_2, \vec\tau\,.\,\vec\ek_1\,\sin\theta)^T
\,, \nonumber \\
\vec y_\theta (\rho,\theta) & = 
(-\vec x\,.\,\vec\ek_1\,\sin\theta, 0, \vec x\,.\,\vec\ek_1\,\cos\theta)^T \,,
\end{align}
and a unit normal vector can be defined via
\begin{equation} \label{eq:nuS}
\vec\nu_{\mathcal{S}}(\rho,\theta) = 
(\vec\nu\,.\,\vec\ek_1\,\cos\theta, \vec\nu\,.\,\vec\ek_2,
\vec\nu\,.\,\vec\ek_1\,\sin\theta )^T\,.
\end{equation}
It follows from \eqref{eq:yrho} and \citet[Remark~8]{bgnreview}
that the coefficients of the first fundamental form of $\mathcal{S}$
are given by
\begin{equation} \label{eq:vecys}
g_{\rho\rho} = |\vec y_\rho|^2 = |\vec x_\rho|^2\,,\quad 
g_{\theta\theta} = |\vec y_\theta|^2 = (\vec x\,.\,\vec\ek_1)^2\,,\quad
g_{\rho\theta} = g_{\theta\rho} = \vec y_\rho\,.\,\vec y_\theta = 0\,,
\end{equation}
with the square of the local area element on $\mathcal{S}$ given by
\begin{equation} \label{eq:localarea}
g = g_{\theta\theta}\,g_{\rho\rho}
= (\vec x\,.\,\vec\ek_1)^2\,|\vec x_\rho|^2
\quad\text{in } I \times [0,2\,\pi)\,.
\end{equation}
Moreover, it follows from (\ref{eq:vecys}) and \citet[Remark~8]{bgnreview}
that the surface gradient and the
surface divergence of smooth functions $f:\mathcal{S} \to \bR$,
$\vec f:\mathcal{S} \to \bR^3$ on $\mathcal{S}$ can be calculated as
\begin{alignat*}{2} 
(\nabS\,f)\circ \vec y & 
= |\vec y_\rho|^{-2}\,(f\circ \vec y)_\rho\,\vec y_\rho 
+ |\vec y_\theta|^{-2}\,(f\circ \vec y)_\theta\,\vec y_\theta && \\ &
= (f\circ \vec y)_s\,\vec y_s + (\vec x\,.\,\vec\ek_1)^{-2}\,
(f\circ  \vec y)_\theta\,\vec y_\theta
&& \quad\text{in } I \times [0,2\,\pi)\,, \\
(\nabS\,.\,\vec f) \circ \vec y & 
= (\vec f \circ \vec y)_s\,.\,\vec y_s + (\vec x\,.\,\vec\ek_1)^{-2}\,
(\vec f \circ \vec y)_\theta\,.\,\vec y_\theta 
&& \quad\text{in } I \times [0,2\,\pi)\,.
\end{alignat*}
Hence, on noting $((\vec x\,.\,\vec\ek_1)^{-1}\,\vec y_\theta)_s=\vec 0$ 
and $(\vec y_s)_\theta \,.\, \vec y_\theta = (\vec x\,.\,\vec\ek_1)\, 
\vec x_s\,.\,\vec\ek_1$, we obtain that
\begin{align*} 
(\Delta_{\mathcal S}\,f) \circ \vec y 
& = (\nabS\,.\,(\nabS\, f)) \circ \vec y \\ &
= (f\circ \vec y)_{ss} + 
\frac{\vec x_s\,.\,\vec\ek_1}{\vec x\,.\,\vec\ek_1}\,(f \circ \vec y)_s + 
(\vec x\,.\,\vec\ek_1)^{-2}\,(f\circ\vec y)_{\theta\theta}
&& \quad\text{in } I \times [0,2\,\pi)\,.
\end{align*}
For a radially symmetric function $f:\mathcal{S}\to\bR$, 
with $f(\vec y(\rho,\theta)) = f(\vec y(\rho,0))$ 
for all $(\rho,\theta) \in I \times [0,2\,\pi)$, it follows that
\begin{equation} \label{eq:LBSrad}
(\Delta_{\mathcal S}\,f) \circ \vec y 
= (\vec x\,.\,\vec\ek_1)^{-1}\,(\vec x\,.\,\vec\ek_1\,(f \circ \vec y)_s)_s
\quad\text{in } I \times [0,2\,\pi) \,.
\end{equation}

On recalling Definitions~10 and 11 in \cite{bgnreview}, we now compute
the principal curvatures of $\mathcal{S}$ as the eigenvalues of the Weingarten
map $W_{\vec p} (\vec{\mathfrak t}) = - \partial_{\vec{\mathfrak t}}\,
\vec\nu_{\mathcal{S}}$ at $\vec p = \vec y(\rho,\theta)$. Choosing for the
tangent vector $\vec{\mathfrak t}$ the first vector in \eqref{eq:yrho}, 
recalling \eqref{eq:nuS}, and noting that $\partial_{\vec y_\rho} 
= \partial_\rho$, we obtain
\begin{subequations} \label{eq:ev}
\begin{align}
W_{\vec p} (\vec y_\rho) = - \partial_{\vec y_\rho}\,\vec\nu_{\mathcal{S}} &
= - [\vec\nu_{\mathcal{S}}]_\rho = - |\vec x_\rho|\,[\vec\nu_{\mathcal{S}}]_s 
= - |\vec x_\rho| (
\vec\nu_s\,.\,\vec\ek_1\,\cos\theta,
\vec\nu_s\,.\,\vec\ek_2,
\vec\nu_s\,.\,\vec\ek_1\,\sin\theta)^T\nonumber \\ &
= |\vec x_\rho| \varkappa\,(
\vec\tau\,.\,\vec\ek_1\,\cos\theta,
\vec\tau\,.\,\vec\ek_2,
\vec\tau\,.\,\vec\ek_1\,\sin\theta)^T 
= \varkappa\,\vec y_\rho\,,
\label{eq:ev1}
\end{align}
where we have used that $\vec\nu_s = - \varkappa\,\vec\tau$, since
$\vec\nu_s\,.\,\vec\tau = - \varkappa$ and $\vec\nu_s\,.\,\vec\nu=0$.
Similarly, choosing for the tangent vector $\vec{\mathfrak t}$ the second 
vector in \eqref{eq:yrho}, and noting that $\partial_{\vec y_\theta} =
\partial_\theta$, yields
\begin{align}
W_{\vec p} (\vec y_\theta) &
= - [\vec\nu_{\mathcal{S}}]_\theta  
= - (-\vec\nu\,.\,\vec\ek_1\,\sin\theta, 0,
\vec\nu\,.\,\vec\ek_1\,\cos\theta )^T
= - \frac{\vec\nu\,.\,\vec\ek_1}{\vec x\,.\,\vec\ek_1}\,\vec y_\theta\,.
\label{eq:ev2}
\end{align}
\end{subequations}
Clearly, \eqref{eq:ev} implies that
the two eigenvalues of the Weingarten map are $\varkappa$ and
$- \frac{\vec\nu\,.\,\vec\ek_1}{\vec x\,.\,\vec\ek_1}$, which means that for
the mean and Gaussian curvatures of $\mathcal{S}$ we obtain the formulas
\begin{equation} \label{eq:appmeanGaussS}
\varkappa_{\mathcal{S}} = \varkappa - 
\frac{\vec\nu\,.\,\vec\ek_1}{\vec x\,.\,\vec\ek_1}
\quad\text{and}\quad
\Gauss_{\mathcal{S}} = -
\varkappa\,\frac{\vec\nu\,.\,\vec\ek_1}{\vec x\,.\,\vec\ek_1}
\quad\text{in }\ I \times [0,2\,\pi)\,.
\end{equation}
\end{appendix}

\section*{Acknowledgements}
The authors gratefully acknowledge the support 
of the Regensburger Universit\"atsstiftung Hans Vielberth.

We would like to dedicate this article to our colleague and dear friend 
John W.\ Barrett, who died much too early on 30 June 2019. This manuscript
marks the conclusion of a long and fruitful collaboration between the three of
us. The idea to apply our knowledge on equidistributing curve approximations 
from the series of papers \cite{triplej,triplejMC,fdfi} to the approximation 
of axisymmetric surfaces was one of John's, in the autumn of 2017. Since then
we have published papers with John on axisymmetric curvature flows,
\cite{aximcf}, axisymmetric surface diffusion and related fourth order flows,
\cite{axisd}, axisymmetric Willmore flow, \cite{axipwf}, as well as papers on
the closely related topic of curve evolutions in Riemannian manifolds,
\cite{hypbol,hypbolpwf}. But at the back of John's and our mind was always
to eventually apply these new ideas to the evolution of two-phase biomembranes,
in order to obtain a very efficient numerical method with which to compute 
possible minimisers of the energy introduced by \cite{JulicherL93,JulicherL96},
which can be used to explain the experimental findings of 
Baumgart, Hess and Webb in their seminal Nature paper \cite{BaumgartHW03}. 
Sadly, John could not join us on this final stage 
of the journey and see his original idea come to fruition.

We miss John every day. We miss our joint laughter, our excitement at
scientific breakthroughs and our discussions on football and politics. But
above all we miss John as a person 
and as a role model: we will miss 
his great sense of humour, his razor sharp intellect, his honesty,
his integrity, his passion and his loyalty.

\end{document}